\documentclass[10pt]{amsart}

\usepackage[pagebackref,colorlinks=true]{hyperref}  

\usepackage{amsfonts}
\usepackage{amsmath,amsthm,amssymb,amscd,enumerate,eucal,url}
\usepackage{eucal,url,amssymb,verbatim,
enumerate,amscd,
}
\usepackage{times}
\usepackage{latexsym}
\usepackage{lipsum}

\usepackage{amsfonts}
\usepackage{amsmath,amsthm,amssymb,amscd,enumerate,eucal,url,stmaryrd}

\usepackage{latexsym}
\numberwithin{equation}{section}

\usepackage[margin=1in]{geometry}
\numberwithin{equation}{section}
\newtheorem{thrm}{Theorem}[section]
\newtheorem{lemma}[thrm]{Lemma}
\newtheorem{prop}[thrm]{Proposition}

\newtheorem{dfn}[thrm]{Definition}

\newtheorem{rmrk}[thrm]{Remark}

\newtheorem{conv}[thrm]{Convention}

\newcommand{\R}{\mathbb{R}}
\newcommand{\Rn}{\mathbb{R}^n}

\newcommand{\Cn}{\mathbb{C}^n}
\newcommand{\Om}{\Omega}
\newcommand{\Hn}{\mathbb{H}^n}
\newcommand{\Kn}{\mathbb{K}^n}
\newcommand{\Hnn}{\mathbb{H}^{n+1}}
\newcommand{\QH}{\boldsymbol {G\,(\mathbb{H})}}
\newcommand{\QC}{\boldsymbol {G\,(\mathbb{C})}}
\newcommand{\QK}{\boldsymbol {G\,(\mathbb{K})}}
\newcommand{\algg}{\mathfrak g}
\newcommand{\algn}{\mathfrak n}
\newcommand{\algs}{\mathfrak s}
\newcommand{\alga}{\mathfrak a}

\newcommand{\lap}{\triangle}

\newcommand{\della}{\delta_{a}}

\newcommand{\norm}[1]{\lVert#1\rVert}
\newcommand{\abs}[1]{\lvert #1 \rvert}

\newcommand{\e}{\textbf {e}}

\newcommand{\domoOm}{{\overset{o}{\mathcal{D}}}\,^{1,p}(\Om)}
\newcommand{\DoOm}{{\overset{o}{\mathcal{D}}}\,^{1,2}(\Om)}

\newcommand{\DoG}{{\overset{o}{\mathcal{D}}}\,^{1,2}(\bG)}

\newcommand{\dxa}[1]{\frac {\partial {#1}} {\partial x_{\alpha}} }
\newcommand{\dta}[1]{\frac{\partial {#1}}{\partial t_{\alpha}}}
\newcommand{\dya}[1]{\frac{\partial {#1}}{\partial y_{\alpha}}}
\newcommand{\dza}[1]{\frac{\partial {#1}}{\partial z_{\alpha}}}

\newcommand{\dx}[1]{\frac{\partial {#1}}{\partial x}}
\newcommand{\dy}[1]{\frac{\partial {#1}}{\partial y}}
\newcommand{\dz}[1]{\frac{\partial {#1}}{\partial z}}

\newcommand{\Cr}{\nabla}

\newcommand{\bG}{{\boldsymbol{G}}}

\newcommand{\essinf}[1]{\underset{#1}{\text{essinf}}}
\newcommand{\esssup}[1]{\underset{#1}{\text{esssup}}}

\def\bi{\nabla}
\newcommand{\vol}{\, Vol_{\eta}}
\newcommand{\rvolg}{dvol_g}

\def\gr{\nabla f}

\newcommand{\tileta}{{\tilde\eta}}

\newcommand{\tvol}{\, {Vol}_{\tilde\eta}}
\newcommand{\tlap}{\tilde\triangle}

\newcommand{\Lap}{\mathcal{L}}

\begin{document}

\begin{abstract}
We report on some aspects and recent progress in certain problems in the sub-Riemannian CR and quaternionic contact (QC) geometries. The focus are the corresponding Yamabe problems on the round spheres, the Lichnerowicz-Obata first eigenvalue estimates, and the relation between these two problems. A motivation from the Riemannian case highlights new and old ideas which are then developed in the settings of Iwasawa sub-Riemannian geometries.
\end{abstract}

\keywords{sub-Riemannian geometry, CR and quaternionic contact structures, Sobolev inequality, Yamabe equation, Lichnerowicz eigenvalue estimate, Obata theorem}
\subjclass{58G30, 53C17}
\title[Two problems on CR and quaternionic contact manifolds]{The Lichnerowicz and Obata  first eigenvalue  theorems and the Obata  uniqueness result in the Yamabe problem on CR and quaternionic contact manifolds}
\date{\today }

\author{Stefan Ivanov}
\address[Stefan Ivanov]{University of Sofia, Faculty of Mathematics and Informatics,
blvd. James Bourchier 5, 1164, Sofia, Bulgaria \& Institute of
Mathematics and Informatics, Bulgarian Academy of Sciences}
\address{and Department of Mathematics,
University of Pennsylvania, DRL 209 South 33rd Street
Philadelphia, PA 19104-6395} \email{ivanovsp@fmi.uni-sofia.bg}

\author{Dimiter Vassilev}
\address[Dimiter Vassilev]{ Department of Mathematics and Statistics\\
University of New Mexico\\
Albuquerque, New Mexico, 87131-0001\\
}
\email{vassilev@unm.edu}
\maketitle
\tableofcontents


\setcounter{tocdepth}{2}
\section{Introduction}
As the title suggests, the goal of this paper is to report on some aspects of  certain problems in the sub-Riemannian CR and quaternionic contact (QC) geometries.  It seems appropriate in lieu of an extensive Introduction to begin with a section about the corresponding problems in the Riemannian case.
 Besides an introduction to the discussed problems we give key steps of the  proofs of some  well known results highlighting ideas which can be used, although with a considerable amount of extra analysis in the sub-Riemannian setting.  In the later sections we show the difficulties and current state of the art in the corresponding results on CR and quaternionic contact manifolds.  However, this article is not designed to be a complete survey of the subjects, especially in the case of the Yamabe problem,  but rather a collection  of particular results with which we have been involved directly while giving references to important works in the area, some of which are covered in this volume.

\begin{conv}\label{convention}
A convention due to traditions: when considering eigenvalue problems, it is more convenient to use the non-negative   (sub-)Laplacian. Correspondingly,  $\lap u=-tr^g(\nabla^2 u)$ for a function $u$ and metric $g$.  On the other hand the (sub-)Laplacian appearing in the Yamabe problem is the "usual" negative (sub-)Laplacian $\lap u=tr^g(\nabla^2 u)$.
\end{conv}
 \section{Background - The Riemannian problems}
The only new result here is Proposition  \ref{p:Obata Einstein}. This fact is exploited later in a new simplified proof of the Obata type  theorem in the qc-setting. Interestingly, the CR case presents another type of behavior.

 \subsection{The Lichnerowicz and Obata first eigenvalue theorems}\label{ss:Riem Lich&Obata}
The relation between the spectrum of the Laplacian and
geometric quantities has been a topic of continued interest.
One such relation was given by  Lichnerowcz \cite{Li} who showed
that on a compact Riemannian manifold $(M,g)$ of dimension $n$ for which the
Ricci curvature  satisfies $
Ric(X,X)\geq (n-1)g(X,X) $  the first positive eigenvalue $\lambda_1$ of the (positive) Laplace operator $\lap f=-tr^g\nabla^2f$
satisfies the inequality $%
\lambda_1\geq n.$  Here $\nabla$ is the Levi-Civita connection of $g$.
In particular, $n$ is the smallest eigenvalue of
the Laplacian on compact Einstein spaces of scalar curvature equal to $n(n-1)$-the scalar curvature of the round unit sphere. Subsequently, Obata \cite{O3} proved that equality is
achieved iff the Riemannian manifold is isometric to the round unit sphere. It should be noted that the smallest possible value  $n$ is achieved on the round unit sphere by the restrictions of the linear functions to the unit sphere (spherical harmonics of degree one), which give the associated eigenspace.
Later, Gallot \cite{Gallot79} generalized these results to statements
involving the higher eigenvalues and corresponding eigenfunctions of the
Laplace operator.

The above described results of Lichnerowicz and Obata we want to discuss in detail are summarized in the next theorem.
\begin{thrm}\label{t:Riem LichObata}
Suppose $(M,g)$ is a compact Riemannian manifold of dimension $n$ which satisfies a positive lower Ricci bound
\begin{equation}\label{e:Ricci lower bound}
Ric(X,X)\geq (n-1)g(X,X).
\end{equation}
\begin{enumerate}[a)]
\item  If $\lambda$ is a non-zero eigenvalue of the (positive) Laplacian, $\triangle f=-\lambda  f$, then $\lambda \geq n$, see \cite{Li}.
\item if there is $\lambda =n$, then $(M,g)$ is isometric with the round sphere $S^n(1)$, see \cite{Ob}
\end{enumerate}
\end{thrm}

Let us briefly sketch the proof of Theorem \ref{t:Riem LichObata} including a new observation, Proposition \ref{p:Obata Einstein}, which will be exploited in the sub-Riemannian setting.
The key to Lichnerowicz' inequality is
Bochner's identity ($\triangle\geq 0)$,
\begin{equation}\label{e:RBochner}
-\frac12\triangle |\nabla f|^2=|\nabla df|^2-g(\nabla(\triangle f),\nabla f)+Ric(\nabla
f,\nabla f).
\end{equation}
After an  integration over the compact manifold we find
\[
0=\int_M |(\nabla df)_0|^2+\frac {1}{n}(\triangle f)^2-g(\nabla(\triangle f),\nabla f)+Ric(\nabla
f,\nabla f)\, {\rvolg}.
\]
Let us assume at this point the inequality $Ric(\nabla
f,\nabla f)\geq (n-1)|\gr|^2$ for any eigenfunction $f$, $\triangle f=\lambda f$. We obtain then the inequality
\begin{multline*}{0=\int_M |(\nabla df)_0|^2+\frac {1}{n}\lambda|\gr|^2-\lambda|\gr|^2+Ric(\nabla
f,\nabla f)\, {\rvolg}}\\
{
= \int_M |(\nabla df)_0|^2\, {\rvolg} +\int_M Ric(\nabla
f,\nabla f) -\frac {n-1}{n}\lambda|\gr|^2\, {\rvolg}}\\
{
\geq \int_M |(\nabla df)_0|^2 \, {\rvolg} +\frac {n-1}{n}\int_M (n-\lambda)|\gr|^2\, {\rvolg}.}
\end{multline*}
{
 Hence $(\nabla df)_0=0$ and $0\geq n-\lambda$}, which proves Lichnerowicz' estimate. Furthermore, if the lowest possible eigenvalue is achieved then the trace-free part of the Riemannian Hessian of an
eigenfunction $f$ with eigenvalue $\lambda=n$ vanishes, i.e., it satisfies
the system
\begin{equation}  \label{e:Riem Hess eqn}
\nabla^2 f = -fg.
\end{equation}

 Obata's result which describes the case of equality was preceded by several results where the case of equality was characterized under the additional assumption that $g$ is Einstein \cite{YaNa59} or has constant scalar curvature \cite{IshTa59}. It turns out that besides Obata's proof these assumptions can also be removed as we found in Proposition \ref{p:Obata Einstein}. Nevertheless, even under the assumption that $g$ is Einstein the proof that $(M,g)=S^n$ requires further delicate analysis involving geodesics and the distance function from a point. Furthermore, Obata showed in fact a more general result, namely,  on a \emph{complete} Riemannian manifold $(M,g)$ equation \eqref{e:Riem Hess eqn} above allows a
non-constant solution iff the manifold is isometric to the round unit
sphere $S^n$.
\begin{rmrk}\label{r:sphere charcat}
A good reference for Hessian equations characterizing the spaces of constant curvature  is \cite{Kuh88}. For example, if $(M,g)$ is compact Riemannian manifold admitting a non-constant solution to $\nabla^2 f = \frac {\lap f}{n} g$ then $(M,g)$  is conformally diffeomorphic to the unit round sphere. Furthermore, if the scalar curvature of $(M,g)$ is constant then $(M,g)$  is isometric to a Euclidean sphere of certain radius.
\end{rmrk}
Thanks to the Bonnet-Myers and S.-Y. Cheng's improved Toponogov  theorems we can sketch the proof of this fact as described in details in \cite[Chapter III.4]{Chav84}. First, we note that assuming $(M,g)$ is complete and satisfies \eqref{e:Ricci lower bound} we have
\begin{enumerate}[(i)]
  \item (Bonnet-Myers) $M$ is compact, the diameter $d(M)\leq \pi$ and $\pi_1(M)$ is finite;
 \item (improved Toponogov theorem)  $d(M)= \pi$ iff $M$ is isometric to $S^n(1)$, \cite{Cheng75}.
 \end{enumerate}
The Hessian equation \eqref{e:Riem Hess eqn} implies that if $\gamma(t)$ is a unit speed geodesics we have $(f\circ\gamma)''+f\circ\gamma=0$, hence $f(\gamma(t))=A\cos t +B\sin t$
for some constants $A$ and $B$. Let $p\in M$ be such that $f(p)=\max_M f$ which exists since $M$ is compact. For any unit tangent vector $\xi\in T_p(M)$ the unit speed geodesic $\gamma_\xi(t)$ from $p$ in the direction of $\xi$ satisfies $f(\gamma_\xi(t))=f(p)\cos t$ sinc the derivative at $t=0$ is zero. Therefore, $f(\gamma_\xi(t))$ is injective for $0\leq t\leq \pi$ which implies $d(M)\geq \pi$. This shows that $d(M)=\pi$ and by Cheng's theorem we conclude $M=S^n$.

\begin{rmrk}
We remark explicitly that the above approach to Obata's theorem  cannot be used in  sub-Riemannian setting {s where}
both (i) and (ii) are very challenging open problems  with the exception of  some results generalizing (i) in some special cases, see Section \ref{ss:sub-Riemannian compariosn note}.
\end{rmrk}


We turn to our result mentioned in context above.
\begin{prop}\label{p:Obata Einstein}
 Suppose $(M,g)$ is a compact Riemannian manifold of dimension $n$ which satisfies \eqref{e:Ricci lower bound}. If the lowest possible eigenvalue is achieved,  $\triangle f =nf$ for some function $f$, then $(M,g)$ is an Einstein space.
\end{prop}

\begin{proof}
The proof follows from several  calculations and a use of the divergence formula. By the proof of Lichnerowicz' estimate the eigenfunction $f$ satisfies \eqref{e:Riem Hess eqn}. Differentiating \eqref{e:Riem Hess eqn} and using Ricci's identity $\nabla^3f(X,Y,Z)-\nabla^3f(Y,X,Z)=-R(X,Y,Z,\gr)$ we find the next formula for the curvature tensor
\begin{equation}\label{e:Obata Einstein 1}
R(X,Y,Z,\gr)=df(X)g(Y,Z)-df(Y)g(X,Z).
\end{equation}
Taking a trace in the above formula we see
\begin{equation}\label{e:Obata Einstein 2}
Ric(X,\gr)=(n-1)df(X).
\end{equation}
A differentiation of \eqref{e:Obata Einstein 2} and another use of \eqref{e:Riem Hess eqn} gives
\begin{equation}\label{e:Obata Einstein 4}
(\nabla_Z Ric)(Y,\gr)=fRic(X,Y)-(n-10fg(X,Y).
\end{equation}
On the other hand, taking the covariant derivative of \eqref{e:Obata Einstein 1} and then using \eqref{e:Riem Hess eqn} for $\nabla_V(\gr)$, we obtain
\[
(\nabla_V R)(Z,X,Y,\gr)=fR(Z,X,Y,V)-fg(V,Z)g(X,Y)+fg(V,X)g(Z,Y).
\]
Therefore, taking a trace, it follows.
\begin{equation}\label{e:Obata Einstein 5}
(\nabla^*R)=Y,X,\gr)=fRic(X,Y)-(n-1)fg(X,Y).
\end{equation}
A substitution of \eqref{e:Obata Einstein 5} in the formula $(\nabla_Z Ric_0)=(\nabla_Z Ric)(X,Y)-\frac {1}{n}dS(Z)g(X,Y)$ with $Z=\gr$ gives
the key identity
\begin{multline}\label{e:Obata Einstein 6}
(\nabla_{\gr} Ric_0) (X,Y)=2f Ric_0(X,Y)  -\frac {2S}{n}g(X,Y)-2(n-1)fg(X,Y)-\frac {1}{n}dS(\gr)g(X,Y).
\end{multline}
Hence, $L_{\gr}|Ric_0|^{2k}=4kf |Ric_0|^{2k}$.
Integrating over compact manifold $M$ with respect to the Riemannian volume we obtain
\begin{multline*}
\int_M |Ric_0|^{2k} f^2 \,{\rvolg}=\frac {1}{n}\int_M g(\nabla |Ric_0|^{2k} f, \nabla f) \,{\rvolg}\\
=\frac {1}{n} \int_M |Ric_0|^{2k} |\gr|^2\,dvol + \frac {4k}{n}\int_M|Ric_0|^{2k} f^2 \,{\rvolg}.
\end{multline*}
Therefore,
$$(n-4k)\int_M|Ric_0|^{2k} f^2 \,{\rvolg}=\int_M |Ric_0|^{2k} |\gr|^2\,{\rvolg},$$ hence choosing $k>n/4$ it follows $Ric_0=0$, i.e., $g$ is an Einstein metric.
\end{proof}

It should be noted that
\eqref{e:Riem Hess eqn} is obtained in
connection with infinitesimal conformal transformations on Einstein
spaces, see \ref{e:inf conf v.f.}. Thus  the unit sphere  is characterized as the only complete Einstein Riemannian manifold of scalar curvature $n(n-1)$
admitting a non-homothetic conformal transformation, \cite{IshTa59}, \cite{O3}, \cite{Ta65}, \cite{YaNa59}, see also later in this Section for relations with the Yamabe problem on the Euclidean sphere. In addition, this result can be considered as a characterization of the unit sphere as a compact Einstein
space which admits an eigenfunction belonging to the
smallest {possible} eigenvalue of the Laplacian {on a} compact Einstein {space}.

We remark that it is  natural to consider characterizations  of the unit  sphere by its  second
eigenvalue $2(n + l)k$ of the Laplacian. In this case,  the gradients of the corresponding eigenfunctions are infinitesimal projective transformations, which also gives a
system of differential equations of order three satisfied by the divergence
of an infinitesimal projective transformation on an Einstein space.
Furthermore, it is shown that the complete Riemannian manifold
admitting a non-trivial solution for the system is isometric
to the unit sphere provided that the manifold is simply connected{,} \cite{EGKU03}, \cite{GKDU}, \cite{Bl75}, \cite{O65}, \cite{O3}.
There are results in the  K\"ahler case where an infinitesimal holomorphically
projective transformation plays a role similar to that of projective
one on Riemannian manifolds, \cite{O65}. We shall seek a characterization of the model CR and qc  unit spheres through the first eigenfunctions of the respective sub-Laplacians

 \subsection{Conformal transformations}
 Let $(M,g)$ and $( M', g')$ be two Riemannian manifold of dimension $n$. A smooth map $F:M\rightarrow M'$ is called a conformal map if $F^* g'=\phi^{-2}\, g$ for some smooth positive function $\phi$. For our goals we shall consider $(M,g)=(M',g')$, $F$ a diffeomorphism, and let $\bar g=F^* g'$. In this case, we say that $F$ is a conformal difeomorphism while the metrics $g$ and $\bar g$ are called (point-wise) conformal to each other. For $n\geq 3$, we shall need the following well known formulas relating the  traceless Ricci and scalar curvatures of the metrics $g$ and $\bar g$,
 \begin{align}\label{e:Ric_o conf change}
 \overline {Ric}_0=Ric_0+(n-2)\phi^{-1}(\nabla^2\phi)_0\\ \label{e:s conf change}
 \bar S=\phi^2S+2(n-1)\phi\lap \phi -n(n-1)|\nabla \phi|^2,
 \end{align}
 where $\nabla$ is the Levi-Civita connection of $g$, $\lap \phi=tr^g\nabla^2\phi$, and $(\nabla^2\phi)_0$ is the traceless part of the Hessian of $\phi$.

 A conformal vector field on a compact Riemannian manifold
$(M, g)$ is a vector field $X$ whose flow consists of conformal transformations (diffeomorphisms). In the case the flow is a one-parameter group of  isometries the vector field
$X$ is called a Killing field. For $M$ compact, the algebra $\mathfrak{c}(M, g)$ of conformal vector fields is exactly the group of conformal diffeomorphisms $C(M, g)$ of $(M, g)$. It is worth recalling \cite{Ya57} and \cite{Li}  that $X$ is a conformal vector field  iff
\begin{equation}\label{e:conf field}
\mathcal{L}_X g = \frac {2}{n}(div_g X)g,
 \end{equation}
 where $\mathcal{L}_X$ is the Lie derivative operator and $div_g X$ is  the divergence operator $(div_g X)\, vol_g = \mathcal{L}_X vol_g$ defined with the help of Riemannian volume element $vol_g$ associated to $g$. In particular, a  gradient vector field $X=\nabla \phi $ is infinitesimal conformal vector iff
 \begin{equation}\label{e:inf conf v.f.}
 (\nabla^2\phi)_0=0.
 \end{equation}
 A short calculation, see \cite{Ya57}, \cite{Li} or \cite[(1.11)]{YaOb70}, shows that if $X$ is a conformal vector field then
 \begin{equation}\label{e:lap of div conf v.f.}
 \lap (div\, X)=-\frac {1}{n-1}(div\, X)S-\frac {n}{2(n-1)}X(S).
 \end{equation}

 \subsection{The Yamabe problem - Obata's uniqueness theorem}
The Yamabe type equation has its origin in both geometry and analysis. Yamabe \cite {Y} considered  the question of finding a conformal transformations of a given Riemannian metric on a compact manifold to one with a constant scalar curvature, see also
\cite{Tru} and \cite{A1,A2,A3,Au98}. When the ambient space is the Euclidean space $\Rn$,
G. Talenti \cite{Ta} and T. Aubin \cite{A1, A2,A3} described all positive solutions of a more general equation, that is the Euler -
Lagrange equation associated with the best constant, i.e., the norm, in the $L^p$ Sobolev embedding theorem.  With the help of the
stereographic projection, which is a conformal transformation, Yamabe's question for the standard round sphere turns into the $L^2$ case of Talenti's question. The solution of these special cases is an important step in solving the general Yamabe problem, the solution of which in
the case of a compact Riemannian manifold was completed in the 80's after the work of T. Aubin and R. Schoen \cite{A1,A2,A3,Au98,Sch,Sch2,Sch89}, see also \cite{LP}. It should be noted that the  "solution" used the positive mass theorem of R. Schoen and S.-T. Yau \cite{SchYa79}.  An alternative approach was developed by A. Bahri \cite{Bahri} where solutions of the Yamabe equation were obtained through "higher" energies of the Yamabe functional. As well known, in general, there is no uniqueness of the metric of constant scalar curvature within a fixed conformal class. However, with the exception of the round sphere, according to Obata's  theorem uniqueness (up to isometry) holds in a conformal class containing an Einstein metric.

\subsubsection{The Yamabe problem and functional}\label{ss:R Yamabe functional}
Let $(M^n,\, g)$ be a compact Riemannian manifold of dimension $n$.  The Yamabe problem is to find a metric $\bar g$ point-wise conformal to the Riemannnian metric $g$ of constant scalar curvature $\bar S$. Clearly, this is a type of uniformization problem which is one generalization of the classical surface case.

\subsubsection{Riemann surfaces - the 2-D case}
In the 2-D case, where we are dealing  with the uniformization of a closed orientable surface, if we set $\bar g = e^\phi g$, then the equation which needs to be solved is $$\triangle \phi -K=-\bar K e^\phi,$$
for some constant $\bar S$, where $S$ is the Gaussian curvature of $g$. By the {Gauss-Bonnet} formula  $$2\pi\chi(M)=\int_M K \, dv_g,$$ which determines the sign of $\bar K$. By the \emph{uniformization theorem} of the universal cover $\hat M$ of $M$, $M$ is biholomorphic to $\hat M/G$ for some $G$-properly discontinuous subgroup of $Aut(\hat M)$. Thus, depending on the sign, $(M,g)$ is hyperbolic, parabolic, or elliptic Riemann surface, i.e.,  it is conformal to one of constant Gauss curvature $-1,\ 0,\ 1$. Explicitly, depending on its genus,  $M$ is  conformal (in fact biholomorphic) to a surface in one of next three cases:
\begin{enumerate}
\item $\mathbb{H}/\Gamma$, for a properly discontinuous $\Gamma$ subgroup of $PSL(2,R)$-the automorphism group of the unit disc $\mathbb{D}$, when the genus of M at least two;
\item $\mathbb{C}/\Lambda$-elliptic curve, corresponding to a lattice $\Lambda=\{n_1\omega_1+n_2\omega_2\mid n_1, n_2\in Z\},  \omega/\omega_2\notin \mathbb{R}$ when the genus of $M$ is one;
\item $S^2$ of genus $0$.
\end{enumerate}

\subsubsection{The higher dimensional cases}
For $n\geq 3$ such a complete picture is not possible. It is customary to take the conformal factor in a way which is best suited for the problem, accordingly we begin with the form exhibiting the relation the critical Sobolev exponent. As well known, if we write the conformal factor in the form $\bar g=u^{4/(n-2)}\,g$ then the Yamabe problem becomes an existence problem for a positive solution to the Yamabe equation, see \eqref{e:s conf  change},
\begin{equation}\label{e:Riem Yamabe}
  4\frac {n-1}{n-2} \triangle u -  {S}\cdot\,u \ =\ -
\overline{{S}}\cdot\, u^{2^*-1}.
\end{equation}
where $\triangle u\ =\ \text{tr}^g (\nabla du)$,
$\text{S}$ and $\overline{\text{S}}$ are the scalar curvatures
of $g$ and $\bar g$, and $2^* = \frac {2n}{n-2}$ is the Sobolev conjugate exponent.

The Yamabe problem \ref{ss:R Yamabe functional} is of variational natural as we remind next. The critical points of the {Einstein-Hilbert (total scalar curvature) functional} $$\Upsilon(\bar g)=\left ( \int_M \bar S\, dv_{\bar g}\right )/ \left (\int_M \, dv_{\bar g} \right )^{2/{2^*}}$$ are {Einstein metrics}. The \emph{Yamabe functional} is obtained by restricting  $\Upsilon(\bar g)$ to the conformal class  $[g]=\{\bar g=u^{4/(n-2)}g\mid\, 0<u\in \mathcal{C}^\infty(M)\}$ and defining (a conformally invariant functional)
\begin{equation}\label{e:Yamabe Riem}
\Upsilon_g (u) =\left ( \int_M 4\frac {n-1}{n-2}\ \abs{\nabla u}^2\ +\
{S}\, u^2\, dv_g\right )/ \left ( \int_M u^{2^*}\, dv_g\right)^{2/{2^*}}.
\end{equation}
The critical points, i.e., the solutions of  $\frac {d}{dt} \Upsilon (u+t\phi)_{\vert_{t=0}}\ = \ 0, \quad \phi\in \mathcal{C}^\infty(M)$, are  metrics of {constant scalar curvature} (Yamabe metrics) since they are given by the solutions of \eqref{e:Riem Yamabe} with $\bar S$ the corresponding "critical" energy level. The Yamabe constant of $(M,g)$ is
\begin{equation*}
\Upsilon (M,[g])\equiv\Upsilon ([g]) =\ \inf \{ \Upsilon_g (u):\ u>0 \}.
\end{equation*}
The Yamabe invariant is the supremum $ \lambda(M)=\sup_{[g]} \Upsilon
([g])$.

According to the result of Aubin and Talenti, for the round unit sphere $\Upsilon (S^n,[g_{st}])=n(n-1)\omega_n^{2/n}$. The existence of a Yamabe metric is the content of the next theorem, which collects a number of remarkable results, see for example \cite{LP} for a full account,

 \begin{thrm}[N. Trudinger, Th. Aubin, R. Schoen; A. Bahri]
Let $(M^n, g)$, $n\geq 3$, be a compact Riemannian manifold. There is
a $\bar g\in [g]$,  s.t., $\bar S=const$.
\end{thrm}

The main steps in proof of the above theorem are as follows, see \cite{LP,Au98} for a full account.
\begin{itemize}
\item  We have $\Upsilon ([g]) \leq \Upsilon (S^n,
st)$. The Yamabe problem can be solved on any compact manifold M with
$\Upsilon ([g]) < \Upsilon (S^n, [g_{st}])$, see \cite{Y},  \cite{Tru}, and \cite{A2}.
\item  If $n \geq 6$ then $\Upsilon (S^n, [g_{st}])-\Upsilon ([g])  \geq c \norm {W^g}^2$, hence the Yamabe problem can be solved if $n\geq 6$ and $M$ is not locally conformally flat, see \cite{A2}.
\item If $3\leq n \leq 5$, or if $M$ is locally
conformally flat, then  the Yamabe problem has a solution since $\Upsilon (S^n, [g_{st}])- \Upsilon ([g]) \geq c
m_o$, where $m_o$ is the mass of a one point blow-up (stereographic
projection) of $M$, see \cite{Sch}.
\item  If $M$ is locally
conformally flat  a critical point of
the Yamabe functional exists (which may be of higher than $\Upsilon (M,[g])$ energy), see \cite{Bahri}.
\end{itemize}
Given the above existence of a Yamabe metric on every compact Riemannian manifold it is natural to study the question of uniqueness. When $\Upsilon (M,[g])\leq 0$ the Yamabe metric is unique in its conformal class as implied by the maximum principle. However, for $\Upsilon (M,[g])>0$ this is no longer true. An example of non-uniqueness  is provided by the round unit sphere as described next in \ref{ss:Obata uniqueness}.  Another example was given by  \cite{Sch89} in $S^1(R)\times S^n$. Remarkably, in the case of the sphere the set of solutions is non-compact (in the $\mathcal{C}^2$ topology), see further below, and it was conjectured in \cite{Sch89} that this is the only case, which became known as the compactness conjecture, see the review \cite{BrMa11} for further details and references. In short, the conjecture is true in the following cases:
\begin{itemize}
\item for a locally conformally flat manifold (different from the round sphere) \cite{Sch91a} and \cite{Sch91b};
\item for $n\leq 7$, see \cite{LiZh05} and \cite{Ma05};
\item for $8\leq n\leq 24$ provided that the positive mass theorem holds (this covers all cases of a spin manifold), see \cite{LiZh05}, \cite{LiZh07} for $8\leq n\leq 11$ and \cite{KhMaSch09} for $12\leq n\leq 24$.
\end{itemize}
Furthermore, the conjecture is not true for $n\geq 25$, see \cite{Br08} and \cite{BrMa09}. Putting the compactness conjecture aside we turn to the uniqueness result of Obata.

\subsubsection{Obata's uniqueness theorem for the Yamabe problem}\label{ss:Obata uniqueness}
The main result here is that  the conformal class of an Einstein metric on a connected compact Riemannian manifold $(M,\bar g)$ contains a unique Yamabe metric unless $M$ is the round unit sphere $S^n$  in $\mathbb{R}^{n+1}$. It should be noted that if $M$ is not conformal to the round sphere the Yamabe metrics are nondegenerate
global minima of the Einstein-Hilbert functional. The structure of the set of Yamabe metrics in conformal classes near a nondegenerate constant
scalar curvature metric was considered in the  smooth case in \cite{Koiso79}. An extension of Obata's result to a local
uniqueness result for the Yamabe problem in conformal classes near to that of a nondegenerate solution was established in \cite{dlPZ12}. After this short background we turn to Obata's theorem.

\begin{thrm}[\cite{Ob} and \cite{Ob72}]\label{t:Obata Yamabe}

\begin{enumerate}[a)]
 \item Let $(M,\bar g)$ be a connected compact Riemannian manifold which is Einstein and $\bar g=\phi^{-2}g$. If $\bar S=S=n(n-1)$, then $\phi=1$ unless  $(M,\bar g)=(S^n, g_{st})$.
 \item If $g$ is a Riemannian metric conformal to $g_{st}$,  $g_{st}=\phi^{-2} g$, with scalar curvature $S=n(n-1)$, then  $g$ is obtained from $g_{st}$ by a conformal diffeomorphism of the
sphere, i.e., there is $\Phi\in \text{Diff}\,(S^n)$ such that $g=\Phi^* g_{st}$ and up to an additive constant $\phi$ is an eigenfunction for the first eigenvalue of the Laplacian on the round sphere. In particular, $\nabla\phi$ is a gradient conformal field and for some $t$ we have $\Phi=exp(t\nabla \phi)$-the one parameter group of diffeomorphisms generated by  $\nabla\phi$.
\end{enumerate}
\end{thrm}
\begin{proof} In the proof of part a) we use the argument of \cite{BoEz87} and \cite{LP} which is very close to Obata's argument but uses the "new" metric as a background metric rather than the given Einstein metric. Suppose $\bar g$ is
Einstein, hence by \eqref{e:Ric_o conf change} we have $$0=\overline {Ric}_o= {Ric_o} + \frac
{n-2}{\phi}(\nabla^2\phi)_0.$$ Therefore, $(\nabla^2\phi)_0\  = \ -\frac
{\phi}{n-2}{Ric_0}.$ From the  contracted Bianchi identity  and
$S$=const we have $\nabla^*Ric=\frac {1}{2}\nabla S =0,$ hence $\nabla^* \left ( Ric(\nabla \phi, . )\right) =(\nabla^* Ric) (\nabla \phi)+g(Ric,\nabla^2\phi)=\frac 12g(\nabla S, \nabla \phi) -\frac
{\phi}{n-2}|Ric_o|^2.
$
Integration over $M$ and an application of the divergence theorem  shows that $g$ is \emph{also an Einstein metric}, $Ric_o =0$. This implies $(\nabla^2\phi)_0=0$, hence $\nabla \phi$ is a gradient conformal vector field, see \eqref{e:inf conf v.f.}. Now, from \eqref{e:lap of div conf v.f.} taking into account $S=n(n-1)$ it follows  $\lap (\lap\phi+n\phi)=0$, hence by the maximum principle we have $\lap u=-nu$, where $u=\phi+a$ for some constant $a$. Notice that we also have $(\nabla^2u)_0=0$.   Hence by Obata's result in the eigenvalue Theorem \ref{t:Riem LichObata} either $u=$const or  $g$ is isometric to $g_{st}$ and $u$ is a restriction of a linear function to $S^n$, $u=(a_ox_o+\dots+a_nx_n)\vert_{S^n}$, which implies the claimed form of $\phi$.
\end{proof}
We note that in the case of the round sphere, once we proved that $g$ is also Einstein, we can conclude that $g$ is isometric to $g_{st}$ since it is Einstein and  conformally flat, $W=0$, see \cite{Kuh88}, \cite{KuRa09}, \cite{MRTZ09}, \cite{KiMa09} for further details and references on conformal transformations between Einstein spaces in a variety of spaces. Thus, there is an isometry $\Phi:(S^n,g)\rightarrow (S^n,g_{st})$, $\Phi^*g_{st}=\phi^{-2}g_{st}$ hence $F\in C(S^n,g_{st})$. Part b) of the above theorem shows that $\Phi$ belongs to the largest connected subgroup of $C(S^n,g_{st})$ and determines the exact form of $\phi$.
The same conclusion can be reached with the help of the stereographic projection and relates the analysis to the Liuoville's theorem and the best constant in the $L^2$ Sobolev embedding theorem in Euclidean space. In fact, using the stereographic projection we can reduce to a conformal map of the Euclidean space, which sends the Euclidean metric to a conformal to it  Einstein metric. By a purely local argument, see \cite{Br25}, the resulting system can be integrated, in effect proving also Liuoville's theorem, which gives the form of $\phi$ after transferring the equations back to the unit sphere. Such argument was used in the quaternionic contact setting \cite{IMV} to classify all qc-Einstein structures on the unit $4n+3$ dimensional sphere (quaternionic Heisenberg group) conformal to the standard qc-structure on the unit sphere. We will come back to the qc-Liouville theorem later in the paper, see Section \ref{t:qcLiouville}.

In any case,  the key point here which will be used in the sub-Riemannian CR or QC setting is that Obata's argument shows the validity  of a system of partial differential equations, namely, $(\nabla^2\phi)_0=0$ assuming \eqref{e:s conf  change} holds with $\bar g$ being Einstein and $g$ of constant scalar curvature. On the other hand, using the stereographic projection, Yamabe's equation on the round sphere turns into
\eqref{e:Riem Yamabe} for the Euclidean Laplacian with $S=0$ and $\bar S$=const after interchanging the roles of $g$ and $\bar g$, i.e., assuming that $g$ is the "background" standard constant curvature metric and $\bar g$ is the "new" conformal to $g$ metric of constant scalar curvature. This is nothing but the equation characterizing the extremals of the variational problem associated to the $L^2$ Sobolev embedding theorem. An alternative to Obata's argument is then the symmetrization argument (described briefly in Section \ref{s:FS inequality}).

\subsection{Sub-Riemannian comparison results and Yamabe type problems - a summary}\label{ss:sub-Riemannian compariosn note}
The interest in relations between the spectrum of the Laplacian and
geometric quantities justified the interest in Lichnerowicz-Obata type
theorems in other geometric settings such as Riemannian foliations (and the
eigenvalues of the basic Laplacian) \cite{LR98,LR02}, \cite{JKR11} and \cite%
{PP11}, to CR geometry (and the eigenvalues of the sub-Laplacian) \cite{Gr},
\cite{Bar}, \cite{CC07,CC09a,CC09b}, \cite{ChW}, \cite{Chi06}, \cite{LL}, and to general
sub-Riemannian geometries, see \cite{Bau2} and \cite{Hla}. Complete results have been achieved in the settings of (strictly pseudoconvex) CR, \cite{Gr},\cite{CC09a,CC09b}, \cite{CC07}, \cite{Chi06},\cite{LW,LW1},\cite{IVO,IV3}, and QC, \cite{IPV1,IPV2,IPV3}, geometries which shall be covered in Sections \ref{s:CR Lichnerowicz-Obata} and \ref{s:QC Lichnerowicz-Obata}.

As far as other comparison results are concerned we mention

(i) \cite{Ru94} for a  Bonnet-Myers type theorem on general 3-D
CR manifolds;

(ii) \cite{Hughen}, where a Bonnet-Myers type theorem on a three dimensional Sasakian was proved.

Both of the above papers use analysis of the second-variation formula for sub-Riemannian geodesics.

(iii) \cite{ChY} for an isoperimetric inequalities and volume comparison theorems on CR manifolds.

 (iv) \cite{Bau2}, \cite{BauBonGarMun14}, \cite{BauKim14}, \cite{BauKimWa14}, \cite{BauWa14}, \cite{GrTh14a,GrTh14b} where an extension  to the sub-Riemannian setting of the Bakry-Emery technique on curvature-dimension inequalities are used to obtain Myers-type theorems, volume doubling,  Li-Yau, Sobolev and Harnack inequalities, Liouville theorem. Such inequalities are obtained usually under a transverse symmetry assumption. The latter means that we are actually dealing with a Riemannian manifold with bundle like metrics which are foliated by totally geodesic leaves. Thi scondition equivalent to vanishing torsion in the QC setting (qc-Einstein) and is not very far from  the Sasakian case (vanishing torsion) in the CR case.

 (v) \cite{AgLee14}, \cite{AgBaRi14}, \cite{AgL11}, where sub-Riemannian geodesics and measure-contraction properties are used to establish  for Sasakian manifolds results such as a Bishop comparison theorem, Laplacian and Hessian comparison, volume doubling, Poncar\'e and Harnack inequalities, and Liuoville theorem.
comparison results in the Sasakian case.

(vi) \cite{Hla} for Lichnerowicz type estimates and a Bonnet-Myers theorems in some special sub-Riemannian geometries.

A variant of the Yamabe problem in the setting of a compact strictly pseudoconvex pseudohermitian  manifold (called here simply CR manifold) is  the CR Yamabe problem where one seeks in a fixed pseudoconformal class  of pseudo-Hermitian structures on a compact CR manifold one with  constant scalar curvature (of the canonical Tanaka-Webster connection). After the works of D. Jerison \& J. Lee \cite{JL1}  - \cite{JL4} and N.Gamara \& R. Yacoub \cite{Ga}, \cite{GaY} the CR Yamabe problem on a compact manifold is complete. The case of the standard CR structure on the unit sphere in $\mathbb{C}^n$ is equivalent to the problem of determining the best constant in the $L^2$ Folland \& Stein \cite{FS} Sobolev type embedding inequality on the Heisenberg group. The best constant in the $L^2$ Folland \& Stein inequality together with the minimizers were determined recently using a different from \cite{JL3} method by Frank \& Lieb \cite{FrLi}, see also \cite{BFM}. Nevertheless this simpler approach \emph{does not} yield the uniqueness result of D. Jerison \& J. Lee. A positive mass theorem in the three dimensional case was proven recently in \cite{ChMaYa13}.

In the other case of interest, the qc-Yamabe problem was studied in \cite{IMV, IMV1, IMV2} and \cite{Wei}. According to \cite{Wei} the Yamabe constant of a compact qc manifold is less than  or equal to that of the standard qc sphere. Furthermore,
if the constant is strictly less than the corresponding constant of the sphere, the qc-Yamabe problem has a solution, i.e., there is a conformal
3-contact form for which the qc-scalar curvature is constant. The Yamabe constant of the standard qc structure on the unit $(4n+3)$-dimensional
sphere was determined in \cite{IMV2} with the help of a clever center of mass argument following in the footsteps of the CR case
\cite{FrLi} and \cite{BFM}. However, due to the limitations of the method \cite{IMV2} does not exclude the possibility that in the qc-conformal class of the standard
qc structure there are qc Yamabe metrics of higher energies. The seven dimensional case was settled completely earlier in \cite{IMV1}. A conformal curvature tensor was found in \cite{IV}, which should prove useful in establishing existence of a solution to the qc-Yamabe problem in the qc locally non-flat case.

Finally, we mention \cite{ChLZh1,ChLZh2} where the sharp Hardy-Littlewood-Sobolev inequalities in the quaternion and octonian versions of the approach found by Frank and Lieb was developed. In particular, at this point the sharp constants in the Hardy-Littlewood-Sobolev inequalities on all groups of Iwasawa type are known.

\section{The Folland-Stein inequality on groups of Iwasawa type}\label{s:FS inequality}

 We start by recalling the following
embedding theorem due to Folland and Stein \cite{FS}.
\label{T:Folland and Stein} Let $ {\bG}$ be a Carnot group
${\bG}$ of homogeneous dimension $Q$, fixed metric $g$ on the the "horizontal"  bundle spanned by the first layer {\ and Haar measure $dH$}. For any $1<p<Q$ there exists
$S_p=S_p({\bG})>0$ such that for $u\in C^\infty_o(\bG)$ we have
\begin{equation}  \label{FS}
\left( \int_\bG \, |u|^{p^*}\, dH\right)^{1/p^*} \leq\ S_p\ \left(\int_\bG |Xu|^p\, dH\right)^{1/p},
\end{equation}
where $|Xu|=\sum_{j=1}^m |X_ju|^2$ with $X_1,\dots, X_m$ denoting an orthonormal  basis of the first layer of ${\bG}$
{and $p^*= \frac {pQ}{Q-p}$.}
In the case $\bG=\Rn$ this embedding is nothing but the Sobolev embedding theorem.  We insist on $X_1,\dots, X_m$ denoting an orthonormal  basis of the first layer in order to have a well defined constant which obviously depends on the chosen (left invariant) metric.  For the sake of brevity we do not give the definition of a Carnot group since our focus is in the particular case of groups of Iwasawa type, in which case there is a natural metric. Also, the case $p=1$ which we did not include above, is the isoperimetric inequality, see  \cite{CDG2} for the proof in a much wider setting, which as well known \cite{FeFl60,Maz61}, see also \cite{Ta} and \cite{S-C02},   implies the whole range of inequalities \eqref{FS}.

The most basic fact of the above inequality is its invariance under translations and dilations. The latter fact determines the relation between the exponents $p$ and $p^*$  appearing in both sides. For a function $u\in C^{\infty}_o (\bG) $ we let
\begin{equation}\label{translation}
\tau_h u\ \overset{def}{=}\  u\circ \tau_h, \quad\quad \quad h\in \bG ,
\end{equation}
where $\tau_h:\bG\to \bG$ is the operator of left-translation $\tau_h(g) = hg$,
 and also
\begin{equation}\label{scaling}
u_\lambda\ \equiv\ \lambda^{Q/p\text{*}}\  \delta_\lambda u\ \overset{def}{=}\ \lambda^{Q/p\text{*}}\ u\circ \delta_\lambda, \quad\quad\quad \lambda >0.
\end{equation}
It is easy to see that the norms in the two sides of  the Folland-Stein inequality  are invariant under the translations (\ref{translation}) and the rescaling \eqref{scaling}.

Let $S_p$ be the best constant in the Folland-Stein inequality, i.e., the smallest constant for which \eqref{FS} holds. The  equality  is
achieved on the space ${\overset{o}{\mathcal{D}}}\,^{1,p}(\bG)$, where for a domain $\Omega\subset \bG$ the space $\domoOm$  is defined as the closure of $C_o^{\infty}(\Omega)$
with respect to the norm
\begin{equation}
 \norm{u}_{\domoOm} = \left(\int_{\Omega} |Xu|^p dH\right)^{1/p}.
\end{equation}
 This fact  was proved in \cite{Va1} with the help of P.L. Lions' method of concentration compactness. The question of determining the norm of the embedding, i.e., the value of the best constant is open with the exception of the Euclidean case and the $p=2$ case on the (two step Carnot) groups of Iwasawa type.
 In the Euclidean case,  a symmetrization argument involving symmetric decreasing rearrangement, see \cite{Ta62}, can be used to show that equality is achieved for radial functions which can be determined explicitly. As of now there is no such argument in the non-Euclidean setting which to a large degree is the reason for the much more sophisticated analysis in the sub-Riemannian setting. However, the recently found approach \cite{FrLi} and \cite{BFM} based on the center of mass argument allows the determination of the sharp constant (in fact  in the Hardy-Littlewood-Sobolev inequality) in the geometric setting of groups of Iwasawa type when $p=2$. This analysis exploits the Cayley transform and the conformal invariance of the associated Euler-Lagrange equation
which is {the} Yamabe equation on the corresponding Iwasawa group,
  \begin{equation}\label{Yamabeomegasymm}
\lap u\ =- u^{\frac{Q+2}{Q-2}}, \qquad u\in \DoG, \quad u\geq 0.
\end{equation}
Of course, in order to give a geometric meaning of the equation one needs to use the relevant geometries and their "canonical" connections which we do {in} Section \ref{s:Iwasawa sub-Riem geom}. In the Euclidean and CR cases these are just the well known Levi-Civita and Tanaka-Webster connections. In the quaternionic and octonian {cases} the geometric picture emerged only after the work of Biquard \cite{Biq1,Biq2}. The goal of this section is to give some ideas surrounding the analysis of the Yamabe equation as a partial differential equation and some of the known results on the optimal constants which largely belong to the area of analysis.  The key results on the optimal constants are summarized in the following two theorems in which $m$ is the dimension of the first layer, while $k$ is the dimension of the center of the Iwasawa algebra.

\begin{thrm}[\cite{JL3},\cite{FrLi},\cite{IMV1,IMV2},\cite{ChLZh1}]\label{T:Iwasawa groups Minimizers}
Let $\bG$ be a group of Iwasawa type. For every $u\in D^{1,2}( \bG)$ one has the Folland-Stein inequality \eqref{FS}
with
\begin{equation}\label{best}
S_2= \frac{1}{\sqrt{m(m + 2(k-1))}}\ 4^{\frac {k}{m+2k}} \pi^{-\frac {m+k}{2(m+2k)}}\ \left(\frac{\Gamma(m + k)} {\Gamma\left(\frac {m+k}{2}\right)}\right)^{\frac {1}{m+2k}} .
\end{equation}
An extremal is given by the function
\begin{equation}\label{e:FS extremal fns}
F(g)= \gamma(m,k)\ \left[(1 + |x(g)|^2)^2+ 16 |y(g)|^2)\right]^{-(Q-2)/4},
\end{equation}
where
\[
\gamma(m,k)= \left[4^k\ \pi^{-(m+k)/2(m+2k)}\ \frac{\Gamma(m + k)}{\Gamma((m + k)/2)} \right]^{(m+2(k-1))/2(m+2k)}.
\]
Any other non-negative extremal is obtained from $F$ by
\eqref{translation} and \eqref{scaling}.
\end{thrm}
We remark that  \eqref{e:FS extremal fns} is a solution to the Yamabe equation on any group of Heisenberg type  \cite{GV2} which was found earlier (and seems to have been forgotten) in the case of Iwasawa groups  in  \cite[Proposition2]{KapPutz1}.
It also should be noted that \cite{JL3} and \cite{IMV1} actually determine all critical points of the associated to \eqref{FS} variational problem rather than only the functions with lowest energy. In fact, \cite{JL3} solves completely the Yamabe equation \eqref{Yamabeomegasymm} on the Heisenberg group while \cite{IMV1} achieves this on the seven dimensional quaternionic Heisenberg group (the higher dimensional case {is settled in the preprint \cite{IMV15a}}). We report on the ideas behind these proofs in Sections \ref{ss:Jerison and Lee} and \ref{ss:7D QC Heisenebrg Yamabe} which involve ideas inspired by Theorem \ref{t:Obata Yamabe}.
In the remaining cases of Iwasawa type groups the partial result in the next Theorem \ref{T:CS} supports the general agreement that \eqref{e:FS extremal fns} gives all solutions.
\begin{thrm}[\cite{GV}]\label{T:CS}
All partially symmetric solutions  of the Yamabe equation on a group of Iwasawa type are given {by} \eqref{e:FS extremal fns} up to dilation and translation.
\end{thrm}
For the definition of partially symmetric solution we refer to Section \ref{s:Iwasawa partial symmetry}.

\subsection{Groups of H-type and the Iwasawa groups}

Let $\algn$ be a 2-step nilpotent Lie algebra equipped with a scalar product $<.,.>$ for which $\algn = V_1 \oplus V_2$-an orthogonal direct sum, $V_2$ is the center of $\algn$.  Consider the map $J:V_2\to End(V_1)$ defined by
\begin{equation}\label{J1}
<J(\xi_2)\xi_1',\xi_1''>\ =\ <\xi_2,[\xi_1',\xi_1'']>,\  \text{ for}\ \xi_2\in V_2\ \text{ and }\ \xi_1', \xi_1''\in V_1 .
\end{equation}
By definition we have that $J(\xi_2)$ is skew-symmetric. Adding the additional condition that it is actually an almost complex structure on $V_1$ when $\xi_2$ is of unit length  \cite{K1} motivates the next definitions.
A 2-step nilpotent Lie algebra $\algn$ is said to be of \emph{Heisenberg type} \index{Heisenberg type!algebra} if for every $\xi_2 \in V_2$, with $|\xi_2|=1$, the map $J(\xi_2):V_1\to V_1$ is orthogonal.
A simply connected connected Lie group $\bG$ is called of Heisenberg type (or H-type) \index{H-type group} if its Lie algebra $\algn$ is of Heisenberg type.
We shall  use the exponential coordinates and regard $G=exp\ \algn$, so that the product of two elements of $N$ is
\begin{equation}\label{e:H-type group product}
(\xi_1,\xi_2)\cdot (\xi'_1,\xi'_2)=(\xi_1+\xi'_1, \xi_2+\xi'_2+\frac 12[\xi_1,\xi'_1]),
\end{equation}
taking into account the Baker-Campbell-Hausdorff formula. Correspondingly we shall use $V_i$, $i=1,2$ to also denote the sub-bundle of left invariant vector fields
which coincides with the given $V_i$ at the identity element.
In \cite{K1} Kaplan found the explicit form of the fundamental solution of the sub-Laplacian on every group of H-type, where the sub-Laplacian is the operator
\begin{equation}\label{e:sub-laplacian H-type}
\lap\ =\  \underset {j=1}{\overset
{m}{\sum}}\,X_j^2,
\end{equation}
for vector fields $X_j$, $j=1, \dots,m$ which are an orthonormal basis of $V_1$.

On a group $N$ of Heisenberg type there is a very important homogeneous norm (gauge) \index{homogeneous!norm} given by \index{gauge}
\begin{equation}\label{Hgauge}
N(g)=\bigl( \abs{\xi_1(g)}^4+16\abs{\xi_2(g)}^2\bigr)^{1/4},
\end{equation}
which induces a left-invariant distance. Kaplan proved in \cite{K1} that in a group of Heisenberg type, in particular in every Iwasawa group, the fundamental solution $\Gamma$ of the sub-Laplacian $\mathcal{L}$, see \eqref{e:sub-laplacian H-type}, is given  by the formula
\begin{equation}\label{GammaH}
\Gamma(g,h)\ =\ C_Q\ N(h^{-1} g)^{-(Q-2)}, \hskip.7in g,h\in N, g\neq h,
\end{equation}
where $C_Q$ is a suitable constant.
\begin{rmrk}\label{r:gromov limit}
It is known that the distance induced by the gauge \eqref{Hgauge} is the Gromov limit of a one parameter family of Riemannian metrics on the group $N$ \cite{Ko2}, see also \cite{BRed} and \cite{CDPT07}.
\end{rmrk}
Kaplan and Putz \cite{KapPutz2}, see also [\cite{Ko2}, Proposition 1.1], observed that the nilpotent part $N$ in the Iwasawa \index{Iwasawa!decomposition} decomposition $\bG=NAK$ of every semisimple Lie group $\bG$ of real rank one is of Heisenberg type. We shall refer to such a group as \emph{Iwasawa group} \index{Iwasawa!group} and call the corresponding Lie algebra \emph{Iwasawa algebra}.

The Heisenberg type groups allowed for the generalization of many important concepts in harmonic analysis and geometry, see \cite{KapPutz2}, \cite{KapR}, \cite{Ko2}, \cite{DamRic2} and the references therein, in addition to the above cited papers.   Another milestone was achieved in \cite{CDKR}, which allowed {to circumvent}  the classification {of the }rank one symmetric spaces and  the heavy machinery of the semisimple Lie group theory, when studying the non-compact symmetric spaces of real rank one. Specifically, in \cite{CDKR} the authors considered the H-type algebras satisfying the so called $J^2$ condition defined in \cite{CDKR}, see also \cite{CDKR2}.
\begin{dfn}\label{d:J2 cond}
We say that the H-type algebra $\algn$ satisfies the $J^2$ condition \index{$J^2$ condition} if for every $\xi_2, \xi_2'\in V_2$ which are orthogonal to each other, $<\xi_2, \xi_2'>=0$, there exists $\xi_2''\in V_2$ such that
\begin{equation}\label{e:J2 cond}
J(\xi_2)J(\xi_2')=J(\xi_2'').
\end{equation}
\end{dfn}

A noteworthy result here is the following Theorem of \cite{CDKR}, see also \cite{Ciatti}, which can be used to show that if $N$ is an H-type group, then the  Riemannian space $S=NA$ is symmetric iff the Lie algebra $\algn$ of $N$ satisfies the $J^2$ condition, see [\cite{CDKR}, Theorem 6.1].
\begin{thrm}\label{t:Iwasawa is J2}
If $\algn$ is an H-type algebra satisfying the
$J^2$-condition, then  $\algn$ is an Iwasawa type algebra.
\end{thrm}
This fundamental result has many consequences, among them {it allows the}  unified proof of some classical results on symmetric spaces, in addition to some beautiful properties of extensions of the classical Cayley transform, inversion and Kelvin transform, which are of a particular importance for our goals.

From a geometric point of view,  the above Iwasawa groups can be seen as the nilpotent part in the Iwasawa decomposition of the isometry group of the non-compact symmetric spaces  $M$ of real rank one.   Such a space can be expressed as a homogeneous space  $G/K$ where $G$  is the identity component of the isometry group of $M$, i.e., one of the simple Lorentz groups $SO_o(n,1)$, $SU(n,1)$, $Sp(n,1)$ or $F_{4(-20)}$,
and $K$ is a maximal compact subgroup of $G$, see \cite{Helgason}, namely, $K = SO(n)$, $SU(n)$, $Sp(n)Sp(1)$, or $Spin(9)$, respectively, see for example \cite[Theorem 8.12.2]{Wolf} or \cite{Helgason}. Thus $M=H^n_\mathbb{K}$ is one of the hyperbolic spaces  over the real, complex, quaternion or Cayley (octonion) numbers, respectively. As well known, these spaces carry canonical
Riemannian metrics with sectional curvature $k= -1$ for $\mathbb{K}=\mathbb{R}$ and $-1 < k <
-1/4$ in the remaining cases cases. Here, $\mathbb{K}$ denotes one of the real division algebras: the real numbers $\mathbb{R}$, the complex numbers $\mathbb{C}$, the quaternions $\mathbb{H}$, or the octonions $\mathbb{O}$.

Writing $G=NAK$ and letting $S=NA$, $A$-one-dimensional Abelian subalgebra, we have that $S$ is a closed subgroup of $G$, which is isometric with the hyperbolic space \index{hyperbolic space} $M$, thus  giving the corresponding hyperbolic space \index{hyperbolic space!Lie group structure} a Lie group structure.  The nilpotent part $N$ is isometrically isomorphic to $\mathbb{R}^n$ in the degenerate case when the Iwasawa group is Abelian or to one of the Heisenberg groups  $\QK \ =\ \Kn\times\text {Im}\, \mathbb{K}$ \index{Heisenberg group $\QK$} with the group law
given by
\begin{equation}\label{e:H-type Iwasawa groups}
  (q_o, \omega_o)\circ(q, \omega)\ =\ (q_o\ +\ q, \omega\ +\ \omega_o\ + \ 2\ \text
{Im}\  q_o\, \bar q),
\end{equation}
where $q,\ q_o\in\Kn$ and $\omega, \omega_o\in \text {Im}\, \mathbb{K}$.  In particular, in the non-Euclidean case the Lie algebra $\algn$ of $N$ has center of dimension $\dim V_2=1$, $3$, or $7$.

Iwasawa groups are distinguished also by the properties of the sphere product $S_1(R_1)\times S_2(R_2)$, where, for $j=1,2$, $S_j(R_j)$ is the  sphere of radius $R_j$ in $V_j$-the two layers of the 2-step nilpotent Lie algebra.  In fact, for a group of Iwasawa type the Kostant double-transitivity theorem \index{Kostant double-transitivity theorem} shows that the action of $A(N)$ is transitive, where as before $A(N)$ stands for the orthogonal automorphisms of $N$, see \cite[Proposition 6.1]{CDKR2}.  This fact points to the importance of the bi-radial or cylindrically symmetric functions. Notice that both the fundamental solution  of the sub-Laplacian and the known solutions  of the Yamabe equation have such symmetry, see \eqref{GammaH} and \eqref{e:FS extremal fns}.

Motivated by the way the Iwasawa type groups appear as "boundaries" of the hyperbolic spaces,  Damek \cite{Dam2} introduced a generalization of the hyperbolic spaces as follows. For a group $N$ of H-type consider a semidirect product with a one dimensional Abelian group, i.e., take the multiplicative group $A=\mathbb{R}^+$ acting on an H-type group by dilations given in exponential coordinates by the formula $\della(\xi_1,\xi_2)=(a^{1/2}\xi_1, a \xi_2)$ and define $S=NA$ as the corresponding semidirect product. Thus, the Lie algebra of $S$ is  $\algs = V_1\oplus V_2\oplus \alga$, $\algn=V_1\oplus V_2$, with the bracket  extending the one on $\algn$ by adding the rules
\begin{equation}\label{e:solvable extension bracket}
[\zeta,\xi_1 ] = \frac 12\xi_1, \quad [\zeta,\xi_2]=\xi_2\quad \xi_i\in V_i,
\end{equation}
where $\zeta$ is a unit vector in $\alga$, so that $S$ is the connected simply connected Lie group with Lie algebra $\algs$.
 In the coordinates $(\xi_1, \xi_2, a)=\exp (\xi_1+\xi_2) \exp (\log a\zeta)$, $a>0$, which parameterize $S=exp\, \algs$, the product rule of $S$ is given by the formula
\begin{equation}\label{e:product on S}
(\xi_1,\xi_2,a)\cdot (\xi'_1,\xi'_2,a')=(\xi_1+a^{1/2}\xi'_1,\, \xi_2+a\xi'_2+\frac 12a^{1/2}[\xi_1,\xi'_1],\, aa'),
\end{equation}
for all $(\xi_1,\xi_2,a),\,  (\xi'_1,\xi'_2,a')\in \algn\times \mathbb{R}^+$.
Notice that $S$ is a solvable group. We equip the Lie algebra $\algs$ with the inner product
\begin{equation}\label{e:metric on S}
<(\xi_1, \xi_2, a), (\tilde\xi_1, \tilde\xi_2, \tilde a)>=<(\xi_1, \xi_2), (\tilde\xi_1, \tilde\xi_2)> + a\tilde a
\end{equation}
 using the fixed inner product on $\algn$ and then define a corresponding translation invariant Riemannian metric on $S$. The main result of \cite{Dam2} is that the group of isometries $Isom(S)$ of $S$ is as small as it can be and equals $A(S)\ltimes S$ with $S$ acting by left translations, unless $N$ is one of the Heisenberg groups \eqref{e:H-type Iwasawa groups}, i.e., $S$ is one of the classical hyperbolic spaces. Here, $A(S)$ denotes the group of automorphisms of $S$ (or $\algs$) that preserve the  left-invariant metric on $S$. The spaces constructed in this manner became known as Damek-Ricci \index{Damek-Ricci space} spaces, see \cite{BeTrVa} for more details. It was shown in \cite{DamRic}  that the just described solvable extension \index{solvable extension} of H-type groups, which are not of Iwasawa type, provide noncompact counter-examples to a conjecture of Lichnerowicz, which asserted that harmonic Riemannian spaces \index{harmonic Riemannian spaces} must be rank one symmetric spaces.

\subsection{The Cayley transform}\label{ss:cayley tranform H-type}
 In this section we focus on the Cayley transform, of which we shall make extensive use  later.  Here, we give the well known abstract definition valid in the setting of groups of H-type. Other explicit formulas will be given in the CR and QC cases in Sections \ref{ss:CR geometry} and \ref{ss:qc sphere}. Starting from an $H$-type group, its solvable extension $S$ defined above has  the following  realizations, \cite{DamRic2}, \cite{CDKR} and \cite{CDKR2}.

First, consider the "Siegel domain"  or an upper-half plane model of the hyperbolic space
\begin{equation}\label{e:siegel model}
D=\{ p=(\xi_1,\xi_2,a)\in \algs=V_1\oplus V_2\oplus \alga: \ a>\frac 14|\xi_1|^2\}.
\end{equation}
Consider the map $\Theta:S\rightarrow S$,
\begin{equation}\label{e:Theta}
\Theta(\xi_1,\xi_2, a) = (\xi_1,\xi_2, a+\frac 14|\xi_1|^2),
\end{equation}
which is injective map of $S$ into itself. Here we use $a$ to denote the element $a\zeta\in A$, $\zeta$ defined after \eqref{e:solvable extension bracket}, and we regard $D$ as a subset of $S$ using the exponential coordinates. Thus, the group $S$ acts simply transitively on $D$ by conjugating left
multiplication in the group $S$ by $\Theta$, $s\cdot p=\Theta s\cdot(\Theta^{-1}p)$ for $s\in S$ and $p\in D$, while $N$ acts simply transitively on the level sets of $h=a-\frac 14|\xi_1|^2$. In particular, we can define an invariant metric on $D$ by pulling via $\Theta$
the left-invariant metric \eqref{e:metric on S} of $S$  to $D$, thus making $\Theta$ an isometry, cf. [\cite{CDKR2}, (3.3)].

Second,  there is the "ball" model of $S$,
\begin{equation}\label{e:ball model}
B = \{ (\xi_1,\xi_2, a) \in   \algs=V_1\oplus V_2\oplus \alga: \ |\xi_1|^2+|\xi_2|^2+a^2< 1\},
\end{equation}
equipped with the metric obtained  from $D$ via the inverse of the so called \emph{Cayley transform} $\mathcal{C}:B\rightarrow D$ \index{Cayley transform} defined by $\mathcal{C}(\xi_1,\xi_2, a)=(\xi_1{'},\xi_2{'}, a{'} )$, where
\begin{equation}\label{e:cayley for H-type}
\begin{aligned}
& \xi_1'= \frac {2}{(1 - a)^2 + |\xi_2|^2}\, \left ( (1 - a)\xi_1 + J(\xi_2)\xi_1\right),\\
& \xi_2'= \frac {2}{(1 - a)^2 + |\xi_2|^2}\, \xi_2, \qquad a'=\frac {1-a^2 - |\xi_2|^2 }{(1 - a)^2 + |\xi_2|^2}.
\end{aligned}
\end{equation}
The inverse map $\mathcal{C}^{-1}:D\rightarrow B$  is given by $\mathcal{C}^{-1}(\xi_1{'},\xi_2{'}, a{'} )= (\xi_1,\xi_2, a)$, where
\begin{equation}\label{e:inverse cayley for H-type}
\begin{aligned}
& \xi_1= \frac {2}{(1 +a{'} )^2 + |\xi_2{'}|^2}\, \left ( (1 +a{'})\xi_1{'}  - J(\xi_2{'})\xi_1{'} \right),\\
& \xi_2 = \frac {2}{(1 +a{'} )^2 + |\xi_2{'}|^2}\, \xi_2{'}, \qquad a=\frac {-1+a{'} ^2 - |\xi_2{'} |^2}{(1 +a{'} )^2 + |\xi_2{'}|^2}.
\end{aligned}
\end{equation}
For other versions of the Cayley transform see [\cite{FarKor}, Chapter X].
The Jacobian of $\mathcal{C}$ and its determinant were computed in \cite{DamRic2}. The latter is given by the formula
$\det\mathcal{ C}' (\xi_1,\xi_2, a) = 2^{m+k+1}\left ((1-a)^2+|\xi_2|^2  \right )^{-(m+2k+2)/2}$,
where, as before, $m=\dim V_1$, $k=\dim V_2$.

It is very important and we shall make use of the fact that the Cayley transform can be extended by continuity to a bijection (denoted by the same letter!)
\begin{equation}\label{e:Cayley transform to Siegel bdry}
C:\partial B \setminus \{(0,0,1)\}\rightarrow \partial D,
\end{equation}
where $(0,0,1)$ (referred to as "$\zeta$" for short) is the point on the sphere where $\xi_1=\xi_2=0$ and the third component is $\zeta$ in agreement with out notation set after equation {\eqref{e:Theta}}.
The boundaries of the ball and Siegel domain models are, respectively,
\begin{equation}\label{e:bdry siegel model}
\Sigma\equiv\partial D=\{ p=(\xi_1',\xi_2',a')\in \algs=V_1\oplus V_2\oplus \alga: \ a'=\frac 14|\xi_1'|^2\}
\end{equation}
and
\begin{equation}\label{e:bdry ball model}
\partial B = \{ (\xi_1,\xi_2, a) \in   \algs=V_1\oplus V_2\oplus \alga: \ |\xi_1|^2+|\xi_2|^2+a^2= 1\},
\end{equation}
The group of Heisenberg type $N$ can be identified with $\Sigma$ via the map
\begin{equation}\label{e:identify bdry of siegel and H-group}
(\xi_1', \xi_2')\mapsto (\xi_1', \xi_2', \frac 14|\xi_1'|^2).
\end{equation}
With this identification we obtain the form of the Cayley transform (stereographic projection) identifying the sphere minus the point "$\zeta$" and the H-type group,
$\mathcal{C}:\partial B \setminus \{(0,0,1)\}\rightarrow N$\index{Cayley transform} defined by $\mathcal{C}(\xi_1,\xi_2, a)=(\xi_1{'},\xi_2{'})$, where
\begin{equation}\label{e:Cayley transform to H-type group}
\begin{aligned}
& \xi_1'= \frac {2}{(1 - a)^2 + |\xi_2|^2}\, \left ( (1 - a)\xi_1 + J(\xi_2)\xi_1\right),\\
& \xi_2'= \frac {2}{(1 - a)^2 + |\xi_2|^2}\, \xi_2.
\end{aligned}
\end{equation}
Later, we shall make use of this "boundary" Cayley transform in the case of the Heisenberg and quaternionic Heisenberg group in which place we shall give some other explicit formulas. In particular, we shall use that the Cayley transform is a pseudoconformal map in the CR case and  quaternionic contact conformal  transformation in the QC case.   The Cayley transform is also  a 1-quasiconformal map \cite{Banner}, see also \cite{ACD}. The definition of the "horizontal" space  in the tangent bundle of the sphere and the distance function on the sphere require a few more details for which we refer to \cite{CDKR2} and \cite{Banner}.  Multicontact maps and their rigidity in Carnot groups have been studied  in \cite{Pansu87}, \cite{Reimann01}, \cite{Ko05}, \cite{CdMKR02}, \cite{CdMKR05}, \cite{CC06}, \cite{dMO10}, \cite{Ot05}, \cite{Ot08}, \cite{OtWa11}.

 \subsection{Regularity of solutions to the Yamabe equation}\label{s:regularity}
 In order for the geometric analysis to proceed we need the next regularity result for the Euler-Lagrange equation associated to the problem of the optimal constant in \eqref{FS}.
 \begin{thrm}\label{t:harnack and holder regul}
Let $\Om$ be an open set in a Carnot group $\bG$. Suppose $u\in \domoOm$ is a weak solution to the equation
\begin{equation}\label{e:e7.10}
\sum_{i=1}^m X_i(|Xu|^{p-2}X_iu)\ =\ -\ V\ u^{p-1} \quad \quad \text{in} \quad \Om.
\end{equation}
\begin{enumerate}[a)]
\item If  $u\geq 0$ and
 $V\in L^t(\Om)$ for some  $t\ >\ \frac{Q}{p}$, then $u$ satisfies the Harnack inequality: for any Carnot-Carath\'edory (or gauge) ball $B_{R_0}(g_0)\subset\Om$ there exists a constant $C_0>0$ such that
\begin{equation}\label{e:harnack}
\esssup {B_R} \ u \leq C_0 \,\essinf {B_R} \ u,
\end{equation}
for any Carnot-Carath\'edory (or gauge) ball $B_{R}(g)$ such that $B_{4R}(g)\subset B_{R_0}(g_0)$.
\item If  $u\in \domoOm$ is a weak solution to \eqref{e:e7.10} and $V\in L^t(\Om)\cap L^{Q/p}(\Om)$, then $u\in \Gamma_\alpha (\Om) $ for some $0<\alpha<1$.
\item If $u\in\DoOm$ is a non-negative  solution of the Yamabe equation on the domain $\Om$,
\begin{equation}\label{e:Yamabe in sec regul}
\lap u \ = \ - u^{2^*-1},
\end{equation}
then either $u>0$ and $u\in \mathcal{C}^\infty\, (\Om)$ or $u\equiv 0$.
\end{enumerate}
\end{thrm}
The H\"older regularity of weak solutions of equation \eqref{e:e7.10} follows from a suitable adaptation of the classical De Giorgi-Nash-Moser result. The higher regularity when $p=2$ follows by an iteration argument based on  sub-elliptic regularity. A detailed proof of Theorem \ref{t:harnack and holder regul} can be found in \cite[Theorem 1.6.9]{IV2}. It is simply a combination of the fundamental  Harnack's inequality of \cite[Theorem  3.1]{CDG1} (valid for H\"ormander type operators), the boundedness of the weak solution  \cite[Theorem 4.1]{Va1}, the regularity of \cite[Theorem 3.35]{CDG1}, and the sub-elliptic regularity result concerning H\"ormander type operators acting on non-isotropic Sobolev or Lipschitz spaces of \cite{FS,F2}, see also \cite{F'77} for a general overview and further details. Note {that} these results together with the idea of \cite{FS} to "osculate" with the Heisenberg group carry over to obtain $\mathcal{C}^\infty$ regularity  in the CR and QC settings, see  \cite{JL1,JL2} and \cite{Wei} for details.

\subsection{Solution of the Yamabe type equation with partial symmetry}\label{s:Iwasawa partial symmetry}

By Theorem \ref{t:harnack and holder regul} any weak solution of the Yamabe equation is actually a smooth bounded function which is everywhere strictly positive, $u>0$ and $u\in \mathcal{C}^\infty\, (\Om)$.
The symmetries we are concerned are the following.
\begin{dfn}\label{D:symm}
 Let $\bG$ be a Carnot group of step two with Lie algebra $\algg = V_1 \oplus V_2$. We say that a function $U:\bG\to \mathbb R$ has \emph{partial  symmetry} \index{partial  symmetry} (with respect to a point $g_o\in \bG$) if there exists a function $u:[0,\infty)\times V_2 \to \mathbb R$ such that for every $g = \exp(x(g) + y(g)) \in \bG$ one has
$\tau_{g_o}\ U (g)=  u(|x(g)|, y(g)).$
 A function $U$ is said to have \emph{cylindrical symmetry} \index{cylindrical symmetry} (with respect to $g_o\in \bG$) if there exists $\phi:[0,\infty)\times [0,\infty) \to \mathbb R$ for which
$\tau_{g_o}\ U(g)=  \phi(|x(g)|, |y(g)|),$
for every $g\in \bG$.
\end{dfn}
 The proof of Theorem \ref{T:CS}  due to \cite{GV} consists of  two steps, first one shows that any entire solution with partial symmetry  is cylindrically symmetric and then  that all entire solutions with cylindrical symmetries.

 The proof of the first result relies {on} an adaption of   the  method of moving hyper-planes due to Alexandrov \cite{Al} and Serrin \cite{S4}. The moving plane technique was developed further in the two celebrated papers \cite{GNN}, \cite{GNN2} by Gidas, Ni and Nirenberg to obtain symmetry for semi-linear equations with critical growth in $\mathbb R^n$ or in a ball. {The proof of \cite{GV}} incorporate{s also} some important simplification of the proof in \cite{GNN2} due to Chen and Li \cite{CL}. We mention that a crucial role is played by the knowledge of the explicit solutions \eqref{e:FS extremal fns} and also by the inversion and the related Kelvin transform introduced by Kor\'anyi for the Heisenberg group \cite{Ko1}, and subsequently generalized to groups of Heisenberg type in \cite{CK}, \cite{CDKR}, see also \cite{GV2} for properties of the Kelvin transform.

 The proof of the second main result has been strongly influenced by the approach of Jerison and Lee for the Heisenberg group, see Theorem 7.8 in \cite{JL2}. After a change in the dependent variable, which relates the Yamabe equation to a new non-linear pde in a quadrant of the Poincar\'e half-plane, one is led to prove that the only positive solutions of the latter are quadratic polynomials of a certain type.

Besides {ideas from} Jerison and Lee's paper, the proof {of Theorem \ref{T:CS}  in \cite{GV}} has  some features of the method of the so-called $P$-\emph{functions} introduced by Weinberger in \cite{W}. Given a solution $u$ of a certain partial differential equation, such method is based on the construction of a suitable non-linear function of $u$ and $grad\ u$,  a $P$-function, which is itself solution (or sub-solution) to a related partial differential equation, and therefore satisfies a maximum principle. In fact, starting with a cylindrical solution $U$ of the Yamabe equation \ref{Yamabeomegasymm}, the function $\phi = v^{-4/(Q-2)}$ where $v=\left(\frac{Q-2}{4}\right)^{-(Q-2)/2}U$ satisfies
\begin{equation}\label{YPhi}
\mathcal L \phi= (\frac{Q-2}{4} + 1)\ \frac{|X\phi|^2}{\phi}+ \frac{Q-2}{4}.
\end{equation}
By the cylindrical symmetry assumption $\Phi$ is a function of the variables \begin{equation}\label{coord}
y= \frac{|\xi_1|^2}{4}, \quad\quad x= |xi_2|,
\end{equation}
which satisfies the equation
\begin{equation}\label{Yfinal}
\Delta \phi= \frac{n + 2}{2}\ \frac{|\nabla \phi|^2}{\phi}- \frac{a}{x}\ \phi_x- \frac{b}{y}\ \phi_y+ \frac{n}{2 y} ,
\end{equation}
in $\Om = \{(x,y)\in \mathbb R^2 \mid x>0, y>0 \}$ with
$a= k - 1\ \geq\ 0$, $b= \frac{m}{2}\ \geq\ 1$ and $ n= a + b\ \geq 1$.  The case $k=1$ corresponds to the Heisenberg group $\Hn$, and it was considered earlier in \cite{JL2}.
A long calculation shows that with $h = x^a y^b \phi^{-(n+1)}$, $$F=2<\nabla \phi, \nabla \phi_x>- 2\frac{n}{2b}\ \phi_{xy}-\phi_x\ \frac{|\nabla \phi|^2}{\phi}\quad \text{ and }\quad G=- 2<\nabla \phi, \nabla \phi_y>+ 2 \frac{n}{2b}\ \phi_{yy}+(\phi_y - \delta)\ \frac{|\nabla \phi|^2}{\phi},$$ the following identity holds true
\begin{multline*}
 (h F)_x- (h G)_y =\ h\ \bigg\{\bigg[ 2\ ||\nabla^2 \phi||^2- (\Delta \phi)^2 \bigg]
 +\ \frac{n+2}{n}\ \left(\Delta \phi- \frac{|\nabla \phi|^2}{\phi} \right)^2
+ \frac{2ab}{n}\ \bigg(\frac{\phi_x}{x} - (\frac{\phi_y}{y}- \frac{n}{2by}) \bigg)^2\bigg\}.
\end{multline*}
An integration over the first quadrant, noting that the integrals are finite as a consequence of the properties of the Kelvin transform on a group of Iwasawa type, we obtain
\begin{equation}\label{positivity}
2\ ||\nabla^2 \phi||^2= (\Delta \phi)^2, \quad \Delta \phi- \frac{|\nabla \phi|^2}{\phi}= 0, \quad \frac{\phi_x}{x}= \frac{\phi_y}{y}- \frac{n}{2by} .
\end{equation}
We remark that  the Kelvin transform allows us to find the asymptotic behavior of every  solution of the Yamabe equation, including all its derivatives. The behaviour at infinity of a finite energy solution can be found in more general settings with the method of \cite{LU}.
From the first two equations in \eqref{positivity}  we conclude  (see, e.g., \cite{W} or also \cite{JL2}) that $\phi$ must be of form
\begin{equation}\label{sym1}
\phi(x,y)= A^2\ (x^2+ y^2)+ 2 A \alpha x+ 2 B \beta y+ \alpha^2+ \beta^2
\end{equation}
for some numbers $A, B, \alpha$ and $\beta$, with $A^2 = B^2$. On the other hand, the third equation in \eqref{positivity} implies that $\alpha= 0$ and $ \beta= \frac{n}{4 b B}$.
Recalling that $x = |\xi_2|, y = |\xi_1|^2/4$ one easily concludes from the above that
\begin{equation}\label{JL}
\phi(|\xi_1|,|\xi_2|)= \frac{A^2}{16}\ \left[(\frac{a + b}{b A^2} + |\xi_1|^2)^2+ 16 |\xi_2|^2 \right]
\end{equation}
for some $A \not = 0$, hence
\begin{equation}
\phi(|\xi_1|,|\xi_2|)= \frac{Q-2}{16 m \epsilon^2}\ [(\epsilon^2 + |\xi_1|^2)^2+ 16 |\xi_2|^2]
\end{equation}
where $\epsilon^2 = \frac{Q-2}{m A}$. Finally,  the relation between $\Phi$ and $U$, we obtain
\[
U(g)= C_\epsilon\ ((\epsilon^2 + |x(g)|^2)^2+ 16 |y(g)|^2)^{-(Q-2)/4},
\]
with $C_\epsilon = [m(Q-2)\epsilon^2]^{(Q-2)/4}$. All other cylindrically symmetric solutions are obtained from this one by left-translation, which completes the proof of Theorem \ref{T:CS}.

We remark that \eqref{Yfinal} was used in \cite{V11}, see also \cite{MFS}, to establish the sharp constant and the extremals in a $L^2$ Hardy-Sobolev inequality involving distance to a lower dimensional subspace.

\subsection{The  best constant in the $L^2$ Folland-Stein inequality on the quaternionic Heisenberg groups}
In this section we explain the ideas behind the proof of Theorem \ref{T:Iwasawa groups Minimizers}.

 The proof relies on the  realization made in \cite{BFM} and used more recently in \cite{FrLi}  that the "center of mass" idea of Szeg\"o  \cite{Sz} and Hersch \cite{He}  can
be used to find the sharp form of (logarithmic) Hardy-Littlewood-Sobolev type inequalities on the Heisenberg group. This method does not give all solutions of the  Yamabe equation on the Iwasawa group, but is enough to determine the best  constant.

 The Cayley transform and the conformal nature of the problem are crucial for its solution. Another key is  Theorem~\ref{t:first eigenspace Iwasawa}  which will be used to see that the constants are the only minimizers on the sphere among all positive local minimizers which viewed as densities place the center of mass of the sphere at the origin. \emph{We shall focus here on the qc case \cite{IMV2} but the argument is valid in any of the groups of Iwasawa type using the just mentioned facts{, see also \cite{ChLZh1,ChLZh2}}.}

Let $\tileta$, cf. \eqref{e:stand cont form on S}, be the standard qc structure on the unit sphere $S^{4n+3}$. Szeg\"o and Hersch's center of mass method suggests  the following lemma.
\begin{lemma}\label{l:hersch}
For every $v\in L^1(S^{4n+3})$ with $\int_{S^{4n+3}} v \ Vol_{\tilde\eta}=1$ there is a quaternionic contact
conformal transformation $\psi:(S^{4n+3}, \tilde\eta)\rightarrow (S^{4n+3}, \tilde\eta)$ such that $$\int_{S^{4n+3}} \psi\, v \ Vol_{\tilde\eta} =0.$$
\end{lemma}
\begin{proof}
 Fix a point $P\in S^{4n+3}$ on the quaternionic sphere and denote by $N$  its antipodal point and consider
the local coordinate system near $P$ defined by the Cayley transform $\mathcal{C}_N$ from $N$, see \eqref{e:QC Cayley}. We know that
$\mathcal{C}_N$ is a quaternionic contact conformal transformation between $ S^{4n+3}\setminus {N}$  and the
quaternionic Heisenberg group, cf. \eqref{e:Cayley transf ctct form}. Notice that in this coordinate system $P$ is mapped to the identity of the group.
For every $r$, $0<r<1$, let $\psi_{r,P}$ be the qc conformal transformation of the sphere, which in the fixed
coordinate chart is given on the group by a dilation  with center the identity by a factor $\delta_{r}$. If we
select a coordinate system in $\mathbb{R}^{4n+4}=\Hn\times\mathbb{H}$ so that $P=(1,0)$ and $N=(-1,0)$. Applying  the Cayley transform  \eqref{e:QC Cayley} to
$(q^*,p^*)=\psi_{r,P}(q,p)$ we have
\begin{equation*}
\begin{aligned}
q^* & =2r\left (  1 + r^2(1+p)^{-1} (1-p) \right )^{-1} \left ( 1+p \right ) q\\
p^* & =\left ( 1 +r^2(1+p)^{-1}(1-p) \right )^{-1} \left (  1-r^2(1+p)^{-1}(1-p) \right ), i.e,
\end{aligned}
\end{equation*}
Consider the map $\Psi: B\rightarrow \bar B$, where $B$  ( $\bar B$ ) is the open (closed) unit ball
in $\mathbb{R}^{4n+4}$, by the formula $$\Psi(rP)=\int_{S^{4n+3}} \psi_{1-r,P}\, v\  Vol_{\tilde\eta}.$$ Notice
that $\Psi$ can be continuously extended to $\bar B$ since for any point $P$ on the sphere, where $r=1$, we have
$\psi_{1-r,P}(Q)\rightarrow P$ when $r\rightarrow 1$. In particular, $\Psi=id$ on $ S^{4n+3}$. Since the sphere
is not a homotopy retract of the closed ball it follows that there are $r$ and $P\in S^{4n+3}$ such that
$\Psi(rP)=0$, i.e., $\int_{S^{4n+3}} \psi_{1-r,P}\,v\  Vol_{\tilde\eta}=0$. Thus, $ \psi=\psi_{1-r,P}$ has the
required property.
\end{proof}

In the  next step one proves that there is a minimizer of the Folland-Stein inequality which satisfies the
zero center of mass condition.  A number of well known invariance properties of the Yamabe functional are
exploited. For the rest of the Section, given a qc form $\eta$ and a function $u$ we will denote by $\nabla^{\eta}u$ the horizontal gradient of $u$.

We shall call a (positive) function $u$ on the sphere a \textit{well centered} function when viewing $u^{2^*}$  as a density it places the center of mass of the sphere at the origin, i.e.,
\begin{equation}\label{e:zero mass}
\int_{S^{4n+3}} P\, u^{2^*}(P) \, Vol_{\tilde\eta}=0, \qquad P\in \mathbb{R}^{4n+4}=\Hn\times\mathbb{H}.
\end{equation}
For the next Lemma recall the functionals $\mathcal{E}_{\tileta}$ and $\mathcal{N}_{\tileta}$ introduced in \eqref{e:E and N  functional}.
\begin{lemma}\label{l:zero mass is enough}
Let $v$ be a smooth positive function on the sphere with $\mathcal{N}_{\tileta}(u) =1$. There is a well centered smooth positive function $u$ such that $\mathcal{E}_{\tileta}(u)=\mathcal{E}_{\tileta}(v)$
and $\mathcal{N}_{\tileta}(u) =1$.
In particular, the Yamabe constant \eqref{e:Yamabe constant Iwasawa}
is achieved for a positive function $u$ which is well centered, i.e., for a function $u$ satisfying \eqref{e:zero mass}.
\end{lemma}

\begin{proof}
Given a positive function
$v$ on the sphere $\int_{S^{4n+3}}  v^{2^*}\, Vol_{\tilde\eta}=1$, consider the function
\begin{equation}\label{e:zero mass transform}
u=\phi^{-1}(v\circ
\psi^{-1}),
 \end{equation}
 where $\psi$ is the qc conformal map of Lemma \ref{l:hersch}, $\eta\equiv(\psi^{-1})^*\tilde\eta$, and $\phi$ is the corresponding conformal factor of $\psi$. The claim of the Lemma follows directly from the conformal invariance  \eqref{e:conf Yamabe volume}.
\end{proof}
The next step shows that a well centered minimizer has to be constant.
\begin{lemma}\label{l:const mnimizer}
If $u$ is a well centered local minimum of the problem \eqref{e:Yamabe constant Iwasawa} for $M=(S^{4n+3},\tileta)$, then $u\equiv const$.
\end{lemma}

\begin{proof}
Let $\zeta$ be a smooth function on the sphere $S^{4n+3}$. Recalling \eqref{e:E and N  functional}, with the help of the divergence formula \eqref{div} we obtain the formula
\begin{equation}\label{e:Upsilon for zeta u}
\mathcal{E} (\zeta u) = \int_{S^{4n+3}}\zeta^2 \Bigl(4\frac {Q+2}{Q-2}\ \lvert \nabla^{\tileta} u \rvert^2 +
\tilde{S}\, u^2\Bigr)\tvol
 -  4\frac {Q+2}{Q-2}\int_{S^{4n+3}} u^2 \zeta\, {\tlap}
\zeta \tvol.
\end{equation}
At this point we let $\zeta$ be an eigenfunction corresponding to the first eigenvalue of the sub-Laplacian $\tlap$ associated to $\tileta$,  $\tlap \zeta =-\lambda_1 \zeta$.
Remarkably, the first eigenspace of the standard sub-Laplacian is spanned by restrictions to the sphere of the linear (coordinate functions) in $R^{4n+4}=\Hn\times\mathbb{H}$, see Theorem \ref{t:first eigenspace Iwasawa}.

Computing the second variation $\delta^2 \Upsilon(u)v = \frac {d^2}{dt^2} \Upsilon(u+tv)_{|_{t=0}}$ of $\Upsilon(u)$ we see that the local minimum condition $\delta^2 \Upsilon (u)v\geq 0$ implies
\begin{equation*}
\mathcal{E} ( v) -(2^*-1)\mathcal{E} ( u)\int_{S^{4n+3}}  u^{2^*-2}v^2\ Vol_{\tilde\eta}  \ \geq 0
\end{equation*}
for any function $v$ such that $\int_{S^{4n+3}}  u^{2^*-1}v\ Vol_{\tilde\eta}=0$.
Therefore, for $\zeta$ being any of the coordinate functions in $\Hn\times\mathbb{H}$ we have (taking $v=\zeta u$ and recalling that $u$ is well centered)
$$
\mathcal{E} (\zeta u) -(2^*-1)\mathcal{E} ( u)\int_{S^{4n+3}}  u^{2^*}\zeta^2\ Vol_{\tilde\eta}  \ \geq 0,
$$
which after a summation over all coordinate functions and a use of  \eqref{e:Upsilon for zeta u} gives
\begin{equation*}
\mathcal{E}(u)  - (2^*-1)\mathcal{E} ( u) +4\lambda_1(2^*-1)   \int_{S^{4n+3}} u^2 \ Vol_{\tilde\eta}\geq 0,
\end{equation*}
which implies, recall $2^*-1=( {Q+2})/({Q-2})$,
\begin{equation*}
0\leq 4(2^*-1)\left ( 2^*-2\right)\int_{S^{4n+3}} | \nabla^{\tilde\eta} u|^2\ Vol_{\tilde\eta}
\leq \ \left (4\lambda_1(2^*-1) -\left (2^*-2\right )\tilde S \right )  \int_{S^{4n+3}}  u^{2^*}\
Vol_{\tilde\eta}.
\end{equation*}
By Theorem a \ref{t:first eigenspace Iwasawa} we
 have actually equality $\lambda_1= {\tilde S}/{(Q+2)}$, hence $| \nabla^{\tilde\eta} u|=0$,
which completes the proof.
\end{proof}

After these preliminaries we turn to the proof of Theorem \ref{T:Iwasawa groups Minimizers}.
\begin{proof}[Proof of Theorem \ref{T:Iwasawa groups Minimizers}]
 Let $F$ be a minimizer (local minimum) of the Yamabe functional $\mathcal{E}$ on $\QH$ and $f$ the corresponding function on the sphere defined with the help of the Cayley transform by
 \begin{equation}\label{e:g def}
f=\mathcal{C}^*(F\Phi^{-1}),
\end{equation}
where $\Phi$ is a solution of the Yamabe equation on $\QH$  defined in \eqref{e:h and Phi}.
By the conformality of the qc structures on the group and the sphere we have by \eqref{e:vol conf change}
$Vol_{\Theta}=\Phi^{2^*}  Vol_{\tilde\Theta}$,
hence $F^{2^*} Vol_{\tilde\Theta} = f^{2^*}
\phi^{-2^*}  Vol_{\tilde\eta}$, where $\phi=\mathcal{C}^*(\Phi)$. This, together with the Yamabe equation implies that the Yamabe integral is preserved
\begin{equation}\label{e:invariance of dirichlet}
\int_{\QH}\ a|\nabla^{\tilde\Theta} F|^2\ Vol_{\tilde\Theta}\ =\  \int_{S^{4n+3}}\ \left ( a|\nabla^{\tilde\eta} f|^2 +  {\tilde S} f^2  \right )\tvol,
\end{equation}
where  $a=4(Q+2)/(Q-2)$. By Lemma \ref{l:zero mass is enough} and
\eqref{e:zero mass transform} the function $f_0=\phi^{-1}(f\circ
\psi^{-1})$ will be well centered  minimizer (local
minimum) of the Yamabe functional $\Upsilon$ on $S^{4n+3}$. The
latter claim uses also the fact that the map $v\mapsto u$ of
equation \eqref{e:zero mass transform} is one-to-one and onto on
the space of smooth positive functions on the sphere. Now, from
Lemma  \ref{l:const mnimizer}  we conclude that $g_o=const$.
Looking back at the corresponding functions on the group we see
that
\begin{equation*}
F_0 =\gamma\,{\left [(1+|q'|^2)^2+ |\omega'|^2 \right ]}^{-(Q-2)/4}
\end{equation*}
for some $\gamma=const.>0.$ Furthermore, the  proof of Lemma
\ref{l:hersch} shows that $F_0$ is obtained from $F$ by a
translation \eqref{translation} and dilation \eqref{scaling}.

\begin{rmrk}
We remark that the above argument shows that any local minimum of the Yamabe
functional $\Upsilon$ on the sphere (or the Iwasawa group) has to be a
global one.
\end{rmrk}

The Yamabe constant of the sphere is  calculated immediately by
taking a constant function in the Yamabe functional and a use of  \eqref{e:S of standard qc sphere}.

The remaining part of the proof (the value of the best constant $S_2$) is quite straightforward. Since it involves mainly calculations depending on the chosen normalization of the contact form we refer to \cite[Section 6.7]{IV2} for the details.
This completes the proof of Theorem \ref{T:Iwasawa groups Minimizers}.
\end{proof}

\begin{rmrk}\label{e:constant for Iwasawa basis}
One should keep in mind that the the standard basis \eqref{qHh}
is not an orthonormal basis which turns  the group $\QH$ into a group of H-type,  cf. also
\eqref{e:H-type basis for Hn} and the paragraph above it. The two
constants differ by a multiple of 
$4^{-k}$ in the
general case of a group of Iwasawa type with center of dimension
$k$. For more details on the relation between the Haar measure and the volume form associated to the contact form, as well as the exact relation between the best constants computed with respect to different bases {see} \cite[p. 188--189]{IV2}.
\end{rmrk}

\section{Sub-Riemannian geometry as conformal infinities}

\subsection{Riemannian case}
Let $(N,h)$ be a Riemannian manifold with boundary $M=\partial N$ with defining function $r >0$ on interior of $N $ which vanishes of order one on $M$. Suppose that $r\cdot h$ extends continuously to $M$ thus defining a "conformal structure" on the boundary $M$. Fefferman \& Graham  \cite{FeGr85} reversed the construction {and} used "canonical asymptotically hyperbolic (AH) filling" metrics to obtain conformal invariants. This is of interest also because  of the AdS/CFT correspondence in physics relating gravitational theories on $N$ with conformal theories on $M$. More specifically, if one can associate to a conformal
class on $M$ a “canonical” AH filling, then the Riemannian invariants for the interior metric give conformal invariants of the boundary structure.

For a basic example, consider on the open unit ball $B$ in $\mathbb{R}^n$  the hyperbolic metric $$h=\frac {4}{\rho^2}g_{euc},  \qquad \rho=1-|x|^2.$$ The conformal infinity is the conformal class of $g_{euc}\vert_{\partial B}$ - the standard metric on the unit sphere. Graham \& Lee  \cite{GrL91} gave the first general examples of AH Einstein metrics. The idea has been very useful especially due to the {mentioned} relation with the AdS/CFT correspondence \cite{Mal98} in physics.

\subsection{Conformal Infinities and Iwasawa Sub-Riemannian geometries} \label{s:Iwasawa sub-Riem geom}
The main references here are \cite{Biq1,Biq2} where the sub-Riemannian structures and geometries on the spheres at infinity of the hyperbolic spaces were used as model spaces for a wide class of sub-Riemannian structures which we shall call Iwasawa Sub-Riemannian geometries.
As a motivation we start with a few examples based on the real, complex and quaternion hyperbolic cases which. An explicit  description of the octonian hyperbolic plane and the ball model can be found in \cite{M} while \cite{Biq1} is the reference for the corresponding conformal infinity.

 On the open unit ball $B$ in $\mathbb{C}^{n+1}$ consider the Bergman metric $$h\ =\ \frac {1}{\rho}g_{euc}\ +\ \frac{1-\rho}{\rho^2}\left ( (d\rho)^2+(Id\rho)^2\right ),  \qquad \rho\ =\ 1-|x|^2.$$
    Notice that as $\rho\rightarrow 0$ we have that $\rho\cdot h$ is finite only on $H$-the so called \emph{horizontal space}, $H\ = \ Ker\, (I\,d\rho),$ which is the kernel of the contact form $\theta\ =\ I \,d\rho.$ The conformal infinity of $\rho\cdot h$ is the conformal class of a pseudohermitian CR structure defined by $H$ and $\theta$.
    If we look for K\"ahler-Einstein deformations Cheng \& Yau \cite{ChYa80} showed that  any smooth (in fact $C^2$) strictly pseudoconvex domain in $\mathbb{C}^{n+1}$ admits a unique complete K\"ahler-Einstein metric of Ricci curvature $-1$ which is asymptotic to the CR-structure of the boundary, see also \cite{MoYau83} for an extension to an arbitrary bounded domain of holomorhy.

 In the quaternion case, consider the open unit ball $B$ in $\mathbb{H}^{n+1}$ consider the hyperbolic metric $h=\frac {1}{\rho}g_{euc}+\frac{1}{4\rho^2}\left ( (d\rho)^2+(I_1d\rho)^2+(I_2d\rho)^2+(I_3d\rho)^2\right ).$ The conformal infinity is the conformal class of a quaternionic contact structure. In fact, $\rho h$ defines a conformal class of degenerate metrics with kernel $$H=\cap_{j=1}^3 Ker\, (I_j\,d\rho).$$ Biquard showed that the infinite dimensional family \cite{LeB91} of complete quaternionic-K\"ahler deformations of the quaternion hyperbolic metric have conformal infinities which provide an infinite dimensional family of examples of qc structures. Conversely, according to \cite{Biq1} every real analytic qc structure on a manifold $M$ of dimension at least eleven is the conformal infinity of a unique quaternionic-K\"ahler metric defined in a neighborhood of $M$.

 Finally,  \cite{Biq1} considered CR and qc structures as boundaries of infinity of Einstein metrics rather than only as boundaries at infinity of  K\"ahler-Einstein and quaternionic-K\"ahler metrics, respectively.  In fact, \cite{Biq1} showed that in each of the three  cases (complex, quaternionic, octoninoic)  any small perturbation of the standard Carnot-Carath\'eodory structure on the boundary is the conformal infinity of an essentially unique Einstein metric on the unit ball, which is asymptotically symmetric.  Various explicit examples of qc structures were constructed in \cite{AFISUV1}.

  In the above examples the geometry at the conformal infinity is asymptotic to  the hyperbolic geometry of the corresponding symmetric space of noncompact type of real rank one $G/K$, see paragraphs after Theorem \ref{t:Iwasawa is J2}. The corresponding geometries at infinity  are conformal metrics, CR structures, quaternionic-contact structures or octonionic contact structures as we define below following \cite{Biq1,Biq2}.  The symmetric case {belongs to the} “parabolic geometries” modelled on $G/P$, where $P$ is a minimal parabolic subgroup of $G$, see \cite{CS09}. We mention another  class of asymptotic geometries  considered in \cite{ABiq10} which are no {longer} asymptotic to a symmetric space, but  the model at infinity is  a homogeneous Einstein space, which m\emph{ay vary from point to point} on the boundary at infinity. Such a construction is motivated by Heber \cite{He98} who showed that
every deformation of the solvable group $S = NA$ carries a unique homogeneous Einstein metric. Thus, deformations of the nilpotent
Lie algebra $\algn$ give a homogeneous Einstein metric on the corresponding
solvable group $S$.

  Leaving the real case aside, we turn to the precise definition of the general (sub-Riemannian) geometric setting of the above constructions. Let $G$ be one of the groups $ U(n)$, $Sp(n)Sp(1)$ or $Spin(7)$, corresponding to the complex, quaternionic or octonionic cases, respectively, recalling the homogeneous models of the corresponding boundary spheres of the hyperbolic space, namely,  $ S^{2n+1}=U(n+1)/U(n)$, $S^{4n+3}= Sp(n+1)Sp(1)/Sp(n)Sp(1)$ or $S^7=Spin(9)/Spin(7)$. Let $M$ be a manifold with a 1-form $\eta$ with values in $\R,\R^3$ or $\R^7$, respectively, whose kernel  $H = Ker\, \eta$  - the so called \emph{horizontal distribution} - is of co-dimension $k=1,\ 3$ or $7$, respectively. Following Biquard \cite{Biq1}, a Carnot-Carath\'{e}odory metric (positive definite symmetric two tensor) compatible with $d\eta$ is
defined to be a metric $g$ on $H$ such that:
\begin{enumerate}[i)]
\item in the complex case, the restriction $\omega=\frac 12 d\eta\vert_H$  is a symplectic form on  $H$
compatible with $g$, i.e., $\omega(\cdot,\cdot) = g(I\cdot, \cdot)$ where $I$ is an almost
complex structure on $H$;
\item in the quaternionic case, the three 2-forms $\omega_i=\frac 12 d\eta_i\vert_H$, $i=1,2,3$, on $H$ are the fundamental forms of
a quaternionic structure compatible with $g$, i.e.,  $\omega_i(\cdot,\cdot) =
g(I_i\cdot, \cdot)$ for almost complex structures $I_i$ satisfying the quaternionic
commutation relations;
\item in the octonionic case, the seven 2-forms $\omega_i=\frac 12 d\eta_i\vert_H$, $i=1,\dots,7$ on $H$ provide
a $Spin(7)$ structure compatible with $g$, i.e., $\omega_i(\cdot,\cdot) =
g(I_i\cdot, \cdot)$ for almost complex structures $I_i$ satisfying the octonionic commutation
relations.
\end{enumerate}
We shall call the geometric structures above  \emph{Iwasawa sub-Riemannian geometries} since at every point the osculating nilpotent group \cite{FS}, \cite{RS76} is isomorphic to the corresponding (non-degenerate) Iwasawa group.
For simplicity, the above definition of CR, qc and octonionic contact structures requires the existence of a global 1-form defining $H$. The obstructions to global existence of such a form  in the CR and qc cases are the first Stiefel-Whitney class and the first Pontryagin class of $M$, respectively.

The complex case defines a strictly pseudoconvex almost CR manifold with a fixed pseudo-Hermitian structure, which is a CR structure when the integrability condition $[IX,IY]-[X,Y]\in H$ for $X,Y\in H$ holds. In the quaternionic and octonionic cases, a distribution $H$  for which a Carnot-Carath\'{e}odory $H$-metric exists will be called a quaternionic contact structure  or an octonionic contact structure.  The focus here will be mainly the CR and qc cases.  Note that the topological dimensions of these manifolds are $2n+1$ and $4n+3$, respectively. The so called \emph{homogeneous dimension} of $M$ is $Q=m+ 2k$ where $m=\text{dim } H$ and $k=\text{codim } H$. {W}e shall denote with $\vol$ the volume form
$$\vol=\eta\wedge(d\eta)^n\ \quad \text{and }\quad \vol=\eta_1\wedge\eta_2\wedge\eta_3\wedge\Omega^n,$$
  in the CR ($n=m/2$) and qc ($n=m/4$) case respectively, where $\Omega=\omega_1\wedge\omega_1+\omega_2\wedge\omega_2+\omega_3\wedge\omega_3$ is the \emph{fundamental 4-form}. There is a Riemannian metric on $M$ obtained by extending in a natural way the horizontal metric $g$ to a true Riemannian metric, denoted by $h$, explicitly given by
\begin{equation}\label{hmetric}
h=g+\sum_{i=1}^k(\eta_i)^2.
\end{equation}
The Riemannian volume form  is up to a constant multiple the just defined volume form $\vol$.

For each of the considered geometries there is a canonically defined connection $\nabla=\nabla^{\eta}$ with torsion $T$.  In the integrable CR case this is the Tanaka-Webster connection, see \cite{Ta62}and \cite{We2}. In the qc and octonionic cases this is the Biquard connection, see \cite{Biq1,Biq2} and \cite{D}.  The curvature tensor of the corresponding canonical connection $\nabla$ and the associated   (0,4)  tensor, which is denoted with the same letter, are
\begin{equation}\label{e:curv tensor}
{R}(A,B)C=[\nabla_A, \nabla_B]C- \nabla_{[A,B]}C, \qquad R(A,B,C,D)\overset{def}{=}h(R(A,B)C,D).
\end{equation}
Let $\{e_1,\dots,e_{m}\}$, $m=\dim H$, be a local orthonormal basis of the \emph{horizontal space} $H$, $g(e_a,e_b)=\delta_{ab}$.  The Ricci type and scalar curvature tensors
are obtained by taking \emph{horizontal traces}
\begin{equation}  \label{qscs}
Ric(A,B)=\sum_{b=1}^m R(e_b,A,B,e_b),\quad S=\sum_{a,b=1}^m R(e_b,e_a,e_a,e_b), \ A,B\in T(M),
 \end{equation}
which manifests the sub-Riemannian nature of these tensors. In the qc case these tensors are called \emph{qc-Ricci tensor} and \emph{qc-scalar curvature}  tensor of the Biquard connection.

The
(horizontal) divergence of a horizontal vector field/one-form $%
\sigma\in\Lambda^1\, (H)$ defined by $\nabla^*\, \sigma\
=tr^g\nabla\sigma= \nabla \sigma(e_a,e_a)$ supplies the "integration
by parts" over compact $M$ formula, \cite{Ta62}, \cite{IMV}, see also \cite{Wei},
\begin{equation}  \label{div}
\int_M (\nabla^*\sigma)\vol\ =\ 0.
\end{equation}

\subsubsection{The first eigenvalue on the sphere}\label{ss:eigenspace Iwasawa}
\begin{thrm}\label{t:first eigenspace Iwasawa}
The eigenspaces of the first eigenvalue of the sub-Laplacian of the canonical Iwasawa sub-Riemannian structures on the  spheres at infinity of the hyperbolic spaces are the restrictions of all real-linear functions in the corresponding Eucliden space to the sphere.
\end{thrm}
The exact value of the eigenvalue depends on the normalization of the "standard" form $\eta$, which will be made explicit  later. Of course, in the real case the eigenspace is  the space of spherical harmonics of order one. Various proofs of Theorem \ref{t:first eigenspace Iwasawa} are possible. The  simplest proof is a direct computation based on the explicit definition of the corresponding sub-Laplacians, see \cite{FrLi,BFM} for the complex case, \cite[Lemma 2.3]{IMV2} for the quaternion, and \cite{ChLZh1} for the octonian cases. Alternatively, one can relate the sub-Laplacian to the corrsponding Laplace-Beltrami operator on the sphere, see \eqref{obsa} and \eqref{llex} for the complex and quaternion case. Finally, the result follows from an abstract approach as in \cite{ACD} where the corresponding "spherical harmonics" are studied.

\subsubsection{The Yamabe problem on Iwasawa sub-Riemannian manifolds}\label{ss:Yamabe Iwasawa mnfld}
\begin{dfn}\label{d:sub-Riemannian conf transf}
  The "conformal" class of $[\eta]$  consists of all 1-forms $\bar\eta=\phi^{4/(Q-2)} \Psi\eta$ for a smooth positive function $\phi$ and $\Psi\in SO(k)$ with smooth functions as entries.
  \end{dfn}
  In the CR case $\Psi\equiv 1$, while in the QC case $\Psi$ is an $SO(3)$ matrix with entries smooth functions.  We note that the canonical connection is independent of $\Psi$, but depends on $\phi$, which brings us to the Yamabe type problems.
The Yamabe functional is
\begin{equation}\label{e:Yamabe functional def}
\Upsilon_{[\eta]} (\phi) = \left ( \int_M\Bigl(4\frac {Q+2}{Q-2} \lvert \nabla \phi \rvert^2 + S\, \phi^2\Bigr)
Vol_\eta \right ) \Big / \left ( \int_M \phi^{2^*}\ Vol_\eta \right )^{2/2^*},
\end{equation}
where $2^*=2Q/(Q-2)$. $\nabla=\nabla^\eta$ is the  connection  of $\eta$, $S$ is the scalar curvature \eqref{qscs} of $(M,\, \eta)$ and $|\nabla \phi|=\left (\sum_{a=1}^m (d\phi(e_a))^2 \right)^{1/2}$ is the length of the horizontal gradient.
It will be useful to introduce the functionals
\begin{equation}\label{e:E and N  functional}
\mathcal{E}_{\eta}(\phi)\overset{def}{=}\int_{M} \Bigl(4\frac {Q+2}{Q-2}\ \lvert \nabla^{\eta} \phi\rvert^2\ +\
{S}\, \phi^2\Bigr)\vol,\hskip.4in
\mathcal{N}_{\eta}(\phi)=\left ( \int_{M}  \phi^{2^*}\vol\right)^{2/2^*},
  \end{equation}
 hence the Yamabe functional  can be written as $\Upsilon(\phi){=} \mathcal{E}(\phi) / N(\phi)$ (dropping the subscript $\eta$ when there is no confusion). The \emph{Yamabe constant} is
\begin{equation}\label{e:Yamabe constant Iwasawa}
 \Upsilon(M,[\eta]) = \inf \{ \Upsilon_{[\eta]}(\phi):\, \phi\in \mathcal{C}^\infty (M) \}= \inf \{\mathcal{E}_{\eta}(\phi):\, \mathcal{N}_{\eta}(\phi)=1,\ \phi\in \mathcal{C}^\infty (M) \}.
\end{equation}
In the above notation we tacitly introduced $[\eta]$ as a subscript. The reason for this notation is that the Yamabe functional is conformally invariant, which follows from the formulas relating the (sub-Riemannian) scalar curvatures of the associated to $\eta$ and $\bar\eta$ connections, see below, together with the formula for the change of the volume form, \begin{equation}\label{e:vol conf change}
\vol=\phi^{2^*}\tvol.
\end{equation}
Finding the Yamabe constant in the case of the standard  Iwasawa sub-Riemannian structures on the unit spheres is equivalent to the problem of determining the best constant in the $L^2$ Folland \& Stein \cite{FS} Sobolev type embedding inequality on the corresponding Heisenberg group. As noted earlier the best constant in the $L^2$ Folland \& Stein inequality together with the minimizers were determined recently in \cite{FrLi,IMV2,ChLZh1} by the method of Frank \& Lieb \cite{FrLi}, see also \cite{BFM}. Nevertheless this simpler approach does not yield the conjectured uniqueness (up to an automorphism) in the case of the spheres.

Finding the  {Yamabe constant } is closely related to the \emph{Yamabe problem} which   seeks  all  Iwasawa sub-Riemannian structures of  constant scalar curvature conformal to a given structure $\eta$. In fact, taking the conformal factor in the form $\bar\eta=\phi^{4/(Q-2)}%
\eta$ as we did above, a calculation (done separately for each of the cases) gives the following \emph{Yamabe equation},
\begin{equation}\label{e:conf Yamabe}
\mathcal{L} \phi\overset{def}{=} 4\frac {Q+2}{Q-2}\ \triangle \phi -\ S\, \phi \ =\ - \  \overline{S}\,\phi^{2^*-1},
\end{equation}
where $\triangle $ is the horizontal sub-Laplacian, $\triangle \phi\ =\ tr^g_H\,(\nabla du)$, $S$ and
$\overline{S}$ are the  scalar curvatures corresponding to the associated to  $\eta$ and $\bar\eta$ canonical connections.

A natural question is to find
all solutions of the  Yamabe equation \eqref{e:conf Yamabe}. As usual the two fundamental problems are
related by noting that on a compact manifold $M$ with a fixed conformal class $[\eta]$ the Yamabe equation characterizes the non-negative extremals of the Yamabe functional. The operator $\mathcal{L} \phi$ in \eqref{e:conf Yamabe} is the so called conformal sub-Laplacian. Using the divergence formula \eqref{div} we can write equation \eqref{e:conf Yamabe} in the form
\begin{equation}\label{e:conf Yamabe volume}
\phi^{-1}v\,
\mathcal{\bar L}(\phi^{-1}v)\ Vol_{\bar\eta}=v\mathcal{ L}(v)\ Vol_{\eta},
\end{equation}
for any $v\in C^\infty(M)$, which makes explicit the conformal invariance. Here $\mathcal{\bar L}$ denotes the conformal sub-Laplacian associated to the canonical connection $\overline \nabla$ of $\bar\eta$.

\subsection{CR Manifolds}\label{ss:CR geometry}

A CR manifold is a smooth manifold $M$ of real dimension 2n+1, with a fixed
n-dimensional complex sub-bundle $\mathcal{H}$ of the complexified tangent
bundle $\mathbb{C}TM$ satisfying $\mathcal{H} \cap \overline{\mathcal{H}}=0$
and $[ \mathcal{H},\mathcal{H}]\subset \mathcal{H}$. If we let $H=Re\,
\mathcal{H}\oplus\overline{\mathcal{H}}$, the real sub-bundle $H$ is
equipped with a formally integrable almost complex structure $J$. We assume
that $M$ is oriented and there exists a globally defined compatible contact
form $\eta$ such that the \emph{horizontal space} is given by ${H}=Ker\,\eta.$ In other words, the hermitian
bilinear form
$g(X,Y)= 1/2 \,d\eta(JX,Y)$
is non-degenerate. The CR structure is called strictly pseudoconvex if $g$
is a positive definite tensor on $H$. For brevity we shall frequently use the term \emph{CR manifold} to refer to a strictly pseudoconvex pseudohermitian   manifold. In other words, unless specified otherwise a CR manifold will be an \emph{integrable strictly pseudoconvex CR manifold with a fixed pseudohermitian stricture}.

The almost complex structure $J$ is formally integrable in the
sense that
$([JX,Y]+[X,JY])\in {H}$
and the Nijenhuis tensor vanishes
$N^J(X,Y)=[JX,JY]-[X,Y]-J[JX,Y]-J[X,JY]=0.$
A CR manifold $(M,\eta,g)$ with a fixed compatible contact form $\theta$
is called \emph{a pseudohermitian manifold}. In this case the 2-form
$d\eta_{|_{{H}}}\overset{def}{=}2\omega$
is called the fundamental form. The contact form whose kernel is the horizontal space $H$ is determined up
to a conformal factor, i.e., $\bar\theta=\nu\theta$ for a positive smooth
function $\nu$, defines another pseudohermitian structure called
pseudo-conformal to the original one.

A Riemannian metric is defined in the usual way, written with a slight imprecision as $h=g +\eta^2$. The vector field $\xi$ dual to $\eta$
with respect to $g $ satisfying $\xi\lrcorner d\eta=0$ is called the Reeb
vector field.
\subsubsection{Invariant decompositions}

As usual any endomorphism $\Psi$ of $H$ can be decomposed with respect to
the complex structure $J$ uniquely into its $U(n)$-invariant $(2,0)+(0,2)$
and $(1,1)$ parts. In short we will denote these components correspondingly
by $\Psi_{[-1]}$ and $\Psi_{[1]}$. Furthermore, we shall use the same
notation for the corresponding two tensor, $\Psi(X,Y)=g(\Psi X,Y)$.
Explicitly, $\Psi=\Psi=\Psi_{[1]}+\Psi_{[-1]}$, where
\begin{equation}{\label{compc}}
\Psi_{[1]}(X,Y)=%
\frac {1}{2}\left [ \Psi(X,Y)+\Psi(JX,JY)\right ],\qquad \Psi_{[-1]}(X,Y)=%
\frac {1}{2}\left [ \Psi(X,Y)-\Psi(JX,JY)\right ].
\end{equation}
The above notation is justified by the fact that the $(2,0)+(0,2)$ and $%
(1,1) $ components are the projections on the eigenspaces of the operator
$\Upsilon =\ J\otimes J, \quad (\Upsilon \Psi) (X,Y)\overset{def}{=}\Psi
(JX,JY)$
corresponding, respectively, to the eigenvalues $-1$ and $1$. Note that both
the metric $g$ and the 2-form $\omega$ belong to the [1]-component, since $%
g(X,Y)=g (JX,JY)$ and $\omega(X,Y)=\omega (JX,JY)$. Furthermore, the two
components are orthogonal to each other with respect to $g$.

The Tanaka-Webster connection \cite{Ta62,We2,W} is the unique linear connection $%
\nabla$ with torsion $T$ preserving a given pseudohermitian
structure, i.e., it has the properties that the almost complex structure $J$ and the contact form are parallel, $\nabla\xi=\nabla J=\nabla\eta=\nabla g=0$,  and the torsion tensor is of pure type, i.e., for $X,\, Y\in H$ we have
\begin{gather} \notag
 T(X,Y)=d\eta(X,Y)\xi=2\omega(X,Y)\xi,\\ \label{torha}
 \quad T(\xi,X)\in {H},
  g(T(\xi,X),Y)=g(T(\xi,Y),X)=-g(T(\xi,JX),JY).
\end{gather}

The (Webster) torsion $A$ of the pseudohermitian manifold is the symmetric tensor $A\overset{def}{=}T(\xi,.):H\rightarrow H$. Clearly, equation \eqref{torha} shows that $A\in \Psi_{[-1]}$.
It is also well known \cite{Ta62} that $A$ is the obstruction for a pseudohermitian  manifold
to be Sasakian. We recall that  a contact manifold $(M,\eta)$ is Sasakian if its Riemannian cone $C=M\times \mathbb R^+$ with metric $t^2h+ dt^2$ is K\"ahler (see e.g. \cite{Bl75,BGN}).

\subsubsection{Curvature tensors of the Tanaka-Webster connection}
 The curvature tensors are defined in a standard fashion using \eqref{e:curv tensor} and \eqref{qscs}, noting again that traces are taken only on the horizontal space. {The  Ricci 2-form  is  defined by}

$
\rho (A,B)=\frac 12\,R(A,B,e_a,Je_a).
$
{The horizontal part of the Ricci 2-form is (1,1)-form with respect to $J$ and  the first Bianchi identity implies}
\begin{equation}\label{e:CR Ricci 2-from}
\rho (X,JY)=\frac 12\,R(e_a,Je_a,X,JY).
\end{equation}
{The tensor $\rho(X,JY)\in \Psi_{[1]}$ is a symmetric tensor and is frequently also called  the Webster Ricci tensor.}
The CR Ricci tensor has the following type decomposition in $J$ invariant and skew-invariant forms \cite{Ta62}, \cite{We2}, see also \cite{DT} and \cite[Chapter 7]{IV2}
\begin{equation}\label{e:CR Ric type decomp}
Ric(X,Y)=\rho(JX,Y)+2(n-1)A(JX,Y).
\end{equation}

It is well known that a pseudohermitian manifold with a flat
Tanaka-Webster connection is locally isomorphic to the (complex) Heisenberg
group. For $n>1$ the vanishing of the horizontal part of the Tanaka-Webster connection implies the vanishing of the whole curvature. If $n=1$, in addition to the vanishing of the horizontal part of the curvature one needs also the vanishing of the pseudohermitian torsion to have zero curvature.

\subsubsection{The Heisenberg group}\label{ss:Heis group}
Given the ubiquitous role of the Heisenberg group $\QC\equiv\Cn\times \mathbb R$ in analysis and being the flat model of the considered CR geometries {(}the Tanaka-Webster connection coincides with the invariant flat connection on the group{) w}e shall write  explicitly a number of formulas in this special setting. {Some of these formulas} will  be made explicit in Section \ref{ss:qc geometry} for quaternionic contact structures in which case the quaternionic Heisenberg group will play the role of the flat model space. The group $\QC$ arises as the nilpotent part in the Iwasawa decomposition of the complex hyperbolic space. Thus, $\QC$ \index{Heisenberg group $\QC$}is a Lie group whose underlying manifold is $\mathbb{C}^n \times \R$ with group law given by \eqref{e:H-type Iwasawa groups}
where for $z, z' \in \mathbb{C}^n$ we  let $z\cdot z' = \sum_{j=1}^n z_j {z}'_j$. A (real) basis for the Lie algebra of left-invariant vector fields on $\QC$ is given by
\begin{equation}\label{e:basis for Hn}
X_j = \frac{\partial}{\partial x_j} +  2 y_j \frac{\partial}{\partial t}, \quad X_{n+j}\equiv Y_j =  \frac{\partial}{\partial y_j} -  2 x_j \frac{\partial}{\partial t},\quad \xi=2\frac{\partial}{\partial t}, \quad j=1,...,n,
\end{equation}
with corresponding contact form
\[\tilde\Theta=\frac 12dt+\sum_{j=1}^n(x_jdy_j-y_jdx_j)=\frac 12dt+\sum_{j=1}^n \Im(\bar z_jdz_j)\]
Here, we have identified $z=x+iy \in \mathbb{C}^n$, with the real vector $(x,y)\in \mathbb{R}^{2n}$. Since $[X_j, Y_k] = - 4 \delta_{jk} \frac{\partial}{\partial t}$, the Lie algebra is generated by the system $X=\{X_1,...,X_{2n}\}$. The sub-Laplacian is $\mathcal{L} = \sum_{j=1}^{2n} X_j^2$ which is the real part of the Kohn complex Laplacian. In this case the exponential map is the identity and, as for any group of step two, we have the the parabolic dilations
$\delta_\lambda(z,t)=(\lambda z, \lambda^2 t)$. The corresponding homogeneous dimension is $Q= 2n + 2$.

In regards to the theory of groups of Heisenberg type, cf. Section \ref{s:FS inequality}, some care has to be taken when defining the scalar product which turns $\QC$ into a group of Heisenberg type. For example, the standard inner product of $\mathbb{C}^n \times \R$, i.e., the inner product in which the basis of left invariant vector fields given in \eqref{e:basis for Hn} is an orthonormal basis will not make the Heisenberg group $\QC$ a group of Heisenberg type.  An orthonormal basis of an H-type compatible metric is given by, $j=1,...,n$,
\begin{equation}\label{e:H-type basis for Hn}
X_j = \frac{\partial}{\partial x_j} +  2 y_j \frac{\partial}{\partial t}, \ X_{n+j}\equiv Y_j =  \frac{\partial}{\partial y_j} -  2 x_j \frac{\partial}{\partial t},\ T=\frac {1}{4}\frac{\partial}{\partial t}
\end{equation}
and homogeneous gauge \eqref{Hgauge}  given by
$N(z,t)\ =\ (|z|^4 + 16t^2)^{1/4}$,
$|z|= \left (\sum_{j=1}^{n}(x_j^2+y_j^2)\right )^{1/2}$.

\subsection{The CR sphere and the Cayley transform}
The simplest CR manifolds are the  three hyperquadrics in complex space, \cite{ChM}
\begin{gather}
Q_{+}: \ r=|z_1|^2+\dots+|z_n|^2 + |w|^2=1,\qquad
 Q_{-}: \ \ r=|z_1|^2+\dots+|z_n|^2 - |w|^2=-1\\
  Q_{0}:\ \ r=|z_1|^2+\dots+|z_n|^2 -\Im(w)=0,
  \end{gather}
where $(z_1,z_2,\dots z_n,w)\in \mathbb{C}^{n}\times \mathbb{C}$ with corresponding contact forms $\tileta\overset{def}{=}\eta_{+}$, $\eta_{-}$ and $\tilde\Theta{=}\eta_0$ equal to $-i\partial r$ which define  strictly pseudoconvex pseudohermitian structures. Of course, these are the "standard" (up to a multiplicative factor depending on the reference) pseudohermitian  structures on the sphere $S^{2n+1}$, hyperboloid, and Heisenberg group $\Hn$, the latter identified with the boundary of the Siegel domain via the map $(z,t)\mapsto (z,t+i|z|^2)$.
A transformation mapping $Q_{-}$ onto $Q_{+}$ minus a curve is given by $w=1/w'$ and $z_j=z_j'/w'$. On the other hand, the transformation
$$\mathcal{C}(Z,W)=\Big(\frac {iZ}{1-W}, i\frac {1+W}{1-W} \Big), \qquad \text{with inverse}\qquad \mathcal{C}^{-1}(z,w)=\Big(\frac{2z}{ i+w},\frac{w-i}{w+i}\Big),$$ maps  the sphere $S^{2n+1}\setminus(0,0,...,0,1)$ onto $\Hn$. The map $\mathcal{C}$ is the Cayley transform (with a pole at $(0,0,...,0,1)$). These transformations clearly preserve the CR structure since they are restrictions of holomorphic maps, but do not preserve the contact forms and are in fact  pseudoconformal pseudohermitian maps
\[\tilde\Theta=\frac 12dw-i\bar z \cdot dz=\frac {1}{|1-W|^2}\tileta, \qquad \tileta=-i \left (\overline WdW+ \bar Z\cdot dZ \right).
\]

\subsubsection{CR conformal flatness}
A fundamental fact characterizing CR conformal flatness is the
 Cartan-Chern-Moser  theorem \cite{Car,ChM,W}. A proof
based on the classical approach used by H. Weyl in Riemannian
geometry (see e.g. \cite{Eis}) can be found in \cite{IV2}, see also \cite{IVZ}.
\begin{thrm}\protect[\cite{Car,ChM,W}]\label{crmain}
Let $(M,\theta,g)$ be a 2n+1-dimensional { non-degenerate}
pseudo-hermitian manifold.
If $n>1$ then $(M,\eta,g)$ is locally pseudoconformaly equivalent to
a hyperquadric in $\mathbb{C}^{n+1}$ if and only if the Chern-Moser tensor vanishes, $\mathcal{S}=0$, \cite{ChM,W}.
In the case  $n=1$,  $(M,\eta,g)$ is locally pseudoconformaly equivalent to a hyperquadric in $\mathbb{C}^{2}$ if and only the tensor $F^{car}$ given below vanishes, $F^{car}=0$, \cite{Car}.
\end{thrm}
Here, the Chern-Moser tensor $\mathcal{S}$ is a pseudoconformaly
invariant tensor, i.e., if $\phi$ is a
smooth positive function and  $\tileta=\phi\eta$, then
$\mathcal{S}_{\bar\eta}=\phi\mathcal{S}_{\eta}$. The Chern-Moser tensor  $\mathcal{S}$ \cite{ChM}  is
determined completely by the {(1,1)-part of the curvature and the} Ricci 2-form,
\begin{multline}\label{crpinv}
S(X,Y,Z,V)
=\frac12\Big[R(X,Y,Z,V)+R(JX,JY,Z,V)\Big]\\
-\frac{S}{4(n+1)(n+2)}\Bigl[g(X,Z)g(Y,V)-g(Y,Z)g(X,V)
+\Omega(X,Z)\Omega(Y,V)-\Omega(Y,Z)\Omega(X,V)+2\Omega(X,Y)\Omega_s(Z,V)
 \Bigr]\\ -\frac1{2(n+2)}\Big[g(X,Z)\rho(Y,JV)-g(Y,Z)\rho(X,JV)+g(Y,V)\rho(X,JZ)-g(X,V)\rho(Y,JZ)\Big]\\
 -\frac1{2(n+2)}\Big[\Omega(X,Z)\rho(Y,V)-\Omega(Y,Z)\rho(X,V)+\Omega(Y,V)\rho(X,Z)-\Omega(X,V)\rho(Y,Z)\Big]\\
 -\frac1{n+2}\Big[\Omega(X,Y)\rho(Z,V)+\Omega(Z,V)\rho(X,Y)\Big].
\end{multline}

For $n=1$ the tensor $S$ vanishes identically and the Cartan condition can be expressed by the vanishing of   the \emph{the Cartan tensor} $[-1]$ type tensor $F^{car}$  defined on ${H}$ by (\cite{IVZ,IV2}
\begin{multline}\label{crtreeF}
F^{car}(X,Y)=\Cr^2 S(X,JY)+\Cr^2S(Y,JX)+16(\Cr
^2_{Xe_a}A)(Y,e_a)
+16(\Cr^2_{Ye_a}A)(X,e_a)+36SA(X,Y)\\
+48(\Cr^2_{e_aJe_a}A)(X,JY)+3g(X,Y)\Cr^2S(e_a,Je_a).
\end{multline}

\subsection{Quaternionic Contact Structures}\label{ss:qc geometry}
Following Biquard \cite{Biq1, Biq2}, a 4n+3 dimensional manifold $M^{4n+3}$ is quaternionic contact (qc) if we have:
\begin{enumerate}[i)]
\item a co-dimension three distribution $H$, which locally is the intersction of the kernels of three 1-forms on $M$, $H=\bigcap_{s=1}^3 Ker\, \eta_s$, $\eta_s\in \Gamma(M:T^*M)$;
\item a 2-sphere bundle $\mathcal{Q}$ of "almost complex structures" locally generated by
$I_s\,:H \rightarrow H,\quad I_s^2\ =\ -1$, satisfying $I_1I_2=-I_2I_1=I_3$;
 \item a "metric" tensor $g$ on $H$, such that,  $g(I_sX,I_sY)\ =\ g(X,Y)$, and $2g(I_sX,Y) = d\eta_s(X,Y)$, for all $X,Y\in H$.
 \end{enumerate}
We let $\omega_s(X,Y)=g(I_sX,Y)$ be the associated  2-forms.
The "canonical" Biquard connection is the unique linear connection
defined by the following theorem \cite{Biq1}.
\begin{thrm}[ \cite{Biq1}]
If $(M,\eta)$ is a qc manifold  and $n>1$, then
there exists a unique complementary to $H$ subbundle $V$, $TM=H\oplus V$ and a linear connection $\nabla$  with the properties that
  $V$, $H$, $g$  and the 2-sphere bundle $\mathcal{Q}$ 
 are parallel and  the torsion $T$ of $\nabla$ satisfies:
       a) for $ X, Y\in H, \quad T(X,Y)=-[X,Y]|_V\in V$;\hskip.2truein
        b) for $ \xi\in V,\ X\in H$, $T_\xi(X)\equiv T(\xi,X)\in H$ and $ T_\xi\in (sp(n)+sp(1))^{\perp}$.
\end{thrm}

Biquard also showed that the  "vertical" space $V$ is generated by the Reeb vector fields
$\{\xi_1,\xi_2,\xi_3\}$
determined by
$
\eta_s(\xi_k)=\delta_{sk}, \quad (\xi_s\lrcorner d\eta_s)_{|H}=0,
\quad
(\xi_s\lrcorner d\eta_k)_{|H}=-(\xi_k\lrcorner d\eta_s)_{|H}.
$
If the dimension of $M$ is seven, $n=1$, the Reeb vector fields might not exist.
D. Duchemin \cite{D} showed  that if we assume their existence, then there is a connection as before. Henceforth, by a qc structure in dimension $7$ we 
mean a qc structure satisfying the Reeb conditions.

Note that the extended Riemannian metric  $h$ given by \eqref{hmetric} as well as the Biquard connection do not depend on the action of $SO(3)$ on $V$, but both change  if $\eta$ is multiplied by a conformal factor.

\subsubsection{Invariant decompositions}

As usual any endomorphism $\Psi$ of $H$ can be decomposed with respect to
the hypercomplex complex structure $I_s$, $s=1,2,3$ uniquely into its two $Sp(n)Sp(1)$-invariant  parts. In short we will denote these components correspondingly
by $\Psi_{[-1]}$ and $\Psi_{[3]}$. Furthermore, we shall use the same
notation for the corresponding two tensor, $\Psi(X,Y)=g(\Psi X,Y)$.
Explicitly, $\Psi=\Psi=\Psi_{[3]}+\Psi_{[-1]}$, where
\begin{equation}{\label{comp}}
\begin{aligned}
\Psi_{[3]}(X,Y)=%
\frac {1}{4}\left [ \Psi(X,Y)+\Psi(I_1X,I_1Y)+\Psi(I_2X,I_2Y)+\Psi(I_3X,I_3Y)\right ],\\ \Psi_{[-1]}(X,Y)=%
\frac {1}{4}\left [3 \Psi(X,Y)-\Psi(I_1X,I_1Y)-\Psi(I_2X,I_2Y)-\Psi(I_3X,I_3Y)\right ].
\end{aligned}
\end{equation}
The above notation is justified by the fact that the $[3]$ and $%
[-1]$ components are the projections on the eigenspaces of the operator
$\Upsilon =\ I_1\otimes I_1+I_2\otimes I_2+I_3\otimes I_3, \quad (\Upsilon \Psi) (X,Y)\overset{def}{=}\Psi(I_1X,I_1Y)+\Psi(I_2X,I_2Y)+\Psi(I_3X,I_3Y)$
corresponding, respectively, to the eigenvalues $3$ and $-1$. Note that
the metric $g$  belong{s} to the [3]-component, {while}  the 2-forms $\omega_s, s=1,2,3$ belong to the [-1]-component since. Furthermore, the two
components are orthogonal to each other with respect to $g$. For $n=1$ the [3]-component is 1-dimensional, $\Psi_{[3]}=\frac{tr\Psi}4g$.

\subsubsection{Curvature of a Quaternionic Contact Structure}
The curvature tensors are defined in a standard fashion using \eqref{e:curv tensor} and \eqref{qscs}.
In addition we define the qc Ricci 2-forms
\begin{equation}\label{neww}
{\rho_s(A,B)=\frac{1}{4n}R(A,B,e_a,I_se_a)}, \quad s=1,2,3.
\end{equation}
Biquard \cite{Biq1}  showed that the map $T_{\xi _j}=T^0_{\xi_j}\ +\ I_j U$,  where $T^0_{\xi_j}$ is symmetric while $I_j U$ is a skew-symmetric map and $U\in \Psi_{[3]}$,  $I_sU=UI_s$, $s=1,2,3$. Further properties were found in \cite{IMV}. A remarkable fact, \cite[Theorem 3.12]{IMV}, is that (unlike the CR case!)
the torsion endomorphism determines the (horizontal) qc-Ricci tensor
and the (horizontal) qc-Ricci forms of the Biquard connection. For this we also need  the torsion type tensor $T^0\overset{def}{=}  T^0_{\xi _1}\, I_1+T^0_{\xi _2}\, I_2+T^0_{\xi _3}\, I_3\in \Psi_{[-1]}$ introduced in \cite{IMV}.
\begin{thrm}[\cite{IMV}] On a QC manifold $(M,\eta)$ we have
\begin{equation}\label{sixtyfour}
\begin{aligned}
& Ric(X,Y) \ =\ (2n+2)T^0(X,Y) +(4n+10)U(X,Y)+\frac{S}{4n}g(X,Y)\\
 &\rho_s(X,I_sY) \  =\
-\frac12\Bigl[T^0(X,Y)+T^0(I_sX,I_sY)\Bigr]-2U(X,Y)-%
8n(n+2)Sg(X,Y).
\end{aligned}
\end{equation}
\end{thrm}

We say that $M$ is a \emph{qc-Einstein manifold} if the horizontal
Ricci tensor is proportional to the horizontal metric  $g$,
\begin{equation}\label{qcA}Ric(X,Y)=\frac{S}{4n}g(X,Y)
\end{equation}
which taking
into account  \eqref{sixtyfour}  is equivalent to $T^0=U=0$.
Furthermore, by \cite[Theorem 4.9]{IMV} and \cite[Theorem 1.1]{IMV3} any
qc-Einstein structure has  constant qc-scalar curvature. It should be mentioned that  qc-Einstein structures have proved useful in the construction of metrics with special holonomy \cite{AFISUV1} and heterotic string theory, see Section \ref{s:strings} for some details on the latter. Such applications are possible due to the following properties/ characterization of the qc-Einstein structures, see \cite[Theorem 1.3]{IV2} and \cite[Theorem 4.4.4]{IV3} for $S\not=0$ and \cite[Theorem 5.1]{IMV3} for $S=0$ cases, {respectively}.

\begin{thrm}\label{str_eq_mod_th} Let  $M$ be a qc manifold. The following conditions are equivalent:
\begin{enumerate}[a)]
\item $M$ is a qc-Einstein manifold;
\item  locally, the given qc-structure is defined by  1-form $(\eta_1,\eta_2,\eta_3)$  such that for some constant $S$ we have
\begin{equation}\label{str_eq_mod}
d\eta_i=2\omega_i+\frac{S}{8n(n+2)}\eta_j\wedge\eta_k;
\end{equation}
\item locally,  the given qc-structure is defined by  1-form $(\eta_1,\eta_2,\eta_3)$  such that the corresponding connection 1-forms vanish on $H$, in fact, $\nabla I_i=-\alpha_j\otimes I_k+\alpha_k\otimes I_j$,
$\nabla\xi_i=-\alpha_j\otimes\xi_k+\alpha_k\otimes\xi_j$ with $\alpha_s=-\frac{S}{8n(n+2)}\eta_s$.
\end{enumerate}
\end{thrm}
In particular, in the positive scalar curvature case the qc-Einstein manifold are exactly the locally 3-Sasakian manifolds, i.e, for every $p\in M$  there exist an open neighborhood $U$ of $p$ and a  matrix $\Psi\in \mathcal{C}^\infty(U:SO(3))$, s.t., $\Psi\cdot\eta$ is 3-Sasakian.
 {A (4n + 3)-dimensional
(pseudo) Riemannian manifold $(M,g)$  is 3-Sasakian  if the cone metric  is a (pseudo) hyper-K\"ahler metric \cite{BG,BGN}.  We note explicitly that in some questions it is useful to define 3-Sasakian manifolds in the wider sense of positive (the usual terminology) or negative 3-Sasakian structures, cf. \cite[Section 2]{IV2} and \cite[Section 4.4.1]{IV3} where the "negative" 3-Sasakian term was adopted in the case when the Riemannian cone is hyper-K\"ahler of signature $ (4n,4)$.  As well known, a positive 3-Sasakian manifold is Einstein with a positive Riemannian scalar curvature \cite{Kas} and, if complete,  it is compact with finite fundamental group due to Myer’s
theorem.  The negative 3-Sasakian structures are Einstein with respect to the corresponding pseudo-Riemannian metric of signature $(4n,3)$ \cite{Kas,Tan}.  In this case, by a simple change of signature, we obtain a positive definite $nS$ metric on $M$,  \cite{Tan,Jel,Kon}.}

\subsubsection{The quaternionic Heisenberg Group $\QH$.}\label{ss:qc Heisenberg}
{The basic  example of a qc manifold is provided by the quaternionic Heisenberg group $\QH$}  on which we introduce coordinates by regarding $\QH=\mathbb{H}^n\times\text{Im}\mathbb{H}$, $\quad
(q,\omega)\in \QH$ so that the multiplication takes the form \eqref{e:H-type Iwasawa groups}.

The "\emph{standard}" qc contact form in quaternion variables is
$
\tilde\Theta= (\tilde\Theta_1,\ \tilde\Theta_2, \
\tilde\Theta_3)= \frac 12\ (d\omega - q \cdot d\bar q + dq\,
\cdot\bar q) $ or, using real coordinates,
\begin{gather}\label{e:Heisenbegr ctct forms}\notag
\tilde\Theta_1 = \frac 12\ dx- x^\alpha d t^\alpha+ t^\alpha d
x^\alpha-z^\alpha d y^\alpha + y^\alpha d z^\alpha, \quad
 \tilde\Theta_2= \frac 12\ dy- y^\alpha d t^\alpha+ z^\alpha d x^\alpha+ t^\alpha d y^\alpha - x^\alpha d z^\alpha,\\
\tilde\Theta_2= \frac 12\ dz- z^\alpha d t^\alpha- y^\alpha d
x^\alpha+ x^\alpha d y^\alpha + t^\alpha d z^\alpha.
\end{gather}
The left-invariant horizontal vector fields are

\begin{equation}\label{qHh}
\begin{aligned}
T_{\alpha} = \dta {} +2x^{\alpha}\dx {}+2y^{\alpha}\dy
{}+2z^{\alpha}\dz {} , \
  X_{\alpha} = \dxa {}-2t^{\alpha}\dx {}-2z^{\alpha}\dy
{}+2y^{\alpha}\dz {} ,\\
Y_{\alpha} =  \dya {} +2z^{\alpha}\dx {}-2t^{\alpha}\dy
{}-2x^{\alpha}\dz {},  \  Z_{\alpha} = \dza {} -2y^{\alpha}\dx
{}+2x^{\alpha}\dy {}-2t^{\alpha}\dz {}\,,
\end{aligned}
\end{equation}
with corresponding sub-Laplacian
\begin{equation}\label{e:qc Heis sub-laplacian}
\lap_{\tilde\Theta} u=\sum_{\alpha=1}^n \left (T_{\alpha}^2u+ X_{\alpha}^2u+Y_{\alpha}^2u+Z_{\alpha}^2u \right ).
\end{equation}
The (left-invariant vertical) Reeb  fields $\xi_1,\xi_2,\xi_3$ are
$
\xi_1=2\dx {}, \quad \xi_2=2\dy {},\quad \xi_3=2\dz {}.
$
 On $\QH$, the left-invariant flat connection  is the
Biquard connection, hence $\QH$ is a flat qc structure. It should be noted that the latter property characterizes (locally) the qc structure ${\tilde\Theta}$ by \cite[Proposition 4.11]{IMV}, but in fact vanishing of the curvature on the {horizontal} space is {enough} \cite[Propositon 3.2]{IV}.
Thus, by \cite[Propositon 3.2]{IV}, a quaternionic contact manifold is locally isomorphic to the
quaternionic Heisenberg group exactly when the curvature of the
Biquard connection restricted to $H$ vanishes, $R_{|_H}=0$.

\subsubsection{Standard qc-structure on 3-Sasakain sphere and the qc Cayley transform}\label{ss:qc sphere}
{The second example is the {"s}tandard{"} qc-structure on the 3-Sasakain sphere.}
 The "standard" qc 3-form on the sphere $S^{4n+3}=
\{\abs{q}^2+\abs{p}^2=1 \}\subset \Hn\times\mathbb{H}$, is
\begin{equation}\label{e:stand cont form on S}
\tilde\eta\ =\ dq\cdot \bar q\ +\ dp\cdot \bar p\ -\ q\cdot d\bar q -\
p\cdot d\bar p.
\end{equation}
We identify $\QH$ with the boundary $\Sigma$ of a Siegel domain in
$\Hn\times\mathbb{H}$,
$
\Sigma\ =\ \{ (q',p')\in \Hn\times\mathbb{H}\ :\ \text{Re} {\ p'}\ =\
\abs{q'}^2 \},
$
by using the map $(q', \omega')\mapsto (q',\abs{q'}^2 - \omega')$.
The Cayley transform,
$\mathcal{C}:S\setminus{\{\text{pt.}\}}\rightarrow \Sigma$, is
\begin{equation*}\label{e:QC Cayley}
 (q', p')\ =\ \mathcal{C}\ \Big ((q, p)\Big)=( (1+p)^{-1} \ q , (1+p)^{-1} \ (1-p)).
  \end{equation*}
 By \cite[Section 8.3]{IMV} we have on $\QH$
\begin{equation}\label{e:Cayley transf ctct form group}
\Theta\ \overset{def}{=}\ \lambda\ \cdot (\mathcal{C}^{-1})^*\, \tilde\eta\ \cdot \bar\lambda\ =\ \frac
{8}{\abs{1+p'\, }^2}\, \tilde\Theta.
\end{equation}
where $\lambda\ = {\abs {1+p\,}}\, {(1+p)^{-1}}$ is a unit quaternion.
Alternatively, on the sphere this can be written as
\begin{equation}\label{e:Cayley transf ctct form}
\eta\overset{def}{=}\mathcal{C}^*\,\tilde{\Theta}\ =\ \frac
{1}{2\,\abs {1+p\, }^2}\, \lambda\, \tilde\eta\, \bar \lambda,
\end{equation}
where $\lambda$ is a unit quaternion.
 In any case, the above formulas  show  the Cayley transforms is a conformal quaternionic contact map. In addition, we can use it to determine the qc scalar curvature of the sphere $(S^{4n+3}, \tileta)$ and find a solution of the Yamabe equation on $\QH$.
For  $(q',p')\in \Sigma\subset \Hn\times\mathbb{H}, \quad p'=|q'|^2+ \omega'$, consider the function
\begin{equation}
\begin{aligned}\label{e:h and Phi}
 h& =\frac  {1}{16}|1+p'|^2=\frac {1}{16}\left [(1+|q'|^2)^2+
 |\omega'|^2 \right ], \\
\Phi & =\left({2h} \right)^{-(Q-2)/4} = 8^{(Q-2)/4}\ {\left [(1+|q'|^2)^2+ |\omega'|^2 \right ]}^{-(Q-2)/4},
\end{aligned}
\end{equation}
so that we have $$\Theta=\frac {1}{2h}\tilde\Theta =\Phi^{4/(Q-2)}\tilde\Theta.$$
 A small calculation shows that the sub-laplacian of $h$ with respect to $\tilde\Theta$ is given
by $\triangle_{\tilde\Theta} h = \frac {Q-6}{4} + \frac {Q+2}{4}|q'|^2$ and thus  $\Phi$ is a solution of the qc Yamabe equation
on the  Heisenberg group,
\begin{equation}\label{e:Yamabe for Phi}
\triangle_{\tilde\Theta} \Phi = -K\, \Phi^{2^*-1}, \qquad K=(Q-2)(Q-6)/8.
\end{equation}
Denoting with $\Lap$ and $\Lap_{\tilde\Theta}$ the
conformal sub-laplacians of $\Theta$ and $\tilde\Theta$, respectively, we have (see also \eqref{e:conf Yamabe volume})
$$\Phi^{-1}\Lap (\Phi^{-1}u) = \Phi^{-2^*}\Lap_{\tilde\Theta} u.$$
Taking $u=\Phi$ we come to $\Lap ( 1 ) = \Phi^{1-2^*}\triangle_{\tilde\Theta} \Phi$, since the qc structure ${\tilde\Theta}$ is flat, which shows \begin{equation}\label{e:S of standard qc sphere}
S_{\tileta}= S_{\Theta}=4\frac {Q+2}{Q-2}K=8n(n+2)
 \end{equation}
 using  that the two structures are isomorphic via the diffe{o}morphism $\mathcal{C}$, or rather its extension, since we can consider $\mathcal{C}$ as a quaternionic contact conformal transformation between the whole sphere $ S^{4n+3}$  and the compactification $\hat{\Sigma}\cup {\infty}$ of the quaternionic Heisenberg group by adding the point at infinity, cf. \cite[Section 5.2]{IMV1}.

\subsubsection{QC conformal flatness \cite{IV}}
A QC manifold $(M,\eta)$ is called locally qc conformally flat if there is a local  diffeomorphims $F:\QH\rightarrow M$, such that $F^*\eta=\phi\Psi \tilde\Theta$ for some positive function $\phi$.

The qc-conformal flatness of a manifold is characterized by the vanishing of the  \emph{qc-conformal curvature} tensor $W^{qc}$ found in \cite{IV},

\begin{multline*}
W^{qc}(X,Y,Z,V)=\frac14\Big[R(X,Y,Z,V)+\sum_{s=1}^3R(I_sX,I_sY,Z,V)\Big]-\frac12\sum_{s=1}^3\omega_s(Z,V)\Big[T^0(X,I_sY)-T^0(I_sX,Y)\Big]\\+\frac{S}{32n(n+2)}\Big[(g\owedge g)(X,Y,Z,V)+\sum_{s=1}^3(\omega_s\owedge\omega_s)(X,Y,Z,V)\Big]+(g\owedge U)(X,Y,Z,V)+\sum_{s=1}^3(\omega_s\owedge I_sU)(X,Y,Z,V),
\end{multline*}
{where $(A\owedge B)$ denotes the Kulkarni-Nomizu product of two tensor, i.e.,
$$(A\owedge B)(X,Y,Z,V)=A(X,Z)B(Y,V)-A(Y,Z)(B(X,V)+A(Y,V)B(X,Z)-A(X,V)(B(Y,Z).$$

\begin{thrm}[\cite{IV}]\label{T:flat}
a)$W^{qc}$ is qc-conformal invariant, i.e., if
$\bar\eta=\kappa\Psi\eta$ then $ W^{qc}_{\bar\eta}=\phi\,
W^{qc}_{\eta}$, where
$0<\kappa\in \mathcal{C}^\infty(M)$, and $\Psi\in \mathcal{C}^\infty(M:SO(3))$.

b) A qc-structure  is locally
qc-conformal to the standard flat qc-structure on the
quaternionic Heisenberg group $\QH$  if and only if  $W^{qc}=0$.
\end{thrm}
Taking into account the qc Cayley transform we also have
the quaternionic sphere $S^{4n+3}$ if and only if the qc conformal
curvature vanishes, $W^{qc}=0$.

We end this section with the remark that unlike the CR case the realization of qc manifolds as hypersurfaces in a hyper-K\"ahler manifold is very restrictive. For example, it was shown in \cite{IMV'14} that if $M$ is a connected qc-hypersurface in the flat quaternion space  $\R^{4n+4}\cong \Hnn$ then, up to a quaternionic affine transformation of $\Hnn$, $M$ is contained in one of the following three hyperquadrics:
$$(i) \ \ |q_1|^2+\dots+|q_n|^2 + |p|^2=1,\qquad
(ii)\ \  |q_1|^2+\dots+|q_n|^2 - |p|^2=-1 ,\qquad (iii)\ \ |q_1|^2+\dots+|q_n|^2 +\R{e}(p)=0.$$
Here $(q_1,q_2,\dots q_n,p)$ denote the standard quaternionic coordinates of $\Hnn.$
In particular, if $M$ is a compact qc-hypersurface of $\R^{4n+4}\cong \Hnn$ then, up to a quaternionic affine transformation of $\Hnn$, $M$ is the standard 3-Sasakian sphere. For other results and more details we refer to \cite{IMV'14}.

\section{The CR Yamabe problem and the CR Obata theorem}

The CR Yamabe problem seeks pseudoconformal pseudohermitian transformation of a compact CR pseudohermitian manifold which lead to constant scalar curvature of the canonical Tanaka-Webster connection, see Section \ref{ss:Yamabe Iwasawa mnfld}. After the works of D. Jerison \& J. Lee \cite{JL1,JL2,JL3,JL4} and N. Gamara \& R. Yacoub \cite{Ga}, \cite{GaY} the solution of the CR Yamabe problem on a compact manifold is complete.

{Let $(M^{2n+1}, \, \eta)$ be a  strictly
pseudoconvex CR manifold andf $\Upsilon(M,[\eta])$ be the CR-Yamabe constant (cf. \eqref{e:Yamabe constant Iwasawa}). The CR-Yamabe constant $\Upsilon (M,[\eta])$ depends only on the CR structure of M, not {on} the choice of $\eta$. }

{The solution of the CR-Yamabe problem is outlined in the next fundamental results.}

\begin{thrm} [\cite{JL2,JL4,Ga,GaY}]
{Let $(M^{2n+1}, \, \eta)$ be a  strictly
pseudoconvex CR manifold. The CR-Yamabe constant satisfies the ineguality  $ \Upsilon(M,[\eta]) \leq \Upsilon(S^{2n+1},\bar{\eta})$, where $S^{2n+1}\subset \mathbb{C}^{n+1}$ is the sphere with its standard CR structure $\bar{\eta}$.}
 \begin{enumerate}
 \item[a)]  If $\Upsilon([M,\eta]) < \Upsilon(S^{2n+1},[\bar\eta])$, then the Yamabe
equation has a solution. \cite{JL2}
\item[b)]  If $n\geq 2$ then the Yamabe constant  satisfies
\[
\Upsilon(M,[\eta_\epsilon])=
\begin{cases}
    \Upsilon(S^{2n+1},[\bar\eta])\left ( 1-c_n\abs {S(q)}^2\epsilon^4\right ) +\mathcal{O}(\epsilon^5),  \quad{n\geq 2;} \\
    \Upsilon(S^{5},\bar\eta])\left ( 1+c_2\abs {S(q)}^2\epsilon^4\ln \epsilon\right ) +\mathcal{O}(\epsilon^4), \quad{n=2.}
\end{cases}
\]
{Thus, if M is not locally CR equivalent to $(S^{2n+1},\bar\eta)$, then $\Upsilon(M,[\eta]) <
\Upsilon(S^{2n+1},[\bar\eta])${,}} \cite{JL4}{.}
\item[c)] If $n=1$ or $M$ is locally CR equivalent to $S^{2n+1}$, then the Yamabe equation has a
solution{,} \cite{Ga,GaY}.
\end{enumerate}
\end{thrm}

\subsection{Solution of the CR  Yamabe problem on the sphere and Heisenberg group}\label{ss:Jerison and Lee}

{The CR version of the Obata theorem was proved by D. Jerison and J. Lee \cite{JL3}.}
\begin{thrm}[\cite{JL3}]
If $\eta$ is the contact form of a pseudo-Hermitian structure
proportional to the standard contact form $\bar\eta$ on the unit sphere
in $\mathbb{C}^{n+1}$ and the pseudohermitian scalar curvature $S_\eta=$const, then up to a multiplicative
constant $\eta=\Phi^* \,\bar\eta$ with $\Phi$ a CR automorphism of the
sphere.
\end{thrm}
 A key step of the  proof consists of showing that  a CR structure with a constant pseudohermitian scalar curvature is pseudoconformal to the standard pseudo-Einstein torsion-free structure on the  CR sphere iff  it is a pseudo-Einstein with vanishing Webster torsion.
It is well known
that a strictly pseudoconvex torsion-free CR manifold is Sasakian.
In addition, if the CR space is pseudo{-}Einstein then  it is not
hard to observe that it is a Sasaki-Einstein space with respect to the associated Riemannian metric $h$. Indeed, the
Ricci tensors $Ric^g$ and $Ric$ of the Levi-Civita and the
Tanaka-Webster  connection, respectively, of a torsion-free CR
space are connected by \cite{DT} $Ric^g(X,Y)=Ric(X,Y)-2g(X,Y)$,
$Ric^g(\xi,\xi)=2n$. Because of the second identity, a Sasaki{-}
Einstein space has Riemannian scalar curvature $S^g={2n(2n+1)}$.
Hence, a torsion-free pseudo Einstein CR manifold is a Sasaki
Einstein if the pseudohermitian scalar  curvature is equal to
$S=4n(n+1)$ and the  Jerison-Lee theorem  can be stated as follows.
\begin{thrm}[\protect\cite{JL3}]\label{jerl}
If a compact Sasaki-Einstein manifold $(M,\bar\eta)$ is
pse{u}doconformal to a CR manifold $(M,\eta=2h\bar\eta)$ with
constant positive pseudohermmitian scalar curvature $S=4n(n+1)$
then $(M,\eta)$ is again a Sasaki-Einstein space.
\end{thrm}
The proof follows trivially from the  divergence formula
discovered in \cite{JL3}
which we  state  in real coordinates
in Theorem~\ref{t:JL thrm}. First, we need some definitions. Let $h>0$ be a smooth function on a pseudohermitian  manifold $(M, \eta, g)$ and
$\bar\eta=\frac{1}{2h}\eta$ be a pseudoconformal to $\eta$ contact form. We will denote the connection, curvature and torsion tensors  of $\bar\eta$ by over-lining the same object corresponding to $\eta$.  The new  Reeb vector field  $\bar\xi$ is
$
\bar\xi\ =\ 2h\,\xi\ + 2h\ J\Cr h, $
where $\Cr h$ is the horizontal gradient, $g(\Cr h,X)=dh(X)$. The
Webster torsion and the pseudohermitian Ricci tensors of $\eta$
and $\bar\eta$ are related by \cite{L2},
\begin{gather}\label{e:A conf change}
4h\bar A(X,JY)=4hA(X,JY)+\Cr^2h(X,Y)-\Cr^2 h(JX,JY)
\end{gather}
\begin{multline}\label{e:Ric conf change h}
\overline {\rho}(X,JY)=\rho(X,JY)-(4h)^{-1}(n+2)[\Cr^2h(X,Y)+\Cr^2h(Y,X)+\Cr^2h(JX,JY)+\Cr^2h(JY,JX)]\\+(2h^2)^{-1}(n+2)[ dh(X)dh(Y)+dh(JX)dh(JY) ]
-(2h)^{-1}\left(\lap h- h^{-1}(n+2)|\Cr h|^2\right)g(X,Y),
\end{multline}
{where $\triangle h=\Cr^*dh=\sum_{a=1}^{2n}\Cr^2 h(e_a,e_a)$ is the sublaplacian. The pseudohermitian scalar curvatures changes
according to \cite{L2},
\begin{equation}\label{e:conf change scalar curv h}
\overline {S} = 2hS -
2(n+1)(n+2)h^{-1}|\Cr h|^2  +4(n+1)\lap h.
\end{equation}
Let $B$ be the traceless part of $\rho$, $B(X,JY)\overset{def}{=}\rho(X,JY)+\frac {S}{2n}g(X,Y)$ since by \eqref{e:CR Ric type decomp}, we have $\rho(e_a,Je_a)=-Ric(e_a,e_a)=-S$. The above formulas imply
\begin{multline}\label{e:Ric_0 conf change h}
\bar B(X,JY)=B(X,JY)-\frac{n+2}{4h}[\Cr^2h(X,Y)+\Cr^2h(Y,X)+\Cr^2h(JX,JY)+\Cr^2h(JY,JX)]\\+\frac{n+2}{2h^2}[ dh(X)dh(Y)+dh(JX)dh(JY) ]
+\frac {n+2}{2n}\left(\frac{\lap h}{h}-\frac 1{h^2}|\Cr h|^2\right)g(X,Y).
\end{multline}

Suppose $\bar \theta$ is Sasaki-Einstein structure, i.e., $\bar
A=\bar B=0$ and both pseudo-Hermitain structures are of constant
pseudohermitian scalar curvatures $\bar S=S=4n(n+1)$. With these assumptions \eqref{e:conf change scalar curv h} becomes
\begin{equation}\label{sublap}
\lap h=n-2nh +\frac{(n+2)}{2h}|\Cr h|^2.
\end{equation}
At this point we recall the Ricci identities for the
Tanaka-Webster connection \cite{L2} (see e.g. {\cite{IVZ,IV2}} for
these and other expressions in real coordinates),
\begin{equation}\label{Riden}
\begin{split}
\Cr^2h(X,Y)-\Cr^2h(Y,X)=-2\omega(X,Y)dh(\xi){,}\qquad
\Cr^2h(X,\xi)-\Cr^2{h}(\xi,X)=A(X,\Cr h){,}\\
\Cr^3 h(X,Y,Z)-\Cr^3{h}(Y,X,Z)=-R(X,Y,Z,\Cr h)-2\omega(X,Y)\Cr^2 h(\xi,Z).
\end{split}
\end{equation}
The contracted Bianchi identities for Tanaka-Webster connection \cite{L2} are
\begin{equation}\label{bia}
\begin{split}
dS(X)=2\sum_{a=1}^{2n}(\Cr_{e_a}Ric)(e_aX)=-2\sum_{a=1}^{2n}(\Cr_{e_a}\rho)(e_a,JX)+4
(n-1)\sum_{a=1}^{2n}(\Cr_{e_a}A)(e_a,JX);\\
Ric(\xi,X)=\sum_{a=1}^{2n}(\Cr_{e_a}A)(e_a,Z);\quad
dS(\xi)=2\sum_{a,b=1}^{2n}(\Cr^2_{e_ae_b}A)(e_a,e_b).
\end{split}
\end{equation}
When $\bar A=0$, \eqref{e:A conf change} takes the form
\begin{equation}\label{tor}
4A(X,JY)=-\Big[\Cr^2h(X,Y){-}\Cr^2h(JX,JY)\Big] .
\end{equation}
{Differentiating \eqref{tor} using the equation $\Cr J=0$,
taking the trace in the obtained equality and applying  the Ricci identities \eqref{Riden},  \eqref{e:CR Ric type decomp} and the CR Yamabe equation \eqref{e:conf change scalar curv h}, we  find the next formula for the divergence of $A$,}

\begin{multline}\label{divA1}
\Cr^*A(JX)=-2\rho(JX,\Cr h)-\frac{n+2}h[\Cr^2(\Cr
h,X)-2dh(JX)dh(\xi)]\\+2ndh(X)+\frac{n+2}{2h^2}|\Cr
h|^2dh(X)-(2n+4)\Cr^2(JX,\xi),
\end{multline}
where the divergence of a 1-form $\alpha$ is $\Cr^*\alpha=\sum_{a=1}^{2n}(\Cr_{e_a}\alpha)e_a $.
A substitution of \eqref{sublap} into \eqref{e:Ric_0 conf change h}  and a use {of} the Ricci identities together with  $\bar B=0$ give
\begin{multline}\label{be}
B(X,JY)=\frac{n+2}{2h}\Big[\Cr^2h(Y,X)+\Cr^2h(JY,JX)-2\omega(X,Y)dh(\xi)\Big]\\-\frac{n+2}{2h^2}[
dh(X)dh(Y)+dh(JX)dh(JY) ] -\frac {n+2}{2}\left(\frac1{h}-2+\frac
1{2h^2}|\Cr h|^2\right)g(X,Y).
\end{multline}
From $\rho(X,JY)=B(X,JY)-2(n+1)g(X,Y{)}$ and
\eqref{be} it follows
\begin{multline}\label{rho1}
\rho(X,J\Cr h)=\frac{n+2}{2h}\Big[\Cr^2h(\Cr h,X)+\Cr^2h(J\Cr
h,JX)-2dh(\xi)dh(JX)-\frac3{2h}|\Cr h|^2dh(X)-dh(X)\Big]\\-ndh(X).
\end{multline}
Substituting equation \eqref{rho1} into equation \eqref{divA1} shows
\begin{equation}\label{divA2}
\Cr^*A(JX)=\frac{n+2}4\Big[\frac{\Cr^2{\color{blue} h}(J\Cr
h,JX)}{h^2}-\frac{|\Cr
h|^2}{h^3}dh(X)-\frac1{h^2}dh(X)-\frac2{h}\Cr^2h(JX,\xi)\Big].
\end{equation}
With the help of \eqref{tor}, \eqref{be} and \eqref{divA2} we define the
following 1-forms

\begin{equation}
\begin{split}
d(X)& =-4h^{-1}A(X,J\nabla h)=\frac{\nabla ^{2}h(\nabla h,X)-\nabla
^{2}h(J\nabla h,JX)}{h^{2}};\hspace{3cm} \\
e(X)& =\frac{2}{n+2}h^{-1}B(X,J\nabla h)=\frac{\nabla ^{2}h(\nabla
h,X)+\nabla ^{2}h(J\nabla h,JX)}{h^{2}}-\frac{2dh(\xi )dh(JX)}{h^{2}} \\
& -\Big(\frac{1}{h^{2}}-\frac{2}{h}+\frac{3|\nabla h|^{2}}{2h^{3}}\Big)dh(X);
\\
u(X)& =\frac{4}{n+2}\nabla ^{\ast }A(JX)=\frac{\nabla ^{2}{h}(J\nabla h,JX)}{%
h^{2}}-\frac{2}{h}\nabla ^{2}h(JX,\xi )-\Big(\frac{|\nabla h|^{2}}{h^{3}}+%
\frac{1}{h^{2}}\Big)dh(X).
\end{split}
\label{oneforms}
\end{equation}%
We obtain easily from \eqref{oneforms} the next identity
\begin{equation}
u(X)=\frac{e(X)-d(X)}{2}-\frac{2\nabla ^{2}h(JX,\xi )}{h}-\frac{1}{h^{2}}%
\Big(\frac{1}{2}+h+\frac{|\nabla h|^{2}}{4h}\Big)dh(X)+\frac{dh(\xi )dh(JX)}{%
h^{2}}.  \label{iden1}
\end{equation}%
Define the following {tensors}
\begin{equation}
\begin{split}
{D}(X,Y)& =-4A({X,Y}),\quad D^{h}(X,Y,Z)=h^{-1}\left [ {D}(.,Z)dh(.)\right ]_{{[1]}}, \\
{E}(X,Y)& =\frac{2}{n+2}B(X,Y),\quad E^{h}(X,Y,Z)=h^{-1}\left [{E}%
(X,.)dh(.)\right ]_{{[-1]}}.
\end{split}
\label{3tens}
\end{equation}
{In other words, we have
\begin{align*}
D^{h}(X,Y,Z) & =\frac{1}{2h}\left[ dh(X)D(Y,Z)+dh(JX){D}(JY,Z)%
\right],\\
 E^{h}(X,Y,Z)& =\frac{1}{2h}\left[ dh(Z){E}(X,Y)-dh(JZ){E}(X,JY)%
\right].
\end{align*}
}
At this point we can state one of the main results of \cite{JL3}.

\begin{thrm}[\protect\cite{JL3}]
\label{t:JL thrm} Let $(M,\bar{\eta})$ be a Sasaki-Einstein manifold
pseudoconformally equivalent to a CR manifold $(M,\eta ,\bar{\eta}=\frac{1}{%
2h}\eta )$ of constant pseudohermitian scalar curvature so that $\bar{S}%
=S=4n(n+1)$. For
\begin{equation}
f=\frac{1}{2}+h+\frac{|\nabla h|^{2}}{4h},  \label{function}
\end{equation}%
we have
\begin{multline}
\nabla ^{\ast }\Big(f[{d}+{e}]-dh(\xi )J{d}+dh(\xi )J{e}-6dh(\xi )J{u}\Big)
\label{divm} \\
=\frac{1}{2}\left( \frac{1}{2}+h\right) \left( |{D}|^{2}+|{E}|^{2}\right) +\frac{%
h}{8}|E^h{+}D^h|^{2}+Q(d,e,u).
\end{multline}%
where $Q(d,e,u)$ is non-negative quadratic form of {the vector fields} $(d,e,u)
$.
\end{thrm}

\begin{proof} {We recall, that for a horizontal 1-form $\alpha$ the 1-form $J\alpha$ is
defined by $J\alpha(X)=-\alpha(JX)$. The divergences of the involved vector fields are calculated  using \eqref{oneforms} and the Bianchi identies %
\eqref{bia}. Since $S=4n(n+1)$ the Bianchi identities \eqref{bia} take the
form
\begin{equation}  \label{biac}
\nabla^*B(JX)=2(n-1)\nabla^*A(JX)=\frac{(n+2)(n-1)}2u(X), \quad \nabla^*Ju=%
\frac{4}{n+2}\sum_{a,b=1}^{2n}(\nabla^2_{e_ae_b}A)(e_a,e_b)=0.
\end{equation}
A direct computation gives
\begin{equation}  \label{divD}
\nabla^*d=\sum_{a=1}^{2n}(\nabla_{e_a}D)(e_a)=-h^{-1}D(\nabla
h)-(n+2)h^{-1}u(J\Cr h)+\frac12|D|^2.
\end{equation}
Using the properties of $A$, we calculate
\begin{equation}  \label{divJD}
\nabla^* (Jd)=h^{-1}d(J\Cr h)+(n+2)h^{-1}u(J\Cr h),
\end{equation}
taking into account $\sum_{a,b=1}^{2n}A(e_a,Je_b)\nabla^2h(Je_a,e_b)=0$ due to \eqref{tor}%
. Similarly, we calculate
\begin{equation}  \label{divE}
\nabla^* e =(n-1)h^{-1}u(\Cr h) +\frac12|{E}|^2
\end{equation}
after using the equality $\frac2{n+2}h^{-1}\sum_{a,b=1}^{2n}B(e_a,Je_b)%
\nabla^2h(e_a,e_b)=h^{-1}e(\Cr h)+\frac12|{E}|^2$
follows from \eqref{be} and \eqref{3tens}.
Finally, we have
\begin{equation}  \label{divJE}
\nabla^*(Je)=(n-1)h^{-1}u(J\Cr h)
\end{equation}
since $B(JX,JX)=0$ and $\sum_{a,b=1}^{2n}B(e_a,Je_b)\nabla^2h(Je_a,e_b)=0$
due to \eqref{be}. We obtain from \eqref{function} after applying the Ricci
identities and \eqref{oneforms} the identity
\begin{equation}  \label{funcder}
df(X) =\frac{h}2\Big(u(X)+d(X)\Big)+\nabla^2h(JX,\xi)+h^{-1}fdh(X)-h^{-1}dh(%
\xi)dh(JX).
\end{equation}
At this point we are ready to calculate the divergence formula using \eqref{divD},\eqref{divJD},%
\eqref{divE},\eqref{divJE}, \eqref{funcder} and \eqref{iden1} which give
\begin{multline}  \label{divmf}
\nabla ^{\ast }\Big(f[d+e]-dh(\xi )Jd+dh(\xi )Je-6dh(\xi )Ju\Big) \\
=\frac{1}{2}\left ( \frac{1}{2}+h+\frac{|\nabla h|^{2}}{4h} \right)\Big[|{D}|^{2}+|{E}|^{2}\Big]+\frac{h}{2}\Big[%
|d|^{2}+|e|^{2}+6|u|^{2}+4g(d,u)-4g(u,e)\Big] \\
=\frac{1}{2}\left( \frac{1}{2}+h\right) \left( |{D}|^{2}+|{E}|^{2}\right) +%
\frac{h}{4}|D^h+E^h|^{2}+\frac{h}{2}\left(
|d|^{2}+|e|^{2}+6|u|^{2}+4g(d,u)-4g(u,e)-2g(d,e)\right) \\
=\frac{1}{2}\left( \frac{1}{2}+h\right) \left( |{D}|^{2}+|{E}|^{2}\right) +%
\frac{h}{4}|D^{h}+E^{h}|^{2}+\frac{h}{2}Q(d,e,u)
\end{multline}
with $Q=%
\begin{bmatrix}
1 & -1/2 & 2 \\
-1/2 & 1 & -2 \\
2 & -2 & 6%
\end{bmatrix}%
$ using the next identity in the last equality
\begin{equation}\label{e:key 3-tensors}
\frac{|\nabla h|^{2}}{4h}\left( |{D}|^{2}+|{E}|^{2}\right) =\frac{h}{2}%
|D^h+E^h|^{2}-{h}g(d,e).
\end{equation}
Notice  that  $Q$ has eigenvalues $\frac {15\pm\sqrt {209}}{4}$ and $\frac 12$, hence it is a positive definite matrix.
}
{Finally, the validity of \eqref{e:key 3-tensors} can be seen as follows,
\begin{multline}
|D^h+E^h|^2=\frac {1}{4h^2}\left \vert dh(e_a){D}(e_b,e_c)+dh(Je_a){D}(Je_b,e_c)
+dh(e_c){E}(e_a,e_b)-dh(Je_c){E}(e_a,Je_b) \right\vert^2\\
=\frac {1}{2h^2}|\nabla h|^2\left (| {D}|^2+ |{E}|^2\right) +\frac
{2}{h^2}\left ({D}(\nabla h,e_a){E}(\nabla h,e_a)\right ) =\frac
{1}{2h^2}|\nabla h|^2\left (| {D}|^2+ | {E}|^2\right) + {2}g(d,e).
\end{multline}
}

\end{proof}

\subsection{The uniqueness theorem in a Sasaki-Einstein class}\label{ss:Sasaki-Einstein uniqieness Yamabe}
Motivated by Theorem \ref{t:Obata Yamabe} it is natural to
investigate the uniqueness of the pseudohermitian structures of
constant scalar curvature in the Sasaki-Einstein case, especially
in view of  Theorem \ref{jerl}. The fact that the divergence
formula of \cite{JL3} can  be stated as in Theorem \ref{jerl} was
observed earlier in \cite{IMV} which influenced the results
\cite{IMV,IMV1} in which there is a clear separation of the two
steps of Jerison and Lee's argument,  the first involving the
conformal equivalence of an "Einstein structure" to a structure of
"constant scalar curvature" and the second involving the
characterization of the conformal equivalence of two (conformally
flat) Einstein structures. A corresponding QC version of the Obata
uniqueness theorem was formulated by the second author. Clearly,
{in the CR case} Theorem \ref{jerl}
addresses the first step, while a part of the second step is
contained in \cite{JL3} where  the (suitable) conformal factor is
characterized as a pluriharmonic function. For the completion of
the second step one can reduce to the result mentioned in Remark
\ref{r:sphere charcat} with an argument employed in \cite[Theorem
1.3]{IVO} (see also the end of Section \ref{ss:CR cpct Obata
proof}) rather than relying on the calculation on the Heisenberg
group when in the pseudoconformal class of the Sasaki-Einstein
sphere as Jerison \& Lee did. For details of this last reduction
see \cite{Wang1}. {Alternatively, a
conceptual proof using again in the first step the Jerison \&
Lee's divergence formula and as a second step a generalization of
\eqref{e:lap of div conf v.f.} is possible based on the proof
found in the quaternionic contact case 
\cite{IMV15a}. Next, we sketch briefly the obvious adaption of the
argument from the quaternionic contact case \cite{IMV15a}. First,
we use the well known fact that a vector field $Q$ on a CR
pseudo-Hermitian manifold is an infinitesimal CR transformation
iff there is a (smooth real-valued) function $\sigma$ such that
$Q=-\frac 12J\nabla\sigma-\sigma(\xi)\xi$ and $\sigma$ satisfies
the second order equation $\mathcal{L}_QJ=0$, see \cite{ChLee}. In
fact, decomposing $Q$ in its horizontal and vertical parts
$Q=Q_H-\sigma\xi$  it follows that $Q_H$  ("contact Hamiltonian
field") is determined by $\eta(Q_H)=0$ and $i_{Q_H}\, d\eta\equiv
0 \ \ (mod\ \eta)$ while the preservation of the complex structure
gives  the second order system
$[\nabla^2\sigma]_{[-1]}(X,Y)=\sigma A(JX,Y)$. Next,  as a
consequence of the CR Yamabe equation one obtains a formula as in
Lemma \ref{e:lap of div conf v.f.} for an infinitesimal CR
automorphism $Q$ on $(M,\eta)$, namely
 \begin{equation}\label{e:lapdivCR}
 \Delta(\nabla^*Q_H)\ =\ -\ \frac{n-2}{2(n+2)}Q(\text{Scal})\ -\ \frac{\text{Scal}}{2(n+2)}\nabla^*Q_H,
\end{equation}
where $Q_H$ is the horizontal part of $Q$. In our case, where $A=0$, for $\sigma=dh(\xi)$  it follows from Ricci's identity \eqref{e:CR XYxi Ricci}
and \eqref{tor} that the vector field $Q$ defined by
$Q=-\frac 12 J \nabla dh(\xi)-dh(\xi)\xi$ is an infinitesimal CR  vector field unless it vanishes. Now, for  $f$ defined in \eqref{function}, from \eqref{iden1} and \eqref{funcder} it follows
$Q=-\frac 12 \nabla f-dh(\xi)\xi$. This implies that $\phi=\triangle f$
either vanishes identically or is an eigenfuction of the sublaplacian realizing the smallest
possible eigenvalue on a (pseudo-Einstein) Sasakian manifold. Finally, if $h\not=const$ then the CR Lichnerowicz-Obata theorem \cite{CC09a,CC09b}, see Section \ref{ss:CR Obata}, shows that $(M,\eta)$ is homothetic to the CR unit sphere, which completes the proof. We note that the above arguments have as a corollary that in
 Jerison \& Lee's identity \cite[(3.1)]{JL3}, letting $\phi=\triangle_b Re(f)$  we have $\triangle_b \phi =-2n\phi$. }

Thus,  a pseudoconformal class of a Sasaki-Einstein pseudohermitian form different from the {standard} Sasaki-Einstein form on the round sphere contains a unique (up to homothety) pseudohermitian form of constant CR scalar curvature, namely, the Sasaki-Einstein form itself.

\section{The qc-Yamabe problem and the Obata type uniqueness theorem}

In this section we consider the quaternionic contact version of the Yamabe problem described in Section\ref{ss:Yamabe Iwasawa mnfld}.
{We begin by quoting the next result of Wang \cite{Wei} which follows from the known techniques in the Riemannian and CR settings.}
\begin{thrm}[\cite{Wei}] Let $(M,\eta)$ be a compact quaternionic contact manifold of real dimension $4n+3$.
\begin{enumerate}[a)]
\item The qc Yamabe constant satisfies the inequality $\Upsilon (M,[\eta])\leq \Upsilon (S^{4n+3},
([\tilde\eta])$.
\item If $\Upsilon(M,[\eta])<\Upsilon(S^{4n+3},
[\tilde\eta])$, then the Yamabe problem has a solution.
\end{enumerate}
\end{thrm}
In view of the  Riemannian and CR cases, it is  expected that on a compact qc manifods $\Upsilon(M,[\eta])< \Upsilon(S^{4n+3},
[\tilde\eta])$ unless the qc manifold $(M,[\eta])$ is locally qc conformal to $(S^{4n+3},[\tilde\eta])$.
  Some steps towards the proof of this result include the qc-normal coordinates constructed in \cite{Kun} and the qc conformal tensor, $|W^{qc}|^2$ \cite{IV}.

 Another known general result is the (local) classification of all (local) qc-conformal transformations of the flat structure on the $\QH$ which are also qc-Einstein. This classification is used as a replacement of the theory of the pluriharmonic functions that appear in the CR case. Some attempts in extending the latter in the quaternionic (contact) case can be found in \cite{IMV}. However, so far, these extensions have not proven useful in the solution of the Yamabe problem. The following Theorem precises the result of \cite[Theorem 1.1]{IMV} where only the vanishing qc-scalar curvature case was considered.
\begin{thrm}\label{t:einstein preserving}
Let $\Theta=\frac{1}{2h}\tilde\Theta$ be a conformal deformation of the standard
qc-structure $\tilde\Theta$ on the quaternionic Heisenberg group $\QH$. If $\Theta$ is
also qc-Einstein, then up to a left translation the function $h$ is given by

\begin{equation}\label{e:Liouville conf factor}
h(q,\omega) \ =\ c_0\ \Big [  \big ( \sigma\ +\
 |q+q_0|^2 \big )^2\  +\  |\omega\ +\ \omega_o\ +
\ 2\ \text {Im}\  q_o\, \bar q|^2 \Big ],
\end{equation}
for some fixed $(q_o,\omega_o)\in \QH$ and constants $c_0>0$ and $\sigma\in \mathbb{R}$. Furthermore,
\begin{equation}\label{e:scal for qc-einstein conf to flat}
S_{\Theta}=128n(n+2)c_0\sigma
\end{equation}
\end{thrm}
The proof follows from a careful reading of the proof of \cite[Theorem 1.1]{IMV} and making of the necessary changes.  As in \cite[Theorem 1.1]{IMV},  $h$ satisfies a system of partial differential equations whose solution is a family of polynomial of degree four.

The final known general result concerns the seven dimensional case {while  the higher dimensions are settled in the preprint \cite{IMV15a}}.

\begin{thrm}[\cite{IMV1}]\label{t:div formula} {\ If a quaternionic contact
structure   $(M^7,\eta)$ is conformal to a qc-Einstein structure $(M^7,\bar\eta)$}, $%
\bar\eta\ =\ \frac{1}{2h}\, \eta$ so that $S=\bar S=16n(n+2)$, then $(M^7,\eta)$ is also qc-Einstein.
\end{thrm}

The above results lead to a complete solution of the qc Yamabe problem on the standard qc \emph{seven dimensional} sphere and quaternionic Heisenberg groups. In particular, as conjectured in \cite{GV}, all solutions of the qc Yamabe equations are given by those that realize the Yamabe constant of the sphere or the best constant in the Folland-Stein inequality.

\begin{thrm}[\cite{IMV1}]\label{t:Yamabe} 
\begin{enumerate}[a)] \item  Let $\tilde\eta=\frac{1}{2h}\eta$ be a conformal
deformation of the standard qc-structure $\tilde\eta$ on the quaternionic
unit sphere $S^{7}$. If $\eta$ has constant qc-scalar curvature, then up to
a multiplicative constant $\eta$ is obtained from $\tilde\eta$ by a
conformal quaternionic contact automorphism. In particular, $\Upsilon(S^7)=
48\, (4\pi)^{1/5}$ and this minimum value is achieved only by $\tilde\eta$
and its images under conformal quaternionic contact automorphisms.
 \item On the the seven dimensional quaternionic Heisenberg group the only solutions of the qc-Yamabe equation, up to translations \eqref{translation} and dilations \eqref{scaling}, are those given in \eqref{e:FS extremal fns}.
\end{enumerate}
\end{thrm}

The proof of Theorem~\ref{t:Yamabe} relies on Theorem~\ref{t:einstein preserving} and  Theorem~\ref{t:div formula} and will be sketched near the end of the Section.

\subsection{The Yamabe problem on a 7-D qc-Einstein manifold. Proof of Theorem \ref{t:div formula}}

In this section we give some details on the proof of Theorem \ref{t:div formula}.  The analysis involves a number of intrinsic to the structure vector fields/ 1-forms which are defined in any dimension, so here $n\geq 1$. We shall consistently keep the notation introduced in  \cite{IMV1}, which can be consulted for details.
\subsubsection{Intrinsic vector fields and their divergences}\label{ss:v.f. and div any n}

We begin by defining the horizontal 1-forms $A_s$, also letting  $ A=A_1+A_2+A_3$,
\begin{equation}\label{d:A_s}
A_i(X)=\omega_i([\xi_j,\xi_k],X).
\end{equation}
The contracted Bianchi identity on a $(4n+3)$-dimensional qc manifold with
constant qc-scalar curvature reads \cite{IMV}, Theorem 4.8],
\begin{equation}  \label{div:To}
\nabla^*T^0=(n+2){A}, \qquad \nabla^*U=\frac{1-n}{2}{A}.
\end{equation}
Let $h$ be a positive smooth function on a qc manifold
$(M, g, \eta)$
and $\bar\eta\ =\ \frac{1}{2h}\, \eta$ be a conformal deformation of the qc
structure $\eta$. As usual, the objects related to $\bar \eta$ will be denoted by an over-line. Thus,
$$d\bar\eta=-\frac1{2h^2}dh\wedge\eta+\frac1{2h}d\eta,\quad \bar g=\frac1{2h}g.$$ The new  Reeb vector fields  $\{\bar\xi_1,\bar\xi_2,\bar\xi_3\}$ are $\xi_s=2h\xi_s+I_s\nabla h, s=1,2,3.$
The components  $T^0$, $U$ of the Biquard connection and the qc scalar curvatures change as follows \cite{IMV1}
\begin{equation}\label{Torh}
\bar T^0(X,Y)=T^0(X,Y)+\frac1{4h}\Big(3\nabla^2h(X,Y)-\sum_{s=1}^3\nabla^2h(I_sX,I_sY)\Big)
-\frac1{2h}\sum_{s=1}^3dh(\xi_s)\omega_s(X,Y).
\end{equation}
\begin{multline}\label{defU}
\bar U(X,Y)=U(X,Y)+\frac1{8h}\Big(\nabla^2h(X,Y)+\sum_{s=1}^3\nabla^2h(I_sX,I_sY)\Big)\\
-\frac1{4h^2}\Big(dh(X)dh(Y)+\sum_{s=1}^3dhI_sX)dh(I_sY)\Big)-\frac1{8h}\Big(\triangle h-\frac2{h}|\nabla h||^2\Big)g(X,Y),
\end{multline}
\begin{equation}\label{defs}
\bar S=2hS+8(n+2)\triangle h-8(n+2)^2h^{-1}{|\nabla h|}^2.
\end{equation}
{Suppose  $\bar\eta$ is a positive 3-Sasakian structure, i.e. $\bar T^0=\bar U=0, \bar S=16n(n+2)$. Then  \eqref{defs} takes the form
\begin{equation}\label{defsy}
2n=4nh+\triangle h-(n+2)h^{-1}{|\nabla h|}^2,
\end{equation}
which is the qc Yamabe equation. We also  have
the formulas
\begin{multline}  \label{e:A_s}
A_1(X)\ =\ -\frac12 h^{-2}dh(X)\ -\ \frac 12h^{-3}\lvert \nabla h
\rvert^2dh(X)
-\ \frac 12 h^{-1}\Bigl (\ {\nabla dh} (I_2X, \xi_2)\ +\ \ {\nabla dh}
(I_3X, \xi_3) \Bigr )\\ +\ \frac 12 h^{-2}\Bigl (dh(\xi_2)\,dh (I_2X)\ +\
dh(\xi_3)\,dh (I_3X) \Bigr )
+\ \frac 14 h^{-2}\Bigl ( \ {\nabla dh} (I_2X, I_2 \nabla h)\ +\ \ {\nabla dh%
} (I_3X, I_3 \nabla h) \Bigr ).
\end{multline}
The expressions for $A_2$ and $A_3$ can be obtained from the above formula
by a cyclic permutation of $(1,2,3)$. Thus, we obtain
\begin{multline}  \label{e:A}
A(X)\ =\ -\frac32 h^{-2}dh(X)\ -\ \frac 32h^{-3}\lvert \nabla h \rvert^2dh(X)
-\ h^{-1}\sum_{s=1}^3\ {\nabla dh} (I_sX, \xi_s)\ \\+\
h^{-2}\sum_{s=1}^3dh(\xi_s)\,dh (I_sX)\ +\ \frac 12 h^{-2}\sum_{s=1}^3\ {%
\nabla dh} (I_sX, I_s \nabla h).
\end{multline}
We  need the divergences of various vector/1-forms defined above in addition to a few more.
 We recall that
an orthonormal frame 
$\{e_1,e_2=I_1e_1,e_3=I_2e_1,e_4=I_3e_1,\dots,
e_{4n}=I_3e_{4n-3}, \xi_1, \xi_2, \xi_3 \}$
is called
qc-normal frame at a point of a qc manifold  if the connection 1-forms of the
Biquard connection vanish at that point. As shown in \cite{IMV}, see also \cite[Lemma 6.2.1]{IV2}, a qc-normal frame exists at each point of a qc manifold. If $\sigma$ is a horizontal 1-form, then with respect to a qc-normal frame,  the divergence of $I_s\sigma$, ( $I_s\sigma(X) = -\sigma(I_sX)$) is given by
\begin{equation*}
\nabla^* (I_s\sigma)\ \ =\ - \sum_{a=1}^{4n}(\nabla_{e_a} \sigma)(I_se_a).
\end{equation*}

With  some  calculations  using \eqref{e:A}, \eqref{e:A_s} and the properties of the torsion and curvature of the Biquard connection, we obtain
\begin{equation}\label{diverAA}
\begin{aligned}
 \nabla^*\, \Bigl (\sum_{s=1}^3 dh(\xi_s) I_sA_s\Bigr )=
\sum_{s=1}^3\sum_{a=1}^{4n} \ \nabla dh\,(I_s e_a, \xi_s)A_s(e_a),\\ \nabla^*\, \Bigl
(\sum_{s=1}^3 dh(\xi_s) I_sA \Bigr )= \sum_{s=1}^3 \sum_{a=1}^{4n}\ \nabla dh\,(I_s e_a,
\xi_s)A(e_a). \end{aligned}
\end{equation}
We define  the following one-forms for $s=1,2,3,$
\begin{equation}  \label{d:D_s}
\begin{aligned}
D_s(X)\ =\ - \frac1{2h}\Big[T^0(X,\nabla h)+T^0(I_sX,I_s\nabla h)\Big],
D\ =\ - \frac1{h}\,T^{0}(X,\nabla h),
F_s(X)\  =\ - \frac1{h}\, {T^0}(X,I_s\nabla h).
\end{aligned}
\end{equation}
Using the fact  that the tensor $T^0$ belongs to the [-1]-component we  obtain from  \eqref{d:D_s}
\begin{equation}  \label{d:D_s1}
\begin{aligned}
D\ =\ D_1\ +\ D_2\ +\ D_3,\qquad
F_i(X)\ =\ -D_i(I_iX)\ +\ D_j(I_i X)\ +\ D_k(I_iX),
\end{aligned}
\end{equation}
where  $(ijk)$ is a cyclic permutation of (1,2,3).
As a consequence of \eqref{div:To}, \eqref{e:A_s}, \eqref{e:A}, the qc Yamabe equation \eqref{defsy} and \eqref{Torh} taken with $\bar A=0$, we obtain after some calculations, see  \cite{IMV1} for details, the following theorem.

\begin{lemma}[\cite{IMV1}]
\label{l:div of D} Suppose $(M, \eta)$ is a quaternionic contact manifold
with constant {\ qc-}scalar curvature $S=16n(n+2)$. {Suppose $%
\bar\eta=\frac1{2h}\eta$ has vanishing $[-1]$-torsion component $\overline
T^0=0$}. Then we have
\begin{equation*}
D(X)\ =\ \frac 14 h^{-2}\Bigl (3\ {\nabla dh} (X, \nabla h)\ -\
\sum_{s=1}^3\ {\nabla dh} (I_sX, I_s\nabla h) \Bigr )+\
h^{-2}\sum_{s=1}^3dh(\xi_s)\,dh (I_sX).
\end{equation*}
The divergence of $D$ is given by
$\nabla^*\, D\ =\ \lvert T^0 \rvert^2\ -h^{-1}\sum_{a=1}^{4n}dh(e_a)D(e_a)\ -\ h^{-1}
(n+2)\,\sum_{a=1}^{4n}dh(e_a)A(e_a),$
while the divergence of $\sum_{s=1}^3 dh(\xi_s) F_s$ is
\begin{multline*}
\nabla^*\, \Bigl (\sum_{s=1}^3 dh(\xi_s) F_s\Bigr )\ =\ \sum_{s=1}^3 \sum_{a=1}^{4n}\Bigl [%
\ \nabla dh\, (I_se_a,\xi_s)F_s(I_se_a)\Bigr] \\
+ \ h^{-1}\sum_{s=1}^3 \sum_{a=1}^{4n}\Bigl[dh(\xi_s)dh (I_se_a)D(e_a)\ +(n+2)\,dh(\xi_s)dh
(I_s e_a)\, A(e_a)\Bigr ].
\end{multline*}
\end{lemma}

\subsubsection{Solution of the qc-Yamabe equation in 7-D}\label{ss:7D QC Heisenebrg Yamabe} At this point we restrict our considerations to the 7-dimensional case turn to the proof of a key divergence formula motivated by the Riemannian and CR cases of the considered problem. As in the CR case \cite{JL3}, the Bianchi identities \cite[Theorem 4.8]{IMV}  are not enough for the proof, unlike what happens in the Riemannian case as we saw in the proof of Theorem \ref{t:Obata Yamabe}.

In fact, the proof of Theorem~\ref{t:div formula} follows by an   integration of the following divergence formula \eqref{e:div formula}, which implies  $T^0=0$. In dimension  seven the tensor $U$ vanishes identically, $U=0$ , and \eqref{sixtyfour} yields the claim. Thus, the crux of the proof of Theorem \ref{t:div formula} is the next formula, in which for $
f\ = \ \frac 12\ +\ h\ +\ \frac 14 h^{-1}\lvert \nabla h \rvert^2$,
 the following identity holds true
\begin{equation}\label{e:div formula}
\nabla^*\Big(fD\ +\ \sum_{s=1}^3 dh(\xi_s)\, F_s \ +\ 4\sum_{s=1}^3
dh(\xi_s)I_sA_s \ -\ \frac {10}{3}\sum_{s=1}^3 dh(\xi_s)\,I_s A \Big) =
f||T^0||^2 + hVLV^t.
\end{equation}
Here, $L$ is the following  positive semi-definite matrix
\begin{equation*}
L=\left[ {%
\begin{array}{cccccc}
2\, & 0 & 0 & {\displaystyle\frac{10}{3}}\, & -{\displaystyle\frac{2}{3}}\,
& -{\displaystyle\frac{2}{3}}\, \\[2ex]
0 & 2\, & 0 & -{\displaystyle\frac{2}{3}}\, & {\displaystyle\frac{10}{3}}\,
& -{\displaystyle\frac{2}{3}}\, \\[2ex]
0 & 0 & 2\, & -{\displaystyle\frac{2}{3}}\, & -{\displaystyle\frac{2}{3}}\,
& {\displaystyle\frac{10}{3}}\, \\[2ex]
{\displaystyle\frac{10}{3}}\, & -{\displaystyle\frac{2}{3}}\, & -{%
\displaystyle\frac{2}{3}}\, & {\displaystyle\frac{22}{3}}\, & -{\displaystyle%
\frac{2}{3}}\, & -{\displaystyle\frac{2}{3}}\, \\[2ex]
-{\displaystyle\frac{2}{3}}\, & {\displaystyle\frac{10}{3}}\, & -{%
\displaystyle\frac{2}{3}}\, & -{\displaystyle\frac{2}{3}}\, & {\displaystyle%
\frac{22}{3}}\, & -{\displaystyle\frac{2}{3}}\, \\[2ex]
-{\displaystyle\frac{2}{3}}\, & -{\displaystyle\frac{2}{3}}\, & {%
\displaystyle\frac{10}{3}}\, & -{\displaystyle\frac{2}{3}}\, & -{%
\displaystyle\frac{2}{3}}\, & {\displaystyle\frac{22}{3}}\,%
\end{array}%
}\right]
\end{equation*}%
and $V=( D_1, D_2, D_3,A_1,A_2, A_3)$
with $A_s$, $D_s$ defined, correspondingly, in \eqref{d:A_s} and %
\eqref{d:D_s}.

We sketch the proof of \eqref{e:div formula}. Recall that in dimension seven, $n=1$, the [3]-part of the Biquard torsion vanishes identically, $U=0$. Then \eqref{defU} together with the Yamabe equation \eqref{defsy} imply
\begin{equation}\label{defUh}
\nabla^2h(X,\nabla h)+\sum_{s=1}^3\nabla^2h(I_sX,I_s\nabla h)- (2-4h+3h^{-1}{|\nabla h|}^2)dh(X)=0.
\end{equation}
Combining \eqref{e:A_s}, \eqref{diverAA}, \eqref{defUh} and formulas in Lemma~\ref{l:div of D}  it is easy to check the formula of the theorem.  It is not hard to see that
the eigenvalues of $L$ are given by
\begin{equation*}
\{0,\quad 0,\quad 2\,(2+\sqrt{2}),\quad 2\,(2-\sqrt{2}),\quad 10,\quad 10\},
\end{equation*}%
\noindent which shows that $L$ is a non-negative matrix.

\subsubsection{The 7-D qc Yamabe problem on the sphere and qc Heisenberg group}
At this point we are ready to complete the  proof of Theorem \ref{t:Yamabe}. Recall, that the  Cayley transform \eqref{e:Cayley transf ctct form} is a conformal quaternionic
contact diffeomorphism, hence up to a
constant multiplicative
factor and a quaternionic contact automorphism the forms $\mathcal{C}%
_*\tilde\eta$ and $\tilde\Theta$ are conformal to each other. It follows
that the same is true for $\mathcal{C}_*\eta$ and $\tilde\Theta$. In
addition, $\tilde\Theta$ is qc-Einstein by definition, while $\eta$ and
hence also $\mathcal{C}_* \eta$ are qc-Einstein as we already observed. According to Theorem~\ref{t:einstein preserving}, up to a
multiplicative constant factor, the forms $\mathcal{C}_*\tilde\eta$ and $%
\mathcal{C}_*\eta$ are related by a translation or dilation on the
Heisenberg group. Hence, we conclude that up to a multiplicative constant, $%
\eta$ is obtained from $\tilde\eta$ by a conformal quaternionic contact
automorphism which proves the first claim of Theorem \ref{t:Yamabe}. From the
conformal properties of the Cayley transform and the existence Theorem \cite{Va1} it follows
that the minimum $\Upsilon(S^{4n+3},[\tileta])$ is achieved by a smooth 3-contact form,
{\ which due to the Yamabe equation is of constant qc-scalar curvature.}
This completes the proof of Theorem \ref{t:Yamabe} a). The proof of part b) is reduced to part a) by "lifting" the analysis to the sphere via the Cayley transform. A point, which requires some analysis is that we actually obtain  qc-structures on the whole sphere. This follows from the properties of the Kelvin transform which sends a solution of the Yamabe equation to a solution of the Yamabe equation, see \cite[Section 5.2]{IMV1} or \cite[Section 6.6]{IV2} for the details.

A similar argument will be used in the proof of Theorem \ref{t:qcLiouville}.

\subsection{The uniqueness theorem in a 3-Sasakin conformal class. }
We mention that similarly similarly to the Riemannian and CR cases it is expected that the {qc-conformal} class of a unit volume qc-Einstein {structure} contains a unique metric of constant scalar curvature, with the exception of the 3-Sasakian sphere, see Section \ref{ss:Sasaki-Einstein uniqieness Yamabe} for a comparison with the CR {case}.  What was problematic was the first step as outlined in Section \ref{ss:Sasaki-Einstein uniqieness Yamabe}, since here Theorem \ref{t:div formula} supplies  the first step  in dimension seven while the higher dimensional cases {was} open. {A proof extending the seven dimensional case can be found in \cite{IMV15a} where the reader can also find a proof of the uniqueness result.}

\section{The CR Lichneorwicz and  Obata  theorems}\label{s:CR Lichnerowicz-Obata}
In accordance with  Convention \ref{convention}, in this section  we use the non-negative definite sub-Laplacian, $\lap u=-tr^g(\nabla^2 u)$ for a function $u$ on a strictly pseudoconvex pseudohermitian  manifold $M$ with a Tanaka-Webster connection $\nabla$. Also, the divergence of a vector field is taken with a "-", hence we have $\lap u=\nabla^*(\nabla u)=-\sum_{a=1}^{2n}\nabla^2u(e_a,e_a)$ for an orthonormal basis of the horizontal space.

\subsection{The CR Lichneorwicz first eigenvalue estimate} From the sub-ellipticity of the sub-Laplacian on a strictly pseudoconvex CR manifold it follows that {on a compact manifold its spectrum is
discrete}. It is therefore natural to ask if
there is a sub-Riemannian version of Theorem \ref{t:Riem LichObata}. In fact, a CR analogue of the Lichnerowicz theorem was found by
Greenleaf \cite{Gr} for dimensions $2n+1>5$,  while the corresponding results for  $n=2$ and  $n=1$ were achieved later in \cite{LL}  and
\cite{Chi06}, respectively. As already observed in Theorem \ref{t:first eigenspace Iwasawa}, the standard Sasakian unit sphere has first eigenvalue
equal to 2n with eigenspace spanned by the restrictions of
all linear functions to the sphere, hence the following result is sharp.

\begin{thrm}[\cite{Gr,LL,Chi06}]\label{t:CR Lich}
Let $(M,\eta)$ be a compact strictly pseudoconvex pseudohermitian  manifold of dimension $2n+1$ such that for some $k_0=const>0$ we have the Lichnerowicz-type bound
\begin{equation}  \label{condm-app}
Ric(X,X)+ 4A(X,JX)\geq k_0 g(X,X), \qquad X\in H.
\end{equation}
 If $n>1$, then any eigenvalue $\lambda$ of the sub-Laplacian  satisfies  $\lambda \ge \frac{n}{n+1}k_0$.
 If $n=1$ the estimate $\lambda \ge \frac{1}{2}k_0$ holds
assuming in addition that the CR-Paneitz operator is non-negative
, i.e., for any smooth function $f$ we have
$
\int_M f\cdot Cf \vol\geq 0, 
$
where  $Cf$ is the CR-Paneitz operator.
\end{thrm}
We recall that the  fourth-order CR-Paneitz operator written in real coordinates is defined by the formula
\begin{equation*}
Cf=\sum_{a,b=1}^{2n}\nabla ^{4}f(e_a,e_a,e_{b},e_{b})+\sum_{a,b=1}^{2n}\nabla
^{4}f(e_a,Je_a,e_{b},Je_{b})
-4n\nabla^* A(J\nabla f)-4n\,g(\nabla^2 f,JA).
\end{equation*}
In view of the prominent role of the CR-Paneitz operator in the geometric analysis on a three dimensional CR manifold we pause for a moment to give an idea of several occurrences. We start with a few  definitions,
Given a function $f$ we define the one form,
\[
P_{f}(X)=\sum_{b=1}^{2n}\nabla ^{3}f(X,e_{b},e_{b})+\sum_{b=1}^{2n}\nabla
^{3}f(JX,e_{b},Je_{b})+4nA(X,J\nabla f)
\]
so we have $Cf=-\nabla ^{\ast }P$.
The CR Paneitz operator is called non-negative if
$$
\int_M f\cdot Cf \vol=-\int_MP_f(\gr) \vol \geq 0, \qquad f\in \mathcal{C}^\infty_0(M).
$$

In the three dimensional case the positivity condition is a CR invariant since it is independent of the
choice of the contact form which follows from the conformal invariance of $C$
proven in \cite{Hi93}. In the case of vanishing pseudohermitian torsion,  we have, up to a
multiplicative constant, $C=\Box_b\bar\Box_b$, where $\Box_b$ is the Kohn
Laplacian, hence the CR Paneitz operator is also non-negative. This property is in fact true for any $n>1$ which can be seen through the relation between the CR Paneitz operator and the $[1]$-component of the horizontal
Hessian $(\nabla ^{2}f)(X,Y)$ found in \cite{L1,GL88}. For this, consider the tensor $B(X,Y)$ defined by
\begin{equation*}
B(X,Y)\equiv B[f](X,Y)=(\nabla ^{2}f)_{[1]}(X,Y)
=\frac{1}{2}\left[ (\nabla
^{2}f)(X,Y)+(\nabla ^{2}f)(JX,JY)\right]
\end{equation*}
and also  the completely traceless part of $B,$
\begin{equation*}
B_{0}(X,Y)\equiv B_{0}[f](X,Y)=B(X,Y)+\frac{\triangle f}{2n} g(X,Y)-\frac{1}{2n}%
g(\nabla^2f,\omega)\,\omega (X,Y).
\end{equation*}
Then we have the formula \cite{L1,GL88},
\begin{equation}\label{l:GrLee}
\begin{aligned}
\sum_{a=1}^{2n}(\nabla_{e_{a}} B_0)(e_{a},X) & =\frac{n-1}{2n}P_f(X), \\
\int_M |B_0|^2\vol  & =-\frac{n-1}{2n}\int_M P_f(\gr)\vol =\frac{n-1}{2n}\int_M f\cdot (Cf)\vol.
\end{aligned}
\end{equation}
In particular, if $n>1$ the CR-Paneitz operator is non-negative. As an application of this result, we recall \cite{L1}, see also \cite{BF74} and \cite{Be80}, according to which if $
n\geq 2$, a function $f\in \mathcal{C}^3(M)$ is CR-pluriharmonic, i.e,
locally it is the real part of a CR holomorphic function, if and only if $%
B_0[f]=0$. By \eqref{l:GrLee} only one  fourth-order equation $Cf=0$ suffices for $B_0[f]=0$ to
hold.  When $n=1$ the
situation is more delicate. In the three dimensional case, CR-pluriharmonic
functions are characterized by the kernel of the third order operator $%
P[f]=0 $ \cite{L1}. However, the single equation $Cf=0$ is enough again
assuming the vanishing of the pseudohermitian torsion \cite{GL88}, see also \cite%
{Gra83ab}. On the other hand, \cite{CCC07} showed that if the pseudohermitian torsion
vanishes the CR-Paneitz operator is essentially positive, i.e., there is a
constant $\Lambda>0$ such that
\begin{equation*}
\int_M f\cdot (Cf) \vol \geq \Lambda\int_M f^2 \vol.
\end{equation*}
for all real smooth functions $f\in (Ker\, C)^{\perp}$, i.e., $\int_M f\cdot
\phi \vol =0$ if $C\phi=0$. In addition, the non-negativity of the
CR-Paneitz operator is relevant in the embedding problem for a three
dimensional strictly pseudoconvex CR manifold. In the Sasakian case, it is
known  that $M$ is embeddable, \cite{Le92}, and the CR-Paneitz operator is nonnegative, see \cite{Chi06}, \cite{CCC07}. Furthermore, \cite{ChChY13} showed that if the pseudohermitian scalar curvature   of $M$ is positive and $C$ is non-negative, then $M$ is embeddable in some $\mathbb{C}^n$.

After these preliminaries we are ready to sketch the proof of Theorem \ref{t:CR Lich}.


\subsubsection{Proof of the CR Lichnerowicz type estimate}\label{ss:CR Lich est}
 We shall use  real coordinates as in \cite{IVZ,IV2} and rely on the proof described in details in \cite[Section 8.3]{IVO} valid for $n\geq 1$.
Not surprisingly, a key to the solution is the CR Bochner identity due to \cite{Gr},
\begin{equation}  \label{bohh}
-\frac12\triangle |\nabla f|^2=|\nabla df|^2-g(\nabla(\triangle f),\nabla f)+Ric(\nabla
f,\nabla f)+2A(J\nabla f,\nabla f)
+ 4\nabla df(\xi,J\nabla f).
\end{equation}
The last term can be related to the traces of $\nabla^2 f$, \cite{Gr},
\[
\int_M\nabla^2f(\xi,J\nabla f)\vol=-\int_M\frac{1}{2n}g(\nabla^2f,\omega)^2
+A(J\nabla f,\nabla f) \vol\]
and also using the CR-Paneitz operator
\[
\int_M\nabla^2f(\xi,J\nabla f)\vol
=\int_{M}-\frac{1}{2n}\left(
\triangle f\right) ^{2}+A(J\nabla f,\nabla f)-\frac{1}{2n}P_f(\gr)\vol.
\]

Integrating the CR Bochner idenity (for arbitrary function $f$) and using the last two formulas for the  term $\int_M\nabla^2f(\xi,J\nabla f)\vol$ we find

\begin{multline*}
0=\int_M Ric(\nabla f,\nabla f)+4A(J\nabla f,\nabla f)-\frac {n+1}{n}(\triangle f)^2\vol\\
+\int_M \left |(\nabla^2f)\right |^2
-\frac {1}{2n}(\triangle f)^2 -\frac {1}{2n}g(\nabla^2 f,\omega)^2\vol
+\int_M\Big[-\frac {3}{2n}P(\gr)\Big]\vol.
\end{multline*}
Noticing that $\Big \{\frac{1}{\sqrt{2n}}%
g,\ \frac{1}{\sqrt{2n}}\omega\Big\}$ is an orthonormal set  in the $[1]$-space with non-zero traces, we have
 $$\left |(\nabla^2f)_{[0]}\right |^2 \overset{def}{=}\left |(\nabla^2f)\right |^2
-\frac {1}{2n}(\triangle f)^2 -\frac {1}{2n}g(\nabla^2 f,\omega)^2.$$
Let us assume at this point that  $\triangle f=\lambda f$ and the "Ricci" bound \eqref{condm-app} to obtain the inequality
\begin{equation*}
0\geq\int_M\left (k_0 -\frac {n+1}{n}\lambda\right )|\gr|^2\vol+ \int_M\left
|(\nabla^2f)_{[0]}\right |^2\vol
-\frac {3}{2n}\int_MP_f(\gr)\vol,
\end{equation*}
which implies
$
\lambda \ge \frac{n}{n+1}k_0$ with equality holding iff
\begin{equation}\label{e:CR equality hessian}
\nabla^2f= \frac {1}{2n}(\triangle f)\cdot g+\frac {1}{2n}g
(\nabla^2f,\omega)\cdot\omega
\end{equation}
and $\int_M P_f(\gr)\vol=0$ taking into account  the extra assumption for $n=1$. The proof of Theorem \ref{t:CR Lich} is complete.

\subsection{The CR Obata type theorem }\label{ss:CR Obata}

\begin{thrm}[\cite{IVO} ]\label{main11} Let $(M, \theta)$ be a  strictly pseudoconvex
pseudohermitian CR manifold of dimension $2n+1$ with a
divergence-free  pseudohermitian torsion, $\bi^*A=0$. Assume, further, that $M$ is
complete with respect to the associated Riemannian metric \eqref{hmetric}.
 If $n\geq 2$ and there is a smooth function $f\not\equiv 0$ whose Hessian with respect
to the Tanaka-Webster connection satisfies
\begin{equation}\label{e:hessian}
 \nabla ^{2}f(X,Y)=-fg(X,Y)-df(\xi )\omega (X,Y), \qquad X, Y \in H=Ker\, \theta,
\end{equation}
then up to a scaling of $\theta$ by a positive constant
$(M,\theta)$ is the standard (Sasakian) CR structure on the unit
sphere in $\mathbb{C}^{n+1}$. In dimension three, $n=1$,
  the above result holds provided the pseudohermitian torsion vanishes, $A=0$.
\end{thrm}
This is the best known result   for a complete non-compact $M$ unlike the Riemannian and QC cases where the corresponding results are valid without any conditions on the torsion, see the paragraph after \eqref{e:Riem Hess eqn} and Theorem \ref{main2}. It should be noted that besides the Sasakian condition, when $n=1$, one can invoke assumptions such as the vanishing of the divergence of the torsion, the vanishing of the CR-Paneitz operator or the equality case in \eqref{condm-app}.  We insist on the strongest assumption, which avoids a lot of the technicalities which appear when a combination of these assumptions are made while still achieving a (probably) non-optimal result. Results of this nature can be found by combining identities proven in \cite{IVO}. In the compact case,  with the help of a clever integration argument \cite{LW,LW1} were able to complete the arguments of \cite{IVO,IV3} and remove the assumption of divergence free torsion $\bi^*A=0$ for $n\geq 2$, while the case $n=1$ was completed in \cite{IV3}. Taking into account  \eqref{e:CR equality hessian}, a consequence of these results is the Obata type theorem characterizing the case of equality in Theorem \ref{t:CR Lich}. We note that  Theorem \ref{main11} actually shows that in the compact Sasakian case \eqref{e:hessian} characterizes the unit Sasakian sphere. This fact in addition to other results of \cite{IVO} reappeared in \cite{LW1}.

\begin{thrm}[\cite{LW,LW1,IV3}]\label{t:CR Obata} Suppose  $(M, J,\eta)$, $\dim M=2n+1$, is a compact strictly pseudoconvex pseudohemitian manifold which satisfies the Lichnerowicz-type bound \eqref{condm-app}. If $n\geq 2$, then  $\lambda = \frac{n}{n+1}k_0$ is an eigenvalue  iff up-to a scaling  $(M, J,\eta)$ is the
standard pseudo-Hermitian CR structure on the unit sphere in $\mathbb{C}^{n+1}$.
If $n=1$
the same conclusion holds assuming in addition that  the CR-Paneitz operator is non-negative,  $C\geq 0$.
\end{thrm}

Some earlier papers which contributed to the  proof of the above Theorem include  S.-C. Chang \& H.-L. Chiu who proved the above Theorem in the Sasakian case for $n\geq 2$ in \cite{CC09a} and for $n=1$ in \cite{CC09b}. The non-Sasakian case, was considered by Chang, S.-C., \& C.-T. Wu in \cite{ChW2}
assuming
    $ A_{\alpha\beta,\, \bar\beta}=0$ for $n\geq 2$ and $A_{\alpha\beta,\, \gamma\bar \gamma}=0$ and $A_{11,\, \bar 1}=0$ for $n=1$.
$P_1 f=0$.

Let us give an idea of the proof of Theorem \ref{main11} following \cite{IV3}. The first step is to show the vanishing of the Webster torsion $A$. We shall make clear where the cases $n=1$ and $n>1$ diverge. Using the Ricci identity
\begin{multline}\label{e:CR XYxi Ricci}
\nabla^3 f(X,Y,\xi)=\nabla^ 3 f (\xi,X,Y)+\nabla^2f (AX,Y)+\nabla^2f (X,AY)
 +(\nabla_X A)(Y,\nabla f)\\
 +(\nabla_Y A)(X,\nabla f)-(\nabla_{\nabla f}) A( X,Y).
\end{multline}
in which we substitute the term $\nabla^ 3 f (\xi,X,Y)$ by its expression obtained after differentiating \eqref{e:hessian} we come to the next equation \cite[(3.3)]{IVO},
\begin{multline}
\nabla ^{3}f(X,Y,\xi )=-df(\xi )g(X,Y)-(\xi ^{2}f)\omega (X,Y)-2fA(X,Y)
\label{e:D3f extremal bis}
+(\nabla_X A)(Y,\nabla f)+(\nabla_Y A)(X,\nabla f)\\
-(\nabla_{\nabla f}
A)(X,Y),
\end{multline}
With the help of the Ricci identities, \eqref{e:hessian}, \eqref{e:CR Ricci 2-from} and \eqref{e:CR Ric type decomp} we obtain  a formula for $R(X,Y,Z,\nabla f)$, \cite[(4.1)]{IVO}
 \begin{multline}\label{eqc1}
R(Z,X,Y,\nabla f)=\Big[df(Z)g(X,Y)-df(X)g(Z,Y)\Big]+\nabla df(\xi ,Z)\omega
(X,Y)
-\nabla df(\xi ,X)\omega (Z,Y)\\-2\nabla df(\xi ,Y)\omega (Z,X)+A(Z,\nabla
f)\omega (X,Y)
-A(X,\nabla f)\omega (Z,Y),
\end{multline}%
which after taking traces  gives identities for $Ric(X,\gr)$ and $Ric(JX,J\gr)$, \cite[(4.2)]{IVO},
      \begin{equation}\label{eqc02}
\begin{aligned}
& Ric(Z,\nabla f)=(2n-1)df(Z)-A(JZ,\nabla f)-3\nabla
df(\xi,JZ)\\
& Ric(JZ,J\nabla f)=df(Z)-(2n-1)A(JZ,\nabla f)-(2n+1)\nabla df(\xi,JZ).
\end{aligned}
\end{equation}
Note that when $n=1$, $Ric(X,Y)=Ric(JX,JY)$, hence the identities for $Ric(X,\gr)$ and $Ric(JX,J\gr)$ \emph{coincide}, which is the reason for the assumption $n>1$ when $A\not=0$.
 For $n>1$, taking  the $[-1]$ part of $R(.,.,X,Y)$ it follows
\begin{equation}\label{e:vhessian}
\nabla ^{2}f(Y,\xi
)=df(JY)+2A(Y,\nabla f).
\end{equation}
Using the formula for the curvature \eqref{eqc1} we come to
  $|\nabla f|^{2}A(Y,Z) =df(Y)A(\nabla f,Z)-df(JY)A(\nabla f,JZ)$
  found in the proof of \cite[Lemma 4.1]{IVO}. Hence,  the Webster  torsion is determined by $A(J\gr,\gr)$ as follows
  \begin{equation}\label{e:A formula}
  |\gr|^4A(X,Y)=-A(J\gr,\gr)\left [df(X)df(JY)+df(Y)df(JX) \right],
  \end{equation}
  which implies in particular,
  $A(\gr,\gr)=0$.
On the other hand, from \eqref{e:vhessian} we have  \cite[(4.9)]{IVO},
\begin{multline}  \label{e:D3f extremal}
\nabla ^{3}f(X,Y,\xi ) =-df(\xi )g(X,Y)+f\omega (X,Y)-2fA(X,Y)-2df(\xi
)A(JX,Y)
+2\nabla A(X,Y,\nabla f).
\end{multline}
For $n>1$, equations \eqref{e:D3f extremal} and \eqref{e:D3f extremal bis} imply the identity, see the formula in the proof of \cite[Lemma 4.3]{IVO},
\begin{multline}\label{e:old key identity}
2df(\xi )A(JX,Y)-(\nabla_{\nabla f} A)(X,Y)=(\xi ^{2}f)\omega (X,Y)+f\omega (X,Y)+(\nabla_X A)(Y,\nabla f) -(\nabla_Y A)(X,\nabla
f).
\end{multline}
 Notice that the left-hand side is symmetric why the right-hand is skew-symmetric, hence they both vanish,
 \begin{equation}\label{e:CR key identity}
 (\nabla_{\nabla f} A)(X,Y)=2df(\xi)A(JX,Y)\qquad \text{and }\qquad (\nabla_X A)(Y,\nabla f) =(\nabla_Y A)(X,\nabla f),
 \end{equation}
 taking into account $\nabla ^{2}f(\xi ,\xi )=-f+\frac 1n(\nabla^* A)(J\nabla f)=-f$, when $\bi^*A=0$, which follows by taking a trace in the (vanishing) right-hand side of \eqref{e:old key identity}, see \cite[Lemma 4.3]{IVO}.
 \begin{rmrk}\label{r:CR cpct key}
 Notice that the first identity implies  $g({\gr},\nabla|A|^2)=0$.
 \end{rmrk}
The first equation of \eqref{e:CR key identity} gives $(\nabla_{\nabla f} A)(J\gr,\gr)=-2df(\xi)A(\gr,\gr)=0$ as we already showed above, see \eqref{e:A formula}. Finally, differentiating the identity $A(\gr,\gr)=0$ and using \eqref{e:hessian} we obtain $(\nabla_XA)(\gr,\gr)=2fA(\gr,X)-2df(\xi)A(J\gr,X)$ which shows $(\nabla_\gr A)(\gr,\gr)=-2df(\xi)A(J\gr,\gr)$. Therefore, $A(J\gr,\gr)=0$ which implies $|\gr|^4 A=0$. In order to conclude that $A=0$ we need to know that $f$ cannot be a local constant. For this and other facts we turn to the next step of the proof, where we show that $f$ satisfies an elliptic equation for which we can use a unique continuation argument. We remark that the corresponding sub-elliptic result seems to be unavailable.

 Next, we observe that if $f$ satisfies \eqref{e:hessian}, then $f$ satisfies an  elliptic equation \cite[Corollary 4.5 \& Lemma 5.1]{IVO},
\begin{align}\label{e:lap n >1}
 \triangle^h f & =\triangle f- \nabla ^{2}f(\xi ,\xi )=(2n+1)f-\frac 1n(\nabla^* A)(J\nabla f), \qquad \text{ if }\  n>1, \\\label{e:lap n=1}
 \triangle^h f & = \left ( 2+ \frac{S-2}{6} \right ) f -\frac{{1}}{12}g(\gr,\nabla S) +\frac 13(\nabla^* A)(J\nabla f),  \qquad \text{ if }\  n=1,
\end{align}
where $\triangle^h $ is the Riemannian Laplacian  associated to  the Riemannian
metric $h=g+\eta^2$
on $M$. In particular, $f$ cannot be a local constant. The equations follows from the formula relating the Levi-Civita and the Tanaka-Webster connections, see \cite[Lemma~1.3]{DT} and \cite[(4.15)]{IVO}, which shows
\begin{equation}\label{obsa}
-\triangle^h f =-\triangle f+ (\xi^2 f).
\end{equation}
When $n>1$ equation \eqref{e:lap n >1} follows taking into account (the line after) equation \eqref{e:CR key identity}. The case $n=1$ requires some further calculations for which we refer to \cite[Lemma 5.1]{IVO}.

The final step of the prove of Theorem~\ref{main11} is a reduction to the corresponding Riemannian  Obata theorem on a complete Riemannian manifold. In fact, we will show that the Riemannian Hessian computed with respect to the Levi-Civita connection $D$ of the metric $h$  satisfies \eqref{e:Riem Hess eqn}
and then apply the Obata theorem  to conclude that $(M,h)$ is isometric to the unit sphere. We should mention the influence of \cite{CC09a,CC09b}  where  the compact Sasakian case is reduced  to Theorem \ref{t:Riem LichObata}.

For $n>1$ where we proved that  $A=0$ we showed the validity of the next two identities
\begin{equation}\label{e:vhessianc}
 \nabla ^{2}f(\xi ,Y)=\nabla ^{2}f(Y,\xi
)=df(JY), \qquad \xi^2f=-f.
\end{equation}
Next, we show that \eqref{e:vhessianc} also holds in dimension three when the pseudohermitian torsion vanishes. In the three dimensional case we have $Ric(X,Y)=\frac{S}2g(X,Y)$. After a substitution of this  equality in \eqref{eqc02}, taking into account $A=0$, we obtain
\begin{equation}\label{hes3}
\bi^2f(\xi,Z)=\bi^2f(Z,\xi)=\frac{(S-2)}6df(JZ).
\end{equation}
Differentiating \eqref{hes3} and using \eqref{e:hessian} we find
\begin{equation}\label{hes31}
\bi^3f(Y,Z,\xi)=\frac16\Big[dS(Y)df(JZ)+(S-2)f\omega(Y,Z)\Big]
-\frac 16(S-2)df(\xi)g(Y,Z).
\end{equation}
On the other hand, setting $A=0$ in \eqref{e:D3f extremal bis}, we have
\begin{equation}\label{hes32}
\bi^3f(Y,Z,\xi)=-df(\xi)g(Y,Z)-(\xi^2f)\omega(Y,Z).
\end{equation}
In particular, the function $\xi f$ also satisfies \eqref{e:hessian}. From \eqref{e:lap n=1} using again unique continuation $\xi f\not=0$ almost everywhere since otherwise $\nabla f=0$ taking into account \eqref{hes3}, hence $f\equiv 0$, which is not possible by assumption.
Now,  \eqref{hes31} and \eqref{hes32} give
\begin{equation}\label{hes33}
\frac{S-8}6df(\xi)g(Y,Z)-\Big(\xi^2f+\frac{S-2}6\Big)\omega(Y,Z)-\frac16dS(Y)df(JZ)=0,
\end{equation}
which implies
$
\frac{S-8}3df(\xi)|\gr|^2=0.
$ 
 Thus, the pseudohermitian scalar curvature is constant, $S=8$, invoking again the local non-constancy. Equation  \eqref{hes33} reduces then to
$\Big(\xi^2f+\frac{S-2}6\Big)\omega(Y,Z)=0$
since $dS=0$ which yields
$\xi^2f=-f.$
The latter together with \eqref{hes3} and $S=8$ imply the validity of \eqref{e:vhessianc}  in dimension three.

Finally, we use the relation between $D$ and $\bi$,  \cite[Lemma~1.3]{DT} which in the case $A=0$ simplifies to
\begin{equation}\label{wh}
D_BC=\bi_BC+\theta(B)JC+\theta(C)JB-\omega(B,C)\xi, \quad B,C\in \mathcal{T}(M),
\end{equation}
where $J$ is extended with $J\xi=0$.  Using \eqref{wh} together with \eqref{e:hessian} and \eqref{e:vhessianc} we calculate easily  that
\eqref{e:Riem Hess eqn} holds. The proof of Theorem~\ref{main11} is complete.

\subsection{Proof of the Obata CR eigenvalue theorem in the compact case}\label{ss:CR cpct Obata proof}
Turning to Theorem \ref{t:CR Obata} we mention that when
 $n=1$ we need to find an alternative way to  the 'missing" equation \eqref{e:vhessian}, see the remark after \eqref{eqc02}. In fact, in \cite[Lemma 5.1]{IV3}] it was shown that in this case (assuming  the Lichnerowicz' type condition), if $\triangle f=2f$ then we have $A(\gr,\gr)=0$ and \eqref{e:vhessian} holds true. This is proved with an integration (using the compactness!) of the "vertical Bochner" formula \cite[Remark 3.5]{IVO}
 \begin{equation}  \label{e:vertical Bochner}
-\triangle (\xi f)^2 = 2|\nabla(\xi f)|^2-2df(\xi)\cdot \xi (\triangle f) +
4df(\xi)\cdot g(A,\nabla^2 f)
-4df (\xi)(\nabla^*A)(\nabla f).
\end{equation}
At this point, identities \eqref{eqc02}-\eqref{e:CR key identity} are available for $n\geq 1$, which imply in particular the identity in Remark \ref{r:CR cpct key} holds true. We come to the idea of \cite[Lemma 4]{LW1} where  integration by parts  involving suitable powers of $f$ are used in order to conclude $A=0$. By Remark \ref{r:CR cpct key}  we have for any $k>0$
\begin{equation}\label{e:CR1}
g(\gr,\nabla|A|^k) =0,
\end{equation}
while the Lichnerowicz' condition implies \textit{point-wise}  the inequality $A(\gr,J\gr)\leq 0,$ hence
 \begin{equation}\label{e:CR2}
 |\gr|^2 |A|=-\sqrt{2} A(\gr, J\gr).
 \end{equation}
Next we shall use an integration by parts argument similar to \cite{LW1} for the case $n=1$, i.e.,  $\triangle f=2f$.
Using \eqref{e:CR1}, we have
$${I\overset{def}{=}\int_M |A|^3f^{2(k+1)}\vol}=-\frac {1}{2}\int_M |A|^3f^{2k+1}\triangle f \vol=\frac {2k+1}{2}{\int_M |A|^3f^{2k}|\gr|^2} \vol\equiv\frac {2k+1}{2}{D}.$$
From \eqref{e:CR2} it follows
\begin{multline*}
\sqrt{2}(2k+1)D
\overset{def}{=}-\int_M|A|^2f^{2k+1}(\nabla^* A)(J\gr)\vol
 \leq ||\nabla^* A||\int_M |A|^2f^{2k+1}|\gr|\vol\\
 \leq \frac {||\nabla^* A||}{a}\int_M f^{k+1}\,f^k|\gr|\, |A|^3\vol,
 \end{multline*}
 assuming  $|A|\geq a>0$ so $|A|^2\leq\frac {1}{a}|A|^3$.
Now, H\"older's inequality gives
\begin{multline*}
\sqrt{2}(2k+1)D
\leq \frac {||\nabla^* A||}{a}\Big ({\int_M |A|^3f^{2(k+1)}\vol}\Big)^{1/2}\, \Big( {\int_M |A|^3f^{2k}|\gr|^2\vol } \Big)^{1/2}\\
=\frac {||\nabla^* A||}{a} \, \Big( \frac {2k+1}{2}D\Big)^{1/2}\,D^{1/2}=\frac {||\nabla^* A||}{a} \Big( \frac{2k+1}{2}\Big)^{1/2}D{.}
\end{multline*}
By taking $k$ sufficiently large { we} conclude $A=0$. The assumption $|A|\geq a>0$ can be removed by employing a suitable cutt-off function, see \cite{LW1}. Once we know that $M$ is Sasakian one applies \cite{CC09a,CC09b} where a reduction to Theorem \ref{t:Riem LichObata} is made.

\section{The Quaternionic Contact Lichnerowicz and Obata theorems}\label{s:QC Lichnerowicz-Obata}

This section concentrates on the qc versions of the Lichnerowicz and Obata eigenvalue theorems. As in the Riemannian and CR cases we are dealing with a sub-elliptic operator, hence the discreteness of its spectrum on a compact qc manifold. The  Lichnerowicz' type result was found in
\cite{IPV1} in dimensions grater than seven and in \cite{IPV2} in the seven
dimensional case. Remarkably, compare with the CR case, the Obata type theorem characterizing the 3-Sasakian sphere through the horizontal Hessian equation holds under no extra assumptions on the Biquard'
torsion {   when the dimension of the qc manifold is at least eleven} as proven in \cite{IPV3}. The
general qc Obata result in dimension seven remains open.

We shall use freely  the curvature and torsion tensors associated to a given qc structure as defined in Section \ref{ss:qc geometry}. As in the previous sections where eigenvalues were concerned we shall use the non-negative sub-Laplacian, $\lap u=-tr^g(\nabla^2 u)$.
\subsection{The QC Lichnerowicz theorem}

\begin{thrm}[\cite{IPV1,IPV2}]
\label{mainpan} Let $(M,\eta)$ be a compact QC manifold of dimension $4n+3$. Suppose, for $\alpha_n=\frac {2(2n+3)}{2n+1}, \quad \beta_n=\frac {4(2n-1)(n+2)}{(2n+1)(n-1)}$ and 
for any $ X\in H$
$$\mathcal{L}(X,X)\overset{def}{=}2 Sg(X,X)+\alpha_n T^0(X,X) +\beta_n U(X,X)\geq 4g(X,X).$$

If $n=1$, assume in addition  the positivity of
the $P$-function of any eigenfunction.
Then, any eigenvalue $\lambda$ of the sub-Laplacian $\triangle$ satisfies the
inequality
$\lambda \ge 4n.$
\end{thrm}
The  3-Sasakian sphere achieves equality in the  Theorem. The eigenspace of the first non-zero eigenvalue
of the sub-Laplacian on the unit 3-Sasakian sphere in Euclidean
space is given by the restrictions to the sphere of all linear
functions by Theorem \ref{t:first eigenspace Iwasawa}.

\subsubsection{The QC P-function}
We turn to the definition of the QC P-function defined  in \cite{IPV2}.
 For a fixed smooth function $f$ we define a one form $P\equiv P_f \equiv P[f]$ on $M$, which we call the $P-$form of $f$, by the following equation
\begin{equation*}
P_f(X) =\sum_{b=1}^{4n}\nabla ^{3}f(X,e_{b},e_{b})+\sum_{t=1}^{3}\sum_{b=1}^{4n}\nabla
^{3}f(I_{t}X,e_{b},I_{t}e_{b})
-4nSdf(X)+4nT^{0}(X,\nabla f)-\frac{8n(n-2)}{n-1%
}U(X,\nabla f).
\end{equation*}
The $P-$function of $f$ is the function $P_f(\nabla f)$. The $C-$operator is the fourth-th order differential operator on $M$, which is independent of $f$,
$$
f\mapsto Cf =-\nabla^* P_f=\sum_{a=1}^{4n}(\nabla_{e_a} P_f)\,(e_a).
$$
We say that the $P-$function of $f$ is non-negative if
$$\int_M f\cdot Cf \, Vol_{\eta}= -\int_M P_f(\nabla f)\, Vol_{\eta}\geq 0.
$$
If the above holds for any $f\in \mathcal{C}^\infty_o\,(M)$ we say that the $C-$operator is \emph{non-negative}, $C\geq 0$.

Several important properties of the C-operator were found in \cite{IPV2}.
 The first notable fact is that the $C-$operator is non-negative, $C\geq 0$, for $n>1$. Furthermore $Cf=0$ iff $(\nabla^2f)_{[3][0]}(X,Y)=0$, {where [3][0] denotes the trace-free part of the [3]-part of the Hessian.}
In this case the $P-$form of $f$ vanishes as well.
The key for the last result is the identity $\sum_{a=1}^{4n}(\nabla_{e_a}(\nabla^2f)_{[3][0]})(e_a,X)=\frac{n-1}{4n}P_f(X)$,
hence
$$\frac{n-1}{4n}\int_Mf\cdot Cf\, Vol_{\eta}=-\frac{n-1}{4n}\int_MP_f(\nabla f)\,
Vol_{\eta}=\int_M|(\nabla^2f)_{[3][0]}|^2\, Vol_{\eta},$$
after using the Ricci identities, the divergence formula and the orthogonality of the components of the
horizontal Hessian.

In dimension seven, the condition of non-negativity of the $C-$operator is also non-void. For example, \cite{IPV2} showed that
on a 7-dimensional compact qc-Einstein manifold with $Scal\geq 0$
the $P-$function of an \emph{eigenfunction} of the sub-Laplacian is non-negative.

The proof  relies on several results.  First, the qc-scalar curvature of a qc-Einstein  is constant \cite{IMV,IMV3}. In the higher dimensions this follows from the Bianchi identities. However, the result is very non-trivial in dimension seven where the qc-conformal curvature tensor $W^{qc}$, see Theorem \ref{T:flat}, is invoked in the proof. Secondly, on a qc-Einstein {{manifold we have}   $\nabla ^{3}f(\xi _{s},X,Y)=\nabla^{3}f(X,Y,\xi _{s})$, the vertical space is integrable and we have $\nabla ^{2}f(\xi _{k},\xi _{j})-\nabla ^{2}f(\xi _{j},\xi _{k})=-Sdf(\xi _{i})$. Finally, a calculation shows $\int_M|P_f|^2\vol=-(\lambda+4S)\int_M P_f(\nabla f) \vol$, which implies the claim.

At this point we can give the main steps in the proof of the Lichnerowicz' type theorem following the $P-$function approach of \cite{IPV3} which unified the seven and higher dimensional cases of \cite{IPV1,IPV2} as we did in the CR case in the proof of Theorem \ref{t:CR Lich}.
By the QC Bochner identity {established in \cite{IPV1}}, letting $R_f=\sum_{s=1}^3\nabla^2f(\xi_s,I_s\nabla f)$, we have
\begin{multline*}
-\frac12\triangle |\nabla f|^2=|\nabla^2f|^2-g\left (\nabla (\triangle f),
\nabla f \right )+2(n+2)S|\nabla f|^2+2(n+2)T^0(\nabla f,\nabla f)
+2(2n+2)U(\nabla f,\nabla f)\\
+ 4R_f.
\end{multline*}
The "difficult" term $R_f$ can be computed in two ways. First with the help of the $P$-function we have
$$\int_M R_f Vol_{\eta}=\int_M -%
\frac{1}{4n}P_n(\nabla f)-\frac{1}{4n}(\triangle f)^2-S|\nabla f|^2 Vol_{\eta}+\int_M\frac{%
n+1}{n-1}U(\nabla f,\nabla f)\, Vol_{\eta}.$$
On the other hand,
using Ricci's identities
$
g(\nabla^2f , \omega_s) \overset{def}{=}\sum_{a=1}^{4n}\nabla^2f(e_a,I_se_a)=-4ndf(\xi_s),
$
we have
$$\int_MR_f Vol_{\eta}=-\int_M%
\frac {1}{4n}\sum_{s=1}^3 g(\nabla^2 f, \omega_s)^2 +T^0(\nabla f,\nabla f)-3U(\nabla f,\nabla f) %
\, Vol_{\eta}.$$
A substitution of a linear combination of the last two identities in the QC Bochner identity shows
\begin{multline*}
0=\int_M|\nabla^2f|^2-\frac{1}{4n}\Big[(\triangle
f)^2+\sum_{s=1}^{3}[g(\nabla^2f,\omega_s)]^2\Big]-\frac{3}{4n}P_n(\nabla f)Vol_{\eta}
\\
+ \frac{2n+1}{2}\int_M \mathcal{L}(\nabla f,\nabla f)-%
\frac{\lambda}{n}|\nabla f|^2\,Vol_{\eta}.
\end{multline*}
With the Lichnerowicz type assumption, $\mathcal{L}(\nabla f,\nabla f)\geq 4|\nabla f|^2$,  it follows
\begin{equation*}  \label{intin}
0\geq\int_M |(\nabla^2f)_{0}|^2
-\frac{3}{4n}P_n(\nabla f)Vol_{\eta}
+%
\frac{2n+1}{2n}\int_M\Big(4n -\lambda\Big)|\nabla f|^2%
Vol_{\eta}.
\end{equation*}
For $n=1$, when $U=0$ trivially, { the formulas are still correct even after removing formally}  the
torsion tensor $U$ terms, which completes the proof of Theorem \ref{mainpan}.

\subsection{The QC Obata type theorem}
\begin{thrm}[\cite{IPV3}]
\label{main2} Let $(M,\eta)$ be a quaternionic contact manifold
of dimension $4n+3>7$ which is complete with respect to the associated Riemannian
metric
$
h=g+(\eta_1)^2+(\eta_2)^2+(\eta_3)^2.
$ 
There exists a smooth $f\not\equiv$const, such that,
$$
\nabla df(X,Y)=-fg(X,Y)-\sum_{s=1}^{3}df(\xi _{s})\omega _{s}(X,Y{)}
$$
if and only if the qc manifold $(M,\eta,g,\mathbb{Q})$ is qc homothetic to the unit  3-Sasakian sphere.
\end{thrm}
It should be noted that {in dimension seven the problem is still open}.
The above theorem suffices to charcaterize the cases of equality in Theorem \ref{mainpan} for $n>1$.
\begin{thrm}[\cite{IPV3}]
Let $(M,\eta)$ be a compact QC manifold of dimension $4n+3$ which satisfies a Lichnerowicz' type bound  $\mathcal{L}(X,X)\geq 4g(X,X)$. Then, there is a function $f$ with $\triangle f=4n f$
if and only if
\begin{itemize}
\item when $n>1$,  $M$ is qc-homothetic to the 3-Sasakian sphere;
\item when $n=1$, and $M$ is qc-Einstein, i.e., $T^0=0$, then $M$ is qc-homothetic to the 3-Sasakian sphere.
\end{itemize}
\end{thrm}

{Next, we give an outline of the key steps in the proof of Theorem \ref{main2}.}

\textbf{Part 1}.   The first step is to show  that $T^0=0$ and $U=0$, i.e., $M$ is qc-Einstein. This is achieved by the following argument. First we determine the remaining parts of the Hessian of $f$  with respect to  the Biquard connection in terms of the torsion tensors. A simple argument shows that $T^0(I_s\gr,\gr)=U(I_s\gr,\gr)=0$. Using the $[-1]$-component of the curvature tensor it follows $T^0(I_s\gr,I_t\gr)=0$, $s,\, t\in \{1,2,3\}$, $s\not=t$. Then we determine the torsion tensors $T^0$ and $U$ in terms of $\nabla f$ and the tensor $U(\gr,\gr)$. For example, $$
|\nabla f|^{4}T^{0}(X,Y)=-\frac{2n}{n-1}U(\nabla f,\nabla f)\Big[%
3df(X)df(Y)-\sum_{s=1}^{3}df(I_{s}X)df(I_{s}Y)\Big].$$
Next, we prove formulas of the same type for  $\nabla T^0$ and $\nabla U$. In particular we have $$
(\nabla _{\nabla f}U)(X,Y)=\frac{2(n-1)}{n+2}fU(X,Y).$$

\begin{rmrk}\label{r:qc obata cpct using riem}
We pause for a moment to remark that the last equation shows in particular $L_{\gr}|U|^2=\frac{4(n-1)}{n+2}f|U|^2$ \emph{as in the Riemannian case for $Ric_0$}. Hence, in the compact case we can use an integration as in Proposition \ref{p:Obata Einstein} to see the vanishing of $U$, hence {also} of $T^0$ by what we have proved.
\end{rmrk}
By what we already proved, the crux of the matter is the proof that $U(\gr,\gr)=0$ (or $T^0(\gr,\gr)=0$). This fact is achieved with the help of the Ricci identities, the contracted Bianchi second identity and many properties of the torsion of a qc-manifolds: $0=\nabla ^{3}f(\xi_i,I_i\nabla f,\nabla f)-\nabla ^{3}f(I_{i}\nabla f,\nabla f,\xi _{i})=\frac{2}{n+2}fU(\nabla f,\nabla f)$. We finish {taking into account} that $|\gr|\not=0$ a.e. using a unique continuation {argument} by showing that on  a qc manifold with $n>1$, the "horizontal Hessian equation" implies that $f$ satisfies an elliptic partial differential equation,
\begin{equation}  \label{llex}
\triangle^h f=(4n+3)f+\frac{n+1}{n(2n+1)}(\nabla_{e_a}T^0)(e_a,\nabla f)+%
\frac{3}{(2n+1)(n-1)}(\nabla_{e_a}U)(e_a,\nabla f).
\end{equation}

 \textbf{Part 2}: By Part 1 it suffices to consider the case of a qc-Einstein structure, in which case we proceed as follows. First, we show that $(M,h)$ ($h$- Riemannian metric!) is isometric to the unit round  sphere by showing that $(\nabla^h)^2f(X,Y)=-fh(X,Y)$ and using Obata's result, see Remark \ref{r:sphere charcat} and the paragraph preceding it.
 Next, we show qc-conformal flatness. For this we use the form of the curvature of the round sphere, $R^h(A,B,C,D)=h(B,C)h(A,D)-h(B,D)h(A,C)$ and the relation between the Riemannian and Biquard curvatures and then the formula for $W^{qc}(X,Y,Z,V)$ which simplifies considerably in the qc-Einstein case. Finally, we employ a standard monodromy argument showing that $(M,g,\eta,\mathbb Q)$  is qc-conformal to ${S^{4n+3}}$, i.e., we have  $\eta=\kappa \Psi F^*\tilde\eta$ for some diffeomorphism $F:M\rightarrow S^{4n+3}$,  $0<\kappa\in \mathcal{C}^\infty(M)$, and $\Psi\in \mathcal{C}^\infty(M:SO(3))$ 
 We conclude the proof of the qc-conformality with the 3-Sasakian sphere by invoking  the qc-Liouville theorem \ref{t:qcLiouville}.

Finally, a comparison of the metrics on $H$ show the desired homothety.

We should mention that an alternative to the use of the qc-conformal curvature tensor in Step 2 was found in \cite{BauKim14} where once the isometry with the round sphere is established the authors invoke the classification of Riemannian submersions with totally geodesic fibers of the sphere
due to Escobales \cite{Esc76}.

Since the above proof used the qc-Liouville theorem and because of its independent interest we devote a short section to it.

\subsection{The QC Liouville theorem}\label{ss:qc liuouville}
\begin{thrm}\label{t:qcLiouville}
          Every qc-conformal transformation between open subsets of the 3-Sasakian unit sphere is the restriction of a global qc-conformal transformation.
          \end{thrm}
This result is proved in the more general setting of parabolic {geometries} in \cite{CS09}.
Here we give a relatively self-contained proof of a version of Liouville's theorem in the case of the quaternionic Heisenberg group and the 3-Sasakian sphere equipped with their standard qc structures. The proof is related to the QC Yamabe problem on the 3-Sasakian sphere since a key step is provided by the proof of \cite[Theorem 1.1]{IMV} in which all qc-Einstein structures qc-conformal to the standard qc-structure on the quaternionic Heisenberg group (or sphere) were determined.  Thus, our proof of  Theorem \ref{t:qcLiouville} establishes  the local  Liouville type property in the setting of a sufficiently smooth qc-conformal maps relying only on the qc geometry. A very general version of the Liouville theorem was also {proved} by  Cowling, M., \& Ottazzi, A., see \cite{CO13}.

In the Euclidean  case Liouville  \cite{Liu1850}, \cite{Liu1850b} showed that every sufficiently smooth conformal map ($\emph{C}^4$ in fact) between two connected open sets of the Euclidean space $\mathbb{R}^3$  is necessarily a restriction of a M\"obius transformation. The latter is the group generated by translations, dilations  and inversions of the extended Euclidean space obtained by adding an ideal point at infinity.  Liouville's result generalizes easily to any dimension $n>3$. Subsequently, Hartman \cite{Har58} gave a proof requiring only $\mathcal{C}^1$ smoothness of the conformal map, see also  \cite{Ne60},  \cite{BoIw82}, \cite{Ja91},    \cite{IwMa98} and  \cite{Fr03} for other proofs.   A CR version of Liouville's result can be found in \cite{Ta62} and \cite{Al74}. Thus, a smooth CR diffeomorphism between two connected open subsets of the $2n+1$ dimensional  sphere is  the restriction of an element from the isometry group $SU(n+1,1)$ of the unit ball equipped with the complex hyperbolic metric. The proof of Alexander \cite{Al74} relies on the extension property of a smooth CR map to a biholomorphism. Tanaka, see also \cite{Poi07}, \cite{Car} and \cite{ChM}, in his study of pseudo-conformal equivalence between analytic real hypersurfaces of complex space showed a more general result
 \cite[Theorem 6]{Ta62} showing that any pseudo-conformal  homeomorphism between connected open sets of the quadric
 \[
 -\sum_{i=1}^r |z_i|^2+\sum_{i=r+1}^n |z_i|^2 =1, \quad (z_1,\dots,z_n)\in \mathbb{C}^n,
 \]
is the restriction of a projective transformation of $P^{n}(\mathbb{C})$.

Another theory began with the introduction of quasiconformal maps  \cite{Ge62} and \cite{Re67}, which imposed metric conditions on the maps,   and with the works of Mostow \cite{M} and Pansu \cite{P}.  In particular, in \cite{P} it was shown  that every global 1-quasiconformal map on the sphere at infinity  of each of the hyperbolic metrics is an isometry of the corresponding hyperbolic space.  The local version of the Liouville's property for  1-quasiconformal map of class $\emph{C}^4$ on the Heisenberg group  was settled in \cite{KoRe85} by a  reduction to the CR  result. The optimal regularity question for quasiconformal maps was settled later by Capogna \cite{Cap97} in much greater generality including the cases of all Iwasawa type groups, see also \cite{Ta96} and \cite{CC06}.

 A closely related  property is the so called rigidity property of quasiconformal or {multicontact} maps, also referred to as Liouville's property, but where the question is the finite dimensionality of the group of (locally defined) quasiconformal or {multicontact} maps, see \cite{Ya93},   \cite{Reimann01},  \cite{CMKR05}, \cite{Ot05}, \cite{Ot08},\cite{Mor09}, \cite{dMO10}, \cite{lDOt11}.

Besides the Cayley transform, we shall need the generalization of the Euclidean inversion transformation to the qc setting.  We recall that in \cite{Ko1} Kor\'anyi introduced such an inversion and an analogue of the Kelvin transform on the Heisenberg group, which were later generalized in \cite{CK} and \cite{CDKR} to all groups of Heisenberg type. The inversion and Kelvin transform enjoy useful properties in the case of the four groups of Iwasawa type of which $\QH$ is a particular case.   For our goals it is necessary to show that the inversion on $\QH$ is a qc-conformal map. In order to prove this fact we shall represent the inversion as the composition of two Cayley transforms, see \cite{IMV1,IV2} where  the seven dimensional case was used.   Let $P_1=(-1,0)$ and $P_2=(1,0)$ be correspondingly the 'south' and 'north'
poles of the unit sphere  ${S^{4n+3}}\ =\ \{\abs{q}^2+\abs{p}^2=1 \}\subset \Hn\times\mathbb{H}$.  Let $\mathcal{C}_1$ and $\mathcal{C}_2$ be the corresponding Cayley transforms defined, respectively, on $S^{4n+3}\setminus\{P_1\}$ and $S^{4n+3}\setminus\{P_2\}$. Note that $\mathcal{C}_1$ was defined in \eqref{e:Cayley transf ctct form}%
, while $\mathcal{C}_2$ is given by $(q', p')\ =\ \mathcal{C}_2\
\Big ((q, p)\Big)$,
$
q'\ =\ -(1-p)^{-1} \ q,\quad p'\ =\ (1-p)^{-1} \ (1+p), \qquad  (q,p)\in S^{4n+3}\setminus\{P_2\}
$
The inversion on the quaternionic Heisenberg group (with center the origin) with respect to the unit gauge sphere   is the map
\begin{equation}\label{d:inversion}
\sigma=\mathcal{C}_2\circ\mathcal{C}_1^{-1}:\QH\setminus \{ (0,0)\}
\rightarrow\QH\setminus \{ (0,0)\}.
\end{equation}
  In particular, $\sigma\ =\ \mathcal{C}_2\circ\mathcal{C}_1^{-1} $ is {an involution}
on the group.   A small calculation shows that $\sigma$ is given by the formula (in the Siegel model)
$
q^*\ =\ -p'^{-1}\, q',\qquad p^*\ =\ p'^{-1},
$
or, equivalently, in the direct product model $\boldsymbol{G\,(\mathbb{H})}$
\begin{equation*}
q^*\ =\ -(|q'|^2-\omega')^{-1}\, q', \qquad \omega^*\ =\ -\frac {\omega'%
}{|q'|^4+|\omega'|^2}.
\end{equation*}
It follows
\begin{equation}
\begin{aligned}
\sigma ^*\ \Theta\ =\ \frac {1}{|p'|^2}
\,\bar\mu\, \Theta\, \mu, \qquad \mu\ =\ \frac {p'}{|p'|},\qquad \text{ (in the Siegel model)}\\
\sigma ^*\ \Theta\ =\ \frac {1}{|q'|^4+|\omega'|^2}
\,\bar\mu\, \Theta\, \mu, \qquad \mu\ =\ \frac {|q'|^2+\omega'}{\left ( |q'|^4+|\omega'|^2\right)^{1/2}},\qquad \text{ (in the product model)},
\end{aligned}
\end{equation}
which, shows the following fundamental fact.
\begin{lemma}\label{l:inversion qc}
The inversion transformation \eqref{d:inversion} is a qc-conformal transformation on the quaternionic Heisenberg group.
\end{lemma}
As usual, using the dilations and translations on the group, it is a simple matter to define an inversion   with respect to any gauge ball.

Turning to the proof of Theorem \ref{t:qcLiouville} let $\Sigma\not={S^{4n+3}}$, noting that  in the case $\Sigma={S^{4n+3}}$ there is nothing to prove.
We shall transfer the analysis to the quaternionic Heisenberg group using the Cayley transform, thereby reducing  to the case of a qc-conformal transformation $\tilde F:\tilde \Sigma \rightarrow \QH$ between two domains of the quaternionic Heisenberg group such that $\Theta={\tilde F}^*\tilde\Theta= \frac{1}{2\phi} \tilde\Theta$ for some positive smooth function $\phi$ defined on the open set $\tilde\Sigma$.  By its definition $\Theta$ is a qc-Einstein structure of vanishinq qc-scalar curvature, hence Theorem \ref{t:einstein preserving} shows  $\sigma=0$ and $F$ is a composition of a translation, cf. \eqref{e:H-type Iwasawa groups}, followed by an inversion and a homothety, cf. Lemma \ref{l:inversion qc}.

The above analysis implies that $F$ is the restriction of an element of $PSp(n+1,1)$. This completes the proof of Theorem \ref{t:qcLiouville}.
{
Similarly to the Riemannian and CR cases,  see \cite{Ku49}, \cite[Theorem  VI.1.6]{SchYa} and \cite{BS76},  Theorem \ref{t:qcLiouville} and a standard monodromy type argument show the validity of the next}
\begin{thrm}\label{t:conf sphere}
If $(M,\eta)$ is a simply connected qc-conformally flat manifold of dimension $4n+3$, then there is a qc-conformal immersion $\Phi:M\rightarrow {S^{4n+3}}$, where ${S^{4n+3}}$ is the 3-Sasakian unit sphere in the $(n+1)$-dimensional quaternion space.
\end{thrm}

\section{Heterotic string theory relations }\label{s:strings}

The seven dimensional  quaternionic Heisenberg group $\QH$
has  applications in the construction of
non-trivial solutions to the so called \emph{Strominger system} in supersymmetric heterotic string theory.

The bosonic fields of the ten-dimensional supergravity which
arises as low energy effective theory of the heterotic string are
the spacetime metric $g$, the NS three-form field strength (flux)
$H$, the dilaton $\phi$ and the gauge connection $A$ with
curvature 2-form $F^A$. The bosonic geometry is of the form
$\mathbb{R}^{1,9-d}\times M^d$, where the bosonic fields are
non-trivial only on $M^d$, $d\leq 8$. One considers the two
connections $
\nabla^{\pm}=\nabla^g \pm \frac12 H, $ 
where $\nabla^g$ is the Levi-Civita connection of the Riemannian
metric~$g$. Both connections preserve the metric,
$\nabla^{\pm}g=0$ and have totally skew-symmetric torsion $\pm H$,
respectively. We denote by $R^g,R^{\pm}$ the corresponding
curvature.

The Green-Schwarz anomaly cancellation condition up to the first
order of the string constant
$\alpha^{\prime }$ reads
\begin{equation}  \label{acgen}
dH=\frac{\alpha^{\prime }}48\pi^2(p_1(\nabla^-)-p_1(E))=\frac{\alpha^{\prime }}4 %
\Big(Tr(R^-\wedge R^-)-Tr(F^A\wedge F^A)\Big),
\end{equation}
where $p_1(\nabla^-)$ and $p_1(E)$ are the first Pontrjagin forms
with respect to a connection $\nabla^-$ with curvature $R$ and a
vector bundle $E$ with connection $A$.

A heterotic geometry preserves supersymmetry iff in ten dimensions
there exists at least one Majorana-Weyl spinor $\epsilon$ such
that the following Killing-spinor equations hold \cite{Str,Berg}
\begin{equation} \label{sup1}
\nabla^+\epsilon=0, \quad
(d\phi-\frac12H)\cdot\epsilon=0, \quad
F^A\cdot\epsilon=0,
\end{equation}
where
$\cdot$ means Clifford action
of forms on spinors.

The system of Killing spinor equations \eqref{sup1} together with
the
anomaly cancellation condition \eqref{acgen} is known as the \emph{%
Strominger system} \cite{Str}. The last equation in \eqref{sup1}
is the instanton condition which means that the curvature $F^A$ is
contained in a Lie algebra of a Lie group which is a stabilizer of
a non-trivial spinor. In dimension 7 the largest such a  group is the exceptional
group $G_2$ {which is the {automorphism } group of the unit imaginary octonions. Denoting by $\Theta$ the non-degenerate  three-form definning the $G_2$ structure,} the
$G_2$-instanton condition has the form
\begin{equation}  \label{in2}
\sum_{k,l=1}^7(F^A)^i_j(e_k,e_l)\Theta(e_k,e_l,e_m)=0.
\end{equation}
Geometrically, the existence of a non-trivial real spinor parallel
with respect to the metric connection $\nabla^+$ with totally
skew-symmetric torsion $T=H$ leads to restriction of the holonomy
group $Hol(\nabla^+)$ of the torsion connection $\nabla^+$. In
dimension seven $Hol(\nabla^+)$ has to be contained in the
exceptional group $G_2$ \cite{FI1,GKMW,GMW,FI2}.

The general existence result \cite{GKMW,FI1,FI2} states that there
exists a non-trivial solution to both dilatino and gravitino
Killing spinor equations (the first two equations in \eqref{sup1})
in dimension d=7 if and only if there exists a globally conformal co-calibrated $G_2$-structure $%
(\Theta,g)$  of pure type and the Lee form $\theta^7=-\frac{1}{3}*(* d\Theta\wedge\Theta) = \frac{1}{3}%
*(* d*\Theta\wedge*\Theta)$ has to be exact, i.e. a $G_2$%
-structure $(\Theta,g)$ satisfying the equations
\begin{equation}  \label{sol7}
d*\Theta=\theta^7\wedge *\Theta, \quad d\Theta\wedge\Theta=0,
\quad \theta^7=-2d\phi.
\end{equation}
Therefore, the torsion 3-form (the flux $H$) is given by
\begin{equation}\label{torstr}
H=T= -* d\Theta - 2*(d\phi\wedge\Theta).
\end{equation}

A geometric model which fits the above structures  was proposed in \cite%
{FIUVdim7-8} as a certain ${\mathbb{T}}^3$-bundle over a
Calabi-Yau surface. For this, let $\Gamma_i$, $1\leq i \leq 3$, be
three closed anti-self-dual $2$-forms on a Calabi-Yau surface
$M^4$, which represent integral cohomology classes. Denote by
$\omega_1$ and by $\omega_2+\sqrt{-1}\omega_3$ the (closed)
K\"ahler form and the holomorphic volume form on $M^4$,
respectively. Then, there is a compact 7-dimensional manifold
$M^{1,1,1}$ which is the total space of a ${\mathbb{T}}^3$-bundle
over $M^4$ and has a $G_2$-structure $
\Theta=\omega_1\wedge\eta_1+\omega_2\wedge\eta_2-\omega_3\wedge\eta_3+\eta_1%
\wedge \eta_2\wedge\eta_3,
$
solving the first two Killing spinor equations in \eqref{sup1}
with constant dilaton in dimension $7$, where $\eta_i$, $1\leq i
\leq 3$, is a $1$-form on $M^{1,1,1}$ such that
$d\eta_i=\Gamma_i$, $1\leq i \leq 3$.

For any smooth function $f$ on $M^4$, the $G_2$-structure on
$M^{1,1,1}$ given by
\begin{equation*}
\Theta_f=e^{2f}\Big[\omega_1\wedge\eta_1+\omega_2\wedge\eta_2-
\omega_3\wedge\eta_3\Big]+\eta_1\wedge\eta_2\wedge\eta_3
\end{equation*}
solves the first two Killing spinor equations in \eqref{sup1} with
non-constant dilaton $\phi=-2f$.

To achieve a smooth solution to the Strominger system \emph{we still
have to determine} an auxiliary vector bundle with an
$G_2$-instanton
in order to satisfy the anomaly cancellation condition
\eqref{acgen}.
\subsection{The quaternionic Heisenberg group}

The Lie algebra $\mathfrak{g(\mathbb{H})}$ of the seven dimensional group $\QH$ has structure equations
\begin{equation}  \label{ecus-qHg}
\begin{aligned} & d\gamma^1=d\gamma^2=d\gamma^3=d\gamma^4=0, \quad
d\gamma^5=\gamma^{12}-\gamma^{34},\quad
d\gamma^6=\gamma^{13}+\gamma^{24},\quad
d\gamma^7=\gamma^{14}-\gamma^{23}.
\end{aligned}
\end{equation}
where $\gamma_1,\dots,\gamma_7$ is a basis of  left invariant
1-forms on $\QH$. In particular, the quaternionic Heisenebrg group
$\QH$ in dimension seven is an $R^3$-bundle over the flat
Calabi-Yau space $R^4$ and therefore fits  the geometric model
described above.

In order to obtain results in dimensions less than seven through
contractions of $\QH$ it will be convenient
to consider the orbit of $G(\mathbb{H})$ under the natural action
of $GL(3,\mathbb{R})$ on the $span\, \{\gamma^5, \gamma^6,
\gamma^7\}$. Accordingly let $K_A$ be a seven-dimensional real Lie
group with Lie bracket $[x,x^{\prime
}]_A=A[A^{-1}x,A^{-1}x^{\prime }]$ for $A\in GL(3,\mathbb{R})$
defined by a basis of left-invariant 1-forms $\{e^1,\ldots,e^7\}$
such that $e^i=\gamma^i$ for $1 \leq i \leq 4$ and $(e^5\ e^6\
e^7)=A\, (\gamma^5\ \gamma^6\ \gamma^7)^T$. Hence, the structure
equations of the Lie algebra $\mathfrak{K}_A$ of the group $K_A$
are
\begin{equation}  \label{ecus-general}
d e^1=d e^2=d e^3=d e^4=0, \qquad d e^{4+i}=
\sum_{j=1}^3a_{ij}\,\sigma_j, \quad i=1,2,3,
\end{equation}
where $\sigma_1=e^{12}-e^{34}$, $\sigma_2=e^{13}+e^{24}$, $%
\sigma_3=e^{14}-e^{23}$ are the three anti-self-dual 2-forms on
$\mathbb{R}^4$ and $A=\{a_{ij}\}$ is a 3 by 3 matrix.
We will denote the norm of $A$ by $|A|$, $|A|^2=\sum_{i,j=1}^3
a_{ij}^2$.

Since $\mathfrak{K}_A$ is isomorphic to
$\mathfrak{g(\mathbb{H})}$, if $K_A$ is connected and simply
connected it is isomorphic to $G(\mathbb{H})$.
Furthermore, any lattice $\Gamma_A$ gives rise to a (compact) nilmanifold $%
M_A=K_A/\Gamma_A$, which is a $\mathbb{T}^3$-bundle over a
$\mathbb{T}^4$ with connection 1-forms of anti-self-dual curvature
on the four torus.

The three closed hyperK\"ahler 2-forms on  $\mathbb R^4$ are given by $\omega_1=e^{12}+e^{34},\quad \omega_2=e^{13}-e^{24},\quad
\omega_3=e^{14}+e^{23}.$
Following \cite{FIUVdim7-8}, for a smooth function $f$ on $\mathbb R^4$,
we consider the $G_2$ structure on  $K_A$ defined by the 3-form
\begin{equation}  \label{g2-general}
\bar{\Theta}=e^{2f}\Big[\omega_1\wedge e^7+\omega_2\wedge
e^5-\omega_3\wedge e^6\Big] + e^{567}{.}
\end{equation}
The corresponding metric $\bar{g}$ on $K_{A}$ has an orthonormal
basis of 1-forms given by
\begin{equation}  \label{conf-general}
\bar{e}^{1}=e^{f}\,e^{1},\quad \bar{e}^{2}=e^{f}\,e^{2},\quad \bar{e}%
^{3}=e^{f}\,e^{3},\quad \bar{e}^{4}=e^{f}\,e^{4},\quad \bar{e}%
^{5}=e^{5},\quad \bar{e}^{6}=e^{6},\quad \bar{e}^{7}=e^{7}
\end{equation}
with self-dual 2-forms $\bar\omega_i=e^{2f}\omega_i$ and anti-self-dual 2-forms
$\bar\sigma_i=e^{2f}\sigma_i$, i=1,2,3.

It is easy to check using \eqref{ecus-general} and the property $%
\sigma_i\wedge\omega_j=0$ for $1\leq i,j \leq 3$ that the
\eqref{sol7} is satisfied, i.e.,
the $G_2$ structure $\bar\Theta$ solves the gravitino and dilatino
equations with non-constant dilaton $\phi=-2f$ \cite{FIUVas2}.
Furthermore, with $f_{ij}=\frac{%
\partial ^{2}f}{\partial x_{j}\partial x_{i}}$, $1\leq i,j\leq 4$, we obtain the next formula for the torsion $\bar T$ of $\bar\Theta$, see \cite{FIUVas2} for details,
\begin{equation}  \label{torsion-general}
d\bar{T} =-e^{-4f}\left[ \triangle e^{2f}+2|A|^2\right] \,\bar{e}^{1234} =-%
\left[ \triangle e^{2f}+2|A|^2 \right] e^{1234},
\end{equation}
where $\triangle
e^{2f}=(e^{2f})_{11}+(e^{2f})_{22}+(e^{2f})_{33}+(e^{2f})_{44}$
 is the Laplacian on $\mathbb{R}^4$.
\subsection{The first Pontrjagin form of the $(-)$-connection}

\label{pon7-general} The connection 1-forms of a connection
$\nabla$ are determined by
$\nabla_Xe_j=\sum_{s=1}^7\omega^s_j(X)e_s$.
From  Koszul's formula, we have that the Levi-Civita connection 1-forms $%
(\omega ^{\bar{g}})_{\bar{j}}^{\bar{\imath}}$ of the metric
$\bar{g}$ are given by
\begin{equation}\label{lc-general}
\begin{array}{ll}
(\omega ^{\bar{g}})_{\bar{j}}^{\bar{\imath}}(\bar{e}_{k}) \!&\! =-\frac{1}{2}\Big(%
\bar{g}(\bar{e}_{i},[\bar{e}_{j},\bar{e}_{k}])-\bar{g}(\bar{e}_{k},[\bar{e}%
_{i},\bar{e}_{j}])+\bar{g}(\bar{e}_{j},[\bar{e}_{k},\bar{e}_{i}])\Big) \\[8pt]
\!&\! =\frac{1}{2}\Big(d\bar{e}^{i}(\bar{e}_{j},\bar{e}_{k})-d\bar{e}^{k}(%
\bar{e}_{i},\bar{e}_{j})+d\bar{e}^{j}(\bar{e}_{k},\bar{e}_{i})\Big)
\end{array}
\end{equation}
taking into account $\bar{g}(\bar{e}_{i},[\bar{e}_{j},\bar{e}_{k}])=-d\bar{e}%
^{i}(\bar{e}_{j},\bar{e}_{k})$. With the help of
\eqref{lc-general} we
compute the expressions for the connection 1-forms $(\omega ^{-})_{\bar{j}}^{%
\bar{\imath}}$ of the connection $\nabla ^{-}$,
\begin{equation}
(\omega ^{-})_{\bar{j}}^{\bar{\imath}}=(\omega ^{\bar{g}})_{\bar{j}}^{\bar{%
\imath}}-\frac{1}{2}(\bar{T})_{\bar{j}}^{\bar{\imath}},\qquad \text{ where }%
\qquad (\bar{T})_{\bar{j}}^{\bar{\imath}}(\bar{e}_{k})=\bar{T}(\bar{e}_{i},%
\bar{e}_{j},\bar{e}_{k}).  \label{minus-general}
\end{equation}
A long straightforward calculation based on \eqref{lc-general},
\eqref{minus-general}  yields that
the first Pontrjagin form of $\nabla^{-}$ is a scalar multiple of
$e^{1234}$ given by \cite{FIUVas2}
\begin{equation}  \label{p1-general}
\pi ^{2}p_{1}(\nabla ^{-}) =\left[ \mathcal{F}_2[f]+\triangle_4 f -\frac{3}{8%
} |A|^{2} \triangle e^{-2f} \right] {e}^{1234},
\end{equation}
where $\mathcal{F}_2[f]$ is the 2-Hessian of $f$, i.e., the sum of
all principle $2\times 2$-minors of the Hessian, and $\triangle_4
f=div(|\nabla f|^2\nabla f)$ is the 4-Laplacian of $f$.
This formula shows, in particular, that even though the
curvature 2-forms of $\nabla^-$ are quadratic in the gradient of
the dilaton, the Pontrjagin form of $\nabla^-$ is also quadratic
in these terms. Furthermore, if $f$ depends on two of the
variables then $\mathcal{F}_2[f]=det (Hess f)$ while if $f$ is a
function of one variable $\mathcal{F}_2[f]$ vanishes.

What remains is to solve the anomaly cancellation condition.
We use the $G_2$-instanton $\mathrm{D}_{\Lambda }$ defined in \cite{FIUVas2}, which depends on a 3 by 3 matrix $\Lambda=(%
\lambda_{ij}) \in {\mathfrak{g}\mathfrak{l}}_3(\mathbb{R})$.

It is shown in \cite{FIUVas2} that the connection  $\mathrm{D}_{\Lambda}$ is a $G_2$-instanton with respect to
the $G_2$ structure defined by \eqref{g2-general} which preserves
the metric if and only if $\mathrm{rank}(\Lambda) \leq 1$. In this
case, the first
Pontrjagin form $p_{1}(\mathrm{D}_{\Lambda})$ of the $G_2$-instanton $%
\mathrm{D}_{\Lambda}$ is given by
\begin{equation}  \label{abinst-general}
8\pi ^{2}p_{1}(\mathrm{D}_{\Lambda})= -4\lambda^2\,e^{1234},
\end{equation}
where $\lambda=|\Lambda\, A|$ is the norm of the product matrix
$\Lambda\, A$.

After this preparation, we are left with solving the anomaly cancellation condition
$d\bar{T}=\frac{\alpha ^{\prime }}{4}8\pi ^{2}\Big(p_{1}(\nabla
^{-})-p_{1}(D_\Lambda)\Big)$, which in general is a highly overdetermined system for the dilaton function $f$. Remarkably, in our case taking
into account \eqref{torsion-general}, \eqref{p1-general} and %
\eqref{abinst-general} the anomaly becomes \emph{the single}  non-linear
equation
\begin{equation}  \label{e:anomaly negative alpha}
\triangle e^{2f}+2|A|^2 +\frac{\alpha ^{\prime }}{4}\left[ 8\mathcal{F}%
_2[f]+8\triangle_4 f -3 |A|^{2} \triangle e^{-2f}
+4\lambda^2\right]=0.
\end{equation}
We remind that this is an equation on $\mathbb{R}^4$ for the dilaton function $f$.

\begin{rmrk} An important question interesting for both string theory and
nonlinear analysis  is whether the non-linear PDE \eqref{e:anomaly
negative alpha} admits a periodic solution.
\end{rmrk}

In \cite{FIUVas2} was found a one dimensional (non-smooth) solution, which we describe briefly. If we assume that the function $f$ depends on one variable,
$f=f(x^{1})$, and for a \emph{negative} $\alpha ^{\prime }$ we
choose $2|A|^2+\alpha ^{\prime }\lambda ^{2}=0$, i.e., we let
$\alpha ^{\prime }=-\alpha^2$ so that $2|A|^2=\alpha ^{2}\lambda
^{2}$. This simplifies \eqref{e:anomaly negative alpha} to the
ordinary differential equation
\begin{equation}  \label{solv4}
\left( e^{2f}\right) ^{\prime }+\frac34\alpha ^{2}|A|^2\left(
e^{-2f}\right) ^{\prime }-2\alpha ^{2 }f^{\prime 3}=C_0=const.
\end{equation}%
A solution of the last equation for $C_0=0$ was found in \cite[Section 4.2]%
{FIUVas}. The substitution $u=\alpha^{-2} e^{2f}$ allows us to
write \eqref{solv4} in the form
\begin{equation*}
\left( e^{2f}\right) ^{\prime }+\frac34\alpha ^{2}|A|^2\left(
e^{-2f}\right)
^{\prime }-2\alpha ^{2 }f^{\prime 3}=\frac{\alpha^2 u^{\prime }}{4u^{3}}%
\left( 4u^{3}-3\frac {|A|^2}{\alpha ^{2 }}u-u^{\prime 2}\right) .
\end{equation*}%
For $C_0=0$ we solve the following ordinary differential equation
for
the function $u=u(x^1)>0$ 
\begin{equation}  \label{solv5}
u^{\prime 2}={4}u^{3}-3\frac {|A|^2}{\alpha ^{2 }}u=4u\left(
u-d\right) \left( u+d\right) ,\qquad d=\sqrt{3|A|^2}/ \alpha.
\end{equation}%
Replacing the real derivative with the complex derivative leads to
the Weierstrass' equation
$
\left (\frac {d\, \mathcal{P}}{dz}\right )^2=4\mathcal{P}\left( \mathcal{P}%
-d\right) \left( \mathcal{P}+d\right)
$
for the doubly periodic Weierstrass $\mathcal{P}$ function with a
pole at the origin.
Letting $\tau_{\pm}$ be the basic half-periods such
that $\tau _{+}$ is real and $\tau _{-}$ is purely imaginary we have that $%
\mathcal{P}$ is real valued on the lines
$\mathfrak{R}\mathfrak{e}\,z=m\tau _{+}$ or
$\mathfrak{I}\mathfrak{m}\,z=im\tau _{-}$, $m\in \mathbb{Z}$.
Thus, $u(x^1)=\mathcal{P}(x^1)$ defines a non-negative
$2\tau_{+}$-periodic function with singularities at the points $2n
\tau_{+}$, $n\in \mathbb{Z}$, which solves the real equation
\eqref{solv5}. By construction, $f= \frac 12 \ln (\alpha^2 u)$ is
a periodic function with
singularities on the real line which is a solution to equation %
\eqref{e:anomaly negative alpha}. {\ 
Therefore the $G_2$ structure defined by $\bar \Theta$ descends to the $7$%
-dimensional nilmanifold $M^{7}=\Gamma \backslash K_{A}$ with
singularity, determined by the singularity of $u$, where $K_{A}$
is the 2-step nilpotent
Lie group with Lie algebra $\mathfrak{K}_{A}$, defined by %
\eqref{ecus-general}, and $\Gamma $ is a lattice with the same period as $f$%
, i.e., $2 \tau_{+}$ in all variables.

In fact, $M^7$ is the total space of a $\mathbb{T}^3$ bundle over
the asymptotically hyperbolic manifold $M^4$  which is a
conformally compact 4-torus with conformal boundary at infinity a
flat 3-torus. Thus, we obtain the complete solution to the
Strominger system in dimension seven with non-constant dilaton,
non-trivial instanton and flux and with a negative $\alpha
^{\prime }$ parameter found in \cite{FIUVas2}.

\subsection{Solutions through contractions}
A contraction of the quaternionic  heisenberg algebra can
be obtained considering the matrix
\begin{equation*}
A_\varepsilon\overset{def}{=}\left(\!\!\!
\begin{array}{ccc}
0 & b & 0 \\
a & 0 & -b \\
0 & 0 & \varepsilon%
\end{array}
\!\right).
\end{equation*}
Letting $\varepsilon\rightarrow 0$ into $A_{\varepsilon}$ we get in the limit, using \eqref{ecus-general},  the structure equations of a six dimensional  two step nilpotent Lie algebra known as $\mathfrak{h}_5$. On the corresponding simply connected two-step nilponent Lie group $H_5$ non-trivial solutions to the Strominger system in
dimension 6 were prsented in\cite{FIUVas}. It is a remarkable fact \cite{FIUVas2}
that the geometric structures, the partial differential equations
and their solutions found in dimension seven starting with  the
quaternionic Heisenbeg group as above converge trough contraction to the heterotic solutions
on 6-dimensional  non-K\"ahler space   on $H_5$ found in \cite{FIUVas}. Moreover, using suiatable contractraction it is possible to obtain non-trivial solutions to the Strominger system in dimension 5 as well (see \cite{FIUVas2} for details).

\textbf{Acknowledgments}:S.Ivanov is visiting University of
Pennsylvania, Philadelphia. S.I. thanks UPenn for providing the
support and an excellent research environment during the final
stage of the paper. S.I. is partially supported by Contract DFNI
I02/4/12.12.2014 and  Contract 168/2014 with the Sofia University
"St.Kl.Ohridski". D.V. was partially supported by Simons
Foundation grant \#279381. D.V. would like to thank the organizers
of the Workshop on Geometric Analysis on sub-Riemannian manifolds
at IHP for the stimulating atmosphere and talks during the
workshop which influenced and inspired the writing of this paper.

\end{document}